\documentclass[a4paper,11pt]{article}

\usepackage{subfiles}

\usepackage[english]{babel}
\usepackage[utf8]{inputenc}
\usepackage[T1]{fontenc}
\usepackage{lmodern}
\usepackage{amsmath}
\usepackage{amssymb}

\usepackage{amsthm}
\usepackage{booktabs}
\usepackage{subcaption}
\usepackage{xcolor}
\usepackage{graphicx}
\usepackage{bm}
\usepackage{enumitem}

\usepackage[top=25mm, bottom=25mm, left=25mm, right=25mm]{geometry}

\usepackage{setspace}

\usepackage[round]{natbib}

\usepackage{tikz}
\usepackage{pgfplots}
\usepgfplotslibrary{fillbetween}
\usetikzlibrary{decorations.pathreplacing}
\pgfplotsset{compat=1.14}
\usepackage[ruled]{algorithm2e}
\usepackage{sectsty}
\usepackage{multirow}
\usepackage{titling}
\usepackage{adjustbox}

\usepackage{letltxmacro}
\usepackage{thmtools}

\usepackage{nameref,hyperref}
\usepackage{cleveref}

\declaretheorem{theorem}

\newlist{lemlist}{enumerate}{1}
\setlist[lemlist]{label=(\roman{lemlisti}), ref=\thelemma(\roman{lemlisti}),noitemsep}
\Crefname{theorem}{Theorem}{Theorems}

\declaretheorem[name=Lemma,
Refname={Lemma,Lemmas}]{lemma}
\Crefname{lemlisti}{Lemma}{Lemmas}
\addtotheorempostheadhook[lemma]{\crefalias{lemlisti}{lemma}}

\newlist{corlist}{enumerate}{1}
\setlist[corlist]{label=(\roman{corlisti}), ref=\thecorollary(\roman{corlisti}),noitemsep}
\declaretheorem[name=Corollary,
Refname={Corollary,Corollaries}]{corollary}
\Crefname{corlisti}{Corollary}{Corollaries}
\addtotheorempostheadhook[corollary]{\crefalias{corlisti}{corollary}}

\theoremstyle{plain}

\newtheorem*{claim*}{Claim}

\theoremstyle{definition}
\newlist{deflist}{enumerate}{1}
\setlist[deflist]{label=(\roman{deflisti}), ref=\thedefinition(\roman{deflisti}),noitemsep}
\declaretheorem[
name=Definition,
Refname={Definition,Definitions}]{definition}
\Crefname{deflisti}{Definition}{Definitions}
\addtotheorempostheadhook[definition]{\crefalias{deflisti}{definition}}

\newtheorem{condition}{Condition}
\Crefname{condition}{Condition}{Conditions}
\newlist{conlist}{enumerate}{1}
\setlist[conlist]{label=(\roman{conlisti}), ref=\thecondition(\roman{conlisti}),noitemsep}
\Crefname{conlisti}{Condition}{Conditions}
\addtotheorempostheadhook[condition]{\crefalias{conlisti}{condition}}

\declaretheorem[qed=$\blacktriangle$]{example}

\usepackage[ruled]{algorithm2e}

\newcommand{\minus}{\scalebox{0.75}[1.0]{$-$}}

\newlist{CC}{enumerate}{1}
\setlist[CC]{label=Case \arabic*)}

\newlist{STP}{enumerate}{1}
\setlist[STP]{label=Step \arabic*)}

\newcommand \linedabstractkw[2]{
  \renewcommand\maketitlehookd{%
    \mbox{}\medskip\par
    \centering
    \hrule\medskip
    \begin{minipage}{0.9\textwidth}
    #1\\

    \textit{Keywords: }#2
    \end{minipage}\medskip\hrule\medskip
    }
}

\graphicspath{{img/}{../img/}}

\include{commands}

\title{ \vspace*{-1cm}
Pareto Adaptive Robust Optimality via \\a Fourier-Motzkin Elimination Lens
}
\author{D. Bertsimas\thanks{Operations Research Center and Sloan School of Management, Massachusetts Institute of Technology, USA} \and S.C.M. ten Eikelder\thanks{Department of Econometrics and Operations Research, Tilburg University, The Netherlands} \and D. den Hertog\thanks{Department of Operations Management, University of Amsterdam, The Netherlands} \and N. Trichakis\protect\footnotemark[1]}

\date{April 30, 2022}

\newcommand{\dobib}{ 
    \bibliographystyle{apalike}
    \bibliography{References} 
}
\begin{document}
\renewcommand{\dobib}{}

\linedabstractkw{We formalize the concept of Pareto Adaptive Robust Optimality (PARO) for linear Adaptive Robust Optimization (ARO) problems. A worst-case optimal solution pair of here-and-now decisions and wait-and-see decisions is PARO if it cannot be Pareto dominated by another solution, i.e., there does not exist another such pair that performs at least as good in all scenarios in the uncertainty set and strictly better in at least one scenario. We argue that, unlike PARO, extant solution approaches---including those that adopt Pareto Robust Optimality from static robust optimization---could fail in ARO and yield solutions that can be Pareto dominated. The latter could lead to inefficiencies and suboptimal performance in practice. We prove the existence of PARO solutions, and present particular approaches for finding and approximating such solutions. We present numerical results for a facility location problem that demonstrate the practical value of PARO solutions. 

\hspace*{4ex} Our analysis of PARO relies on an application of Fourier-Motzkin Elimination as a proof technique. We demonstrate how this technique can be valuable in the analysis of ARO problems, besides PARO. In particular, we employ it to devise more concise and more insightful proofs of known results on (worst-case) optimality of decision rule structures.
} {Robust optimization; adaptive robust optimization; Pareto optimality; Fourier-Motzkin Elimination; decision rules}
\maketitle

\section{Introduction}\label{sec: introduction}
Robust Optimization (RO) is a widespread methodology for modeling decision-making problems under uncertainty that seeks to optimize worst-case performance~\citep{Bertsimas11, Gabrel14, Gorissen15}. In practice, RO problems usually admit multiple worst-case optimal solutions, the performance of which may differ substantially under non-worst-case uncertainty scenarios. Consequently, the choice of an optimal solution often has material impact on performance under real-world implementations. This important consideration, which was first brought forth by~\citet{Iancu14}, has been successfully tackled for static, single-stage (linear) RO problems. For the increasingly popular and broad class of dynamic, multi-stage Adaptive Robust Optimization (ARO) problems~\citep{BenTal04}, however, there is no successful approach for choosing an optimal solution, and the purpose of this paper is to bridge this gap.

In particular, for static RO problems, \citet{Iancu14} proposed the choice of so-called Pareto Robustly Optimal (PRO) solutions. In general, PRO solutions unarguably dominate non-PRO solutions, because, by definition, the former guarantee that there do not exist other worst-case optimal solutions that perform at least as good as the current solution for all scenarios in the uncertainty set, while performing strictly better for at least one scenario. The PRO concept has been applied to several static robust optimization settings, such as distributionally robust mechanism design \citep{Kocyigit19}.

Going beyond static RO problems, it is well understood that the choice of an optimal solution remains crucial for the broader class of multi-stage ARO problems. Similar to RO solutions, by following a worst-case philosophy and not considering performance across the entire spectrum of possible scenarios, ARO optimal solutions could lead to substantial performance losses. For example, see the work by \Citet{DeRuiter16}, who numerically demonstrate existence of multiple worst-case optimal solutions for the classical multi-stage inventory-production model that was considered in \citet{BenTal04}, and find them to differ considerably from each other in their non-worst-case performance.

For ARO problems, however, a solution approach that unarguably chooses ``good solutions,'' similar to PRO solutions for static RO problems, has proved to be elusive thus far. Extant approaches have all attempted to simply apply the concept of PRO to ARO problems. Specifically, they advocate restricting attention to adaptive variables that depend affinely on the uncertain parameters; commonly referred to as Linear Decision Rules (LDRs). Restricting to LDRs reduces the problem to static RO, and enables the search for associated PRO solutions \citep{Bertsimas19,DeRuiter16,Iancu14}. As we shall show, however, this indirect application of the PRO concept fails to produce solutions that cannot be dominated.

Recently, the concept of Pareto Adaptive Robustly Optimal (PARO) solutions was introduced for a specific ARO application \citep{tenEikelder21}, but no general treatment of the topic was included. In this paper, we formalize and study the concept of PARO solutions for linear ARO problems. Similar to PRO solutions for static RO problems, PARO solutions yield worst-case optimal performance and are not dominated by any other such solutions in non-worst-case scenarios. In other words, PARO solutions unarguably dominate non-PARO solutions, leading to improved performance in non-worst-case scenarios, while maintaining worst-case optimality. From a practical standpoint, this means that implementing PARO solutions can only yield performance benefits, without any associated drawbacks. 

To introduce the PARO concept and highlight its practical importance, we provide an illustrative toy example. The example also serves two additional important purposes. First, it enables us to show in a simple setting how PARO solutions can dominate PRO solutions, as remarked above. Second, the example motivates the need for new analysis techniques for studying PARO. 

\begin{example} \label{ex: toy-problem-RT}
\small
In treatment planning for radiation therapy, the goal is to deliver a curative amount of dose to the target volume (tumor tissue), while minimizing the dose to healthy tissues. Consider a simplified case with two target subvolumes. For subvolume $i\in\{1,2\}$, the required dose level $d_i$ depends on the radiation sensitivity of the tissue, which is unknown. Assume that, prior to treatment, the doses lie in 
\begin{align*}
U = \{(d_1,d_2)~|~ 50 \leq d_i \leq 60,~i=1,2\}.
\end{align*} 
Mid-treatment, the required doses are ascertained via biomarker measurements. 

Treatment doses are administered in two stages. The dose administered in the first stage, denoted by $x$, needs to be decided prior to treatment. The dose administered in the second stage, denoted by $y$, can be decided after the required doses have been ascertained, i.e., it can be adapted to uncertainty revelation. Both treatment doses are delivered homogeneously over both volumes in each stage. Dose in each stage is limited to the interval $[20,40]$. The total dose administered is $x+y$, and the healthy tissue receives a fraction $\delta >0$ of it. The Stage-1 dose $x$, and a decision rule $y(\cdot)$ for the adaptive Stage-2 dose can be chosen by solving:
\begin{subequations} \label{eq: RT-model}
\begin{align}
\min_{x,y(\cdot)} ~&~ \max_{(d_1,d_2) \in U} \delta(x  + y(d_1,d_2)) \\
\text{s.t.}~&~ x + y(d_1,d_2) \geq d_1,~~\forall  (d_1,d_2) \in U, \\
		   ~&~ x + y(d_1,d_2) \geq d_2,~~\forall  (d_1,d_2) \in U, \\
      	   ~&~ 20 \leq x \leq 40, \\
      	   ~&~ 20 \leq y(d_1,d_2) \leq 40,~~\forall  (d_1,d_2) \in U.
\end{align}
\end{subequations}
Problem~\eqref{eq: RT-model} is an ARO problem with constraintwise uncertainty, for which static decision rules are worst-case optimal \citep{BenTal04}. Plugging in $y(d_1,d_2) = \hat{y}$ and solving the resulting static RO model yields a worst-case optimal objective value of $60\delta$, achieved by all $(x,\hat{y})$ such that $x+\hat{y}=60$. For any such solution, the objective value remains $60\delta$ in not only the worst-case scenario but in all scenarios. Hence, all these solutions are PRO, according to the definition of \citet{Iancu14}. Consequently, the Stage-1 decisions that are PRO lie in the set:
\begin{align*}
X^{\text{PRO}} =  \{x ~|~ 20 \leq x \leq 40 \}.
\end{align*}

Consider now the decision rule $y^{\ast}(d_1,d_2) = \max\{20,d_1-x,d_2-x\}$, which is feasible for all feasible $x$. Furthermore, this rule minimizes the objective for any fixed $x$, $d_1$ and $d_2$. Plugging this in gives
\begin{align*}
\min_{20 \leq x \leq 40} \max_{(d_1,d_2) \in U} \delta \max\{20 + x,d_1,d_2\}. 
\end{align*}
For given $(d_1,d_2)$ the objective value is at least $\delta \max\{d_1,d_2\}$, and this is achieved by all $x \leq 30$. Thus, it should be preferable to implement one of these solutions for the Stage-1 decision. In fact, these solutions, which we call PARO, cannot be dominated by other solutions. Notably, the set of PARO solutions
\begin{align*}
X^{\text{PARO}} = \{x~|~ 20 \leq x \leq 30\},
\end{align*} 
is a strict subset of $X^{\text{PRO}}$. This implies that PARO solutions could dominate PRO solutions that are non-PARO. To exemplify, compare the following three solutions: (i) PARO solution $x^{\ast} = 25$ with optimal decision rule $y^{\ast}(d_1,d_2)$, (ii) PRO (non-PARO) solution $\hat{x} = 35$ with optimal decision rule $y^{\ast}(d_1,d_2)$, (iii) PRO solution $\hat{x} = 35$ with static decision rule $\hat{y}=25$.
\begin{table}[htb!]
\caption{\small Difference PARO and PRO solutions for \Cref{ex: toy-problem-RT}. \label{table: ex}}
\centering
\begin{tabular}{c c c c}
\toprule
Scenario $(d_1,d_2)$ & $(x^{\ast},y^{\ast})$ & $(\hat{x},y^{\ast})$ & $(\hat{x},\hat{y})$ \\ \midrule
$(60,60)$ & $60\delta$ & $60\delta$ & $60\delta$ \\
$(50,55)$ & $55\delta$ & $55\delta$ & $60\delta$ \\
$(50,50)$ & $50\delta$ & $55\delta$ & $60\delta$ \\
\bottomrule
\end{tabular}
\end{table}
\Cref{table: ex} shows the performance for three scenarios. For worst-case scenario $(60,60)$ all solutions perform equal. For scenario $(50,55)$ the solution (iii) is outperformed by the other two solutions, for scenario $(50,50)$ both solutions (ii) and (iii)) are outperformed by PARO solution (i). There is no scenario where $\hat{x}$ results in a strictly better objective value than $x^{\ast}$, irrespective of the used decision rule. Thus, the PRO solution $\hat{x}$ is dominated by the PARO solution $x^{\ast}$.
\hfill $\blacktriangle$ \end{example}

Besides showing that PRO solutions could be dominated in ARO problems, \Cref{ex: toy-problem-RT} also provides intuition into how. In particular, what unlocks extra performance in ARO problems is the application of decision rules that are not merely worst-case optimal, but rather ``Pareto optimal,'' i.e., they optimize performance over non-worst-case scenarios as well. Note, however, that although for worst-case optimality linear decision rules might suffice under special circumstances, for Pareto optimality non-linear rules appear to be more often necessary, as illustrated by the example. 

The application of non-linear decision rules to study PARO solutions invalidates the techniques used in the analysis of Pareto efficiency in RO in the extant literature, which is solely focused on linear formulations. In other words, analysis of Pareto efficiency in ARO calls for a new line of attack, which brings us to another contribution we make. Specifically, to study PARO solutions, we rely heavily on Fourier-Motzkin Elimination (FME) as a proof technique. Through the lens of FME we consider optimality of decision rule structures, which then enables us to study PARO. As a byproduct, we illustrate how this proof technique can be applied in ARO more generally, by providing more general and more insightful proofs of known results (not related to Pareto efficiency).

\subsection*{Findings and Contributions}

Before we begin our analysis, we summarize the findings and the contributions of this paper. The treatment presented is restricted to two-stage ARO models that are linear in both decision variables and uncertain parameters. 

\begin{enumerate}
    \item \emph{Concept of PARO Solutions.} In the context of linear ARO problems, we formalize the concept of Pareto Adaptive Robustly Optimal (PARO) solutions. PARO solutions have the property that no other solution and associated adaptive decision rule exist that dominate them, i.e., perform at least as good under any scenario, and perform strictly better under at least some scenario. As \Cref{ex: toy-problem-RT} above has already shown, in the context of ARO problems, PARO solutions can dominate other Pareto optimal solution concepts already proposed in the literature \citep{Iancu14}. In practice, PARO solutions can only yield performance benefits compared with non-PARO solutions, as the latter lead to efficiency losses.
    \item \emph{Properties of PARO Solutions.} We derive several properties of PARO solutions. The main results are that, for linear ARO problems with fixed and continuous recourse, affine dependence on uncertain parameters and a compact feasible region, there exists a first-stage PARO solution, and there exists a piecewise linear (PWL) decision rule that is PARO. To arrive at these results, our analysis relies on FME.
    \item \emph{Finding PARO Solutions and their Practical Value.} We present several approaches to find and/or approximate PARO solutions in practice, amongst others using techniques based on FME. We also conduct a numerical study for a facility location example. The results reveal that (approximate) PARO solutions can yield substantially better performance in non-worst-case scenarios than worst-case optimal and PRO solutions, thus demonstrating the practical value of the proposed methodology.
    \item \emph{FME as a Proof Technique for PARO.} \citet{Zhen18} introduce FME as both a solution and proof technique for ARO. We apply and extend the latter idea, and use FME to prove worst-case and Pareto optimality of various decision rule structures. We extend and/or generalize known results in ARO, not related to Pareto optimality, and provide more insightful proofs; for example, one that uses FME to establish the results by \citet{Bertsimas12} and \citet{Zhen18} on optimality of LDRs under simplex uncertainty sets.
\end{enumerate}

Finally, to better position our contributions vis-\`a-vis the extant literature, we note that, as mentioned earlier, PARO solutions were previously discussed in \citep{tenEikelder21}. They consider a specific 3-variable non-linear ARO problem arising in radiation therapy planning; their approach for finding PARO solutions relies heavily on the specific formulation. No general treatment of PARO was included, although this demonstrates that PARO may also be relevant for non-linear ARO problems. In the current paper, we formalize the PARO concept, derive properties, such as existence of PARO solutions, and also discuss constructive approaches towards finding them. With regards to FME, \citet{Zhen18} were the first to recognize its applicability to linear ARO problems, owing to its ability to eliminate adaptive variables. They use FME as both a solution and proof technique; for the latter the main obstacle is the exponential increase in number of constraints after variable elimination. In the current paper, we apply and extend the ideas of \citet{Zhen18}, and use FME as a proof technique. Through the lens of FME we first consider (worst-case) optimality of decision rule structures, and provide more general and more insightful proofs of known results. Subsequently, we investigate Pareto optimality using FME and present numerical results which demonstrate the value of PARO solutions. 

\subsection*{Notation and Organization}

Boldface characters represent matrices and vectors. All vectors are column vectors and the vector $\bm{a}_i$ is the $i$-th row of matrix $\bm{A}$. The space of all measurable functions from $\mathbb{R}^n$ to $\mathbb{R}^m$ that are bounded on compact sets is denoted by $\mathcal{R}^{n,m}$. The vectors $\bm{e}_i$, $\bm{1}$ and $\bm{0}$ are the standard unit basis vector, the vector of all-ones and the vector of all-zeros, respectively. Matrix $\bm{I}$ is the identity matrix. The relative interior of a set $S$ is denoted by $\text{ri}(S)$; its set of extreme points is denoted by $\text{ext}(S)$.

The manuscript is organized as follows. First, \Cref{sec: ARO} introduces the ARO setting and the notion of PARO. \Cref{sec: properties} investigates the existence of PARO solutions, and \Cref{sec: constructive} presents some practical approaches for the construction of PARO solutions. In \Cref{sec: num-exp} we present the results of our numerical experiments, and \Cref{sec: concluding-remarks} concludes the manuscript. \Cref{app: decision-rules} introduces FME and uses
it to establish (worst-case) optimality of decision rule structures.

\section{Pareto Optimality in (Adaptive) Robust Optimization} \label{sec: ARO}
We first generalize the definition of PRO of \citet{Iancu14} to non-linear static RO problems. The reason for this is that there turns out to be a relation between Pareto efficiency for non-linear static RO problems and linear ARO problems. Let $\bm{x} \in \mathcal{X} \subseteq \mathbb{R}^{n_x}$ denote the decision variables and let $\bm{z} \in U \subseteq \mathbb{R}^L$ denote the uncertain parameters. Let $f: \mathbb{R}^{n_x} \times \mathbb{R}^L \mapsto \mathbb{R}$ and consider the static RO problem 
\begin{align}\label{eq: static-nonlinear}
\min_{\bm{x}\in \mathcal{X}} \max_{\bm{z} \in U}~f(\bm{x},\bm{z}).
\end{align}
Let $\mathcal{X}^{\text{RO}}$ denote the set of robustly optimal (i.e., worst-case optimal) solutions. A robustly optimal solution $\bm{x}$ is PRO if there does not exist another robustly optimal solution $\bm{\bar{x}}$ that performs at least as good as $\bm{x}$ for all scenarios in the uncertainty set, while performing strictly better for at least one scenario. If such a solution $\bm{\bar{x}}$ does exist, it is said to \emph{dominate} $\bm{x}$. In practice, solution $\bm{\bar{x}}$ will always be preferred over $\bm{x}$. If all uncertainty in the objective is moved to constraints using an epigraph formulation, Pareto robust optimality may also be defined in terms of slack variables \citep[Section 4.1]{Iancu14}, but we do not use that definition here. We use the following formal definition:
\begin{definition}[Pareto Robustly Optimal] \label{def: PRO} 
A solution $\bm{x} \in \mathcal{X}^{\text{RO}}$ is PRO to \eqref{eq: static-nonlinear} if there does not exist another $\bm{\bar{x}} \in \mathcal{X}^{\text{RO}}$ such that
\begin{align*}
f(\bm{\bar{x}},\bm{z}) &\leq f(\bm{x},\bm{z}),  ~~\forall \bm{z} \in U, \\
f(\bm{\bar{x}},\bm{\bar{z}}) &< f(\bm{x},\bm{\bar{z}}),  ~~ \text{for some } \bm{\bar{z}} \in U. \tag*{$\blacksquare$}
\end{align*} 
\end{definition}

We aim to extend the concept of PRO to ARO problems. In particular, we consider the following adaptive linear optimization problem:
\begin{subequations} \label{eq: P}
\begin{align}
\min_{\bm{x},\bm{y}(\cdot)} ~&~ \max_{\bm{z} \in U}~ \bm{c}(\bm{z})^{\top} \bm{x} + \bm{d}^{\top} \bm{y}(\bm{z}), \label{eq: P-1}\\
\text{s.t.} ~&~ \bm{A}(\bm{z})\bm{x} + \bm{B}\bm{y}(\bm{z}) \leq \bm{r}(\bm{z}),~~\forall \bm{z} \in U, \label{eq: P-2}
\end{align}
\end{subequations}
where $\bm{z}\in U \subseteq \mathbb{R}^{L}$ is an uncertain parameter, with $U$ a compact, convex uncertainty set with nonempty relative interior. Variables $\bm{x} \in \mathbb{R}^{n_x}$ are the Stage-1 (\emph{here-and-now}) decisions. Usually we will assume $\bm{x}$ to be continuous variables, but we emphasize that all results in the paper also hold if (part of) $\bm{x}$ is restricted to be integer-valued. Variables $\bm{y} \in \mathcal{R}^{L,n_y}$ are also continuous and capture the Stage-2 (\emph{wait-and-see}) decisions, i.e., they are functions of $\bm{z}$. The matrix $\bm{B} \in \mathbb{R}^{m \times n_y}$ and vector $\bm{d} \in \mathbb{R}^{n_y}$ are assumed to be constant (fixed recourse), and $\bm{A}(\bm{z})$, $\bm{r}(\bm{z})$ and $\bm{c}(\bm{z})$ depend affinely on $\bm{z}$:
\begin{align*} 
\bm{A}(\bm{z}) := \bm{A}^0 + \sum_{k=1}^{L} \bm{A}^k z_k,~~\bm{r}(\bm{z}) := \bm{r}^0 + \sum_{k=1}^{L} \bm{r}^k z_k,~~\bm{c}(\bm{z}) := \bm{c}^0 + \sum_{k=1}^{L} \bm{c}^k z_k,
\end{align*}
with $\bm{A}^0,\dotsc,\bm{A}^{L} \in \mathbb{R}^{m \times n_x}$, $\bm{r}^0,\dotsc,\bm{r}^{L} \in \mathbb{R}^{m}$ and $\bm{c}^0,\dotsc,\bm{c}^{L} \in \mathbb{R}^{n_x}$. Uncertainty in the objective \eqref{eq: P-1} can be moved to the constraint using an epigraph formulation. Nevertheless, it is explicitly stated in the objective to facilitate a convenient definition of PARO. Let $\text{OPT}$ denote the optimal (worst-case) objective value of \eqref{eq: P}. We continue by stating several assumptions and definitions regarding adaptive robust feasibility and optimality. 
\begin{definition}[Adaptive Robustly Feasible]\label{def: ARF} 
A pair $(\bm{x},\bm{y}(\cdot))$ is Adaptive Robustly Feasible (ARF) to \eqref{eq: P} if 
$\bm{A}(\bm{z})\bm{x} + \bm{B}\bm{y}(\bm{z}) \leq \bm{r}(\bm{z}),~~\forall \bm{z} \in U$.
\hfill $\blacksquare$ \end{definition}
Sometimes it is useful to consider properties of the first- and second-stage decisions separately. Therefore, we also define adaptive robust feasibility for Stage-1 and Stage-2 decisions individually.
\begin{definition}[Adaptive Robustly Feasible $\bm{x}$ and/or $\bm{y}(\cdot)$] \label{def: ARF-xy} 
\leavevmode
\small
\begin{deflist}
\item A Stage-1 decision $\bm{x}$ is ARF to \eqref{eq: P} if there exists a $\bm{y}(\cdot)$ such that $(\bm{x}, \bm{y}(\cdot))$ is ARF to \eqref{eq: P}. \label{def: ARF-x} 
\item A Stage-2 decision $\bm{y}(\cdot)$ is ARF to \eqref{eq: P} if there exists a $\bm{x}$ such that $(\bm{x}, \bm{y}(\cdot))$ is ARF to \eqref{eq: P}. \hfill $\blacksquare$\label{def: ARF-y} 
\end{deflist}
\end{definition}
The set of all ARF solutions $\bm{x}$ is given by
\begin{align*}  
\mathcal{X} = \{\bm{x} \in \mathbb{R}^{n_x}~|~\exists \bm{y} \in \mathcal{R}^{L,n_y} : \bm{A}(\bm{z})\bm{x} + \bm{B}\bm{y}(\bm{z}) \leq \bm{r}(\bm{z}),~~\forall \bm{z} \in U \}.
\end{align*}
We assume set $\mathcal{X}$ is nonempty, i.e., there exists an $\bm{x}$ that is ARF, and we assume \eqref{eq: P} has a finite optimal objective value, i.e., OPT is a finite number. After feasibility, the natural next step is to formally define optimality.
\begin{definition}[Adaptive Robustly Optimal]\label{def: ARO} 
A pair $(\bm{x},\bm{y}(\cdot))$ is adaptive robustly optimal (ARO)\footnote{To ease exposition, we overload and reuse certain acronyms, such as ARO for ``Adaptive Robust Optimization'' and ``Adaptive Robustly Optimal'', as long as they can be readily disambiguated from the context.} to \eqref{eq: P} if it is ARF and $\bm{c}(\bm{z})^{\top} \bm{x} + \bm{d}^{\top} \bm{y}(\bm{z}) \leq \text{OPT},~~\forall \bm{z} \in U.$
\hfill $\blacksquare$ \end{definition}
We also define adaptive robust optimality for Stage-1 and Stage-2 decisions individually.
\begin{definition}[Adaptive Robustly Optimal $\bm{x}$ and/or $\bm{y}(\cdot)$]\label{def: ARO-xy} 
\leavevmode
\small
\begin{deflist}
\item A Stage-1 decision $\bm{x}$ is ARO to \eqref{eq: P} if there exists a $\bm{y}(\cdot)$ such that $(\bm{x}, \bm{y}(\cdot))$ is ARO to \eqref{eq: P}. \label{def: ARO-x} 
\item A Stage-2 decision $\bm{y}(\cdot)$ is ARO to \eqref{eq: P} if there exists a $\bm{x}$ such that $(\bm{x}, \bm{y}(\cdot))$ is ARO to \eqref{eq: P}. \hfill $\blacksquare$\label{def: ARO-y}  
\end{deflist}
\end{definition}

We are now in position to define Pareto Adaptive Robust Optimality for two-stage ARO problems.
\begin{definition}[Pareto Adaptive Robustly Optimal] \label{def: PARO} 
A pair $(\bm{x},\bm{y}(\cdot))$ is Pareto Adaptive Robustly Optimal (PARO) to \eqref{eq: P} if it is ARO and there does not exist a pair $(\bm{\bar{x}},\bm{\bar{y}}(\cdot))$ that is ARO and
\begin{align*}
\bm{c}(\bm{z})^{\top}\bm{\bar{x}} + \bm{d}^{\top} \bm{\bar{y}}(\bm{z}) &\leq \bm{c}(\bm{z})^{\top}\bm{x} + \bm{d}^{\top}\bm{y}(\bm{z}), ~~\forall \bm{z} \in U,  \\
\bm{c}(\bm{\bar{z}})^{\top}\bm{\bar{x}} + \bm{d}^{\top}\bm{\bar{y}}(\bm{\bar{z}}) &< \bm{c}(\bm{\bar{z}})^{\top}\bm{x} + \bm{d}^{\top}\bm{y}(\bm{\bar{z}}), ~~ \text{for some } \bm{\bar{z}} \in U. \tag*{$\blacksquare$}
\end{align*}\end{definition}
As before, the definitions can be extended to Stage-1 and Stage-2 decisions individually.
\begin{definition}[Pareto Adaptive Robustly Optimal $\bm{x}$ and/or $\bm{y}(\cdot)$]\label{def: PARO-xy}
\leavevmode
\small
\begin{deflist}
\item A Stage-1 decision $\bm{x}$ is PARO to \eqref{eq: P} if there exists a $\bm{y}(\cdot)$ such that $(\bm{x}, \bm{y}(\cdot))$ is PARO to \eqref{eq: P}. \label{def: PARO-x}
\item A Stage-2 decision $\bm{y}(\cdot)$ is PARO to \eqref{eq: P} if there exists a $\bm{x}$ such that $(\bm{x}, \bm{y}(\cdot))$ is PARO to \eqref{eq: P}.  \label{def: PARO-y} \hfill $\blacksquare$
\end{deflist} \end{definition}
Our main interest is in \Cref{def: PARO-x}. The reason for this is that the here-and-now decision $\bm{x}$ is usually the only one that the decision maker has to commit to. In contrast, instead of using decision rule $\bm{y}(\cdot)$, one can often resort to re-solving the optimization problem for the second stage once the value of the uncertain parameter has been revealed. This is known as the folding horizon approach, and it is applicable as long as there is time to re-solve between observing $\bm{z}$ and having to implement $\bm{y}(\bm{z})$. There is no such alternative for $\bm{x}$, however, and different decisions in Stage 1 may lead to different adaptation possibilities in Stage 2. 

Lastly, Pareto optimality can also be investigated for Stage-2 decisions if the Stage-1 decision $\bm{x}$ is fixed. 
\begin{definition}[Pareto Adaptive Robustly Optimal extension $\bm{y}(\cdot)$]\label{def: PARO-ext-y}
A Stage-2 decision $\bm{y}(\cdot)$ is a PARO extension to $\bm{x}^{\ast}$ for \eqref{eq: P} if $(\bm{x}^{\ast},\bm{y}(\cdot))$ is ARO to \eqref{eq: P} and there does not exist another $\bm{\bar{y}}(\cdot)$ such that $(\bm{x}^{\ast},\bm{\bar{y}}(\cdot))$ is ARO to \eqref{eq: P} and 
\begin{align*}
\bm{c}(\bm{z})^{\top}\bm{x}^{\ast} + \bm{d}^{\top}\bm{\bar{y}}(\bm{z}) &\leq \bm{c}(\bm{z})^{\top}\bm{x}^{\ast} + \bm{d}^{\top}\bm{y}(\bm{z}), ~~\forall \bm{z} \in U, \\
\bm{c}(\bm{\bar{z}})^{\top}\bm{x}^{\ast} + \bm{d}^{\top}\bm{\bar{y}}(\bm{\bar{z}}) &< \bm{c}(\bm{\bar{z}})^{\top}\bm{x}^{\ast} + \bm{d}^{\top}\bm{y}(\bm{\bar{z}}), ~~ \text{for some } \bm{\bar{z}} \in U. \tag*{$\blacksquare$}
\end{align*}
\end{definition}

\begin{figure}[htb]
\centering
\includegraphics[scale=0.8]{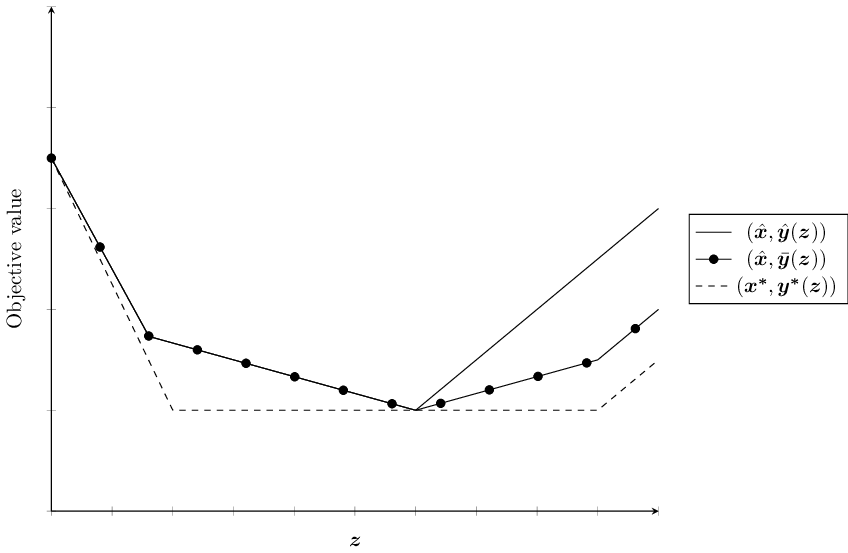}
\caption{\small Illustration of PARO concept. Each graph represents the objective value of \eqref{eq: P} for a given pair $(\bm{x},\bm{y}(\bm{z}))$ as a function of uncertain parameter $\bm{z}$. Solution $(\hat{\bm{x}}, \hat{\bm{y}}(\bm{z}))$ (solid line) is dominated by $(\hat{\bm{x}}, \bar{\bm{y}}(\bm{z}))$ (solid-dotted line). Thus, decision rule $\bar{\bm{y}}(\bm{z})$ may be a PARO extension of $\hat{\bm{x}}$, decision rule $\hat{\bm{y}}(\bm{z})$ is not. Solution $(\bm{x^{\ast}}, \bm{y^{\ast}}(\bm{z}))$ (dashed line) dominates both $(\hat{\bm{x}}, \hat{\bm{y}}(\bm{z}))$ and $(\hat{\bm{x}}, \bar{\bm{y}}(\bm{z}))$ and may be PARO. Solution $\hat{\bm{x}}$ may also be a PARO Stage-1 solution. \label{fig: PARO}}
\end{figure}
\Cref{fig: PARO} illustrates the PARO concept for a single uncertain parameter, and illustrates that care must be exercised when drawing conclusions related to PARO. In the example, $(\bm{x^{\ast}}, \bm{y^{\ast}}(\bm{z}))$ is possibly PARO, but this cannot be concluded from the figure. Also, Stage-1 solution $\hat{\bm{x}}$ might still be PARO, but that cannot be concluded from the figure either. The reason for the latter is that there may be yet another decision rule $\bm{\tilde{y}}(\bm{z})$ so that $(\hat{\bm{x}}, \bm{\tilde{y}}(\bm{z}))$ is not dominated by $(\bm{x^{\ast}}, \bm{y^{\ast}}(\bm{z}))$.

We conclude this section by mentioning three ways that the PARO concept can be generalized and relaxed, although we do not consider these any further. First, \citet{Bertsimas19} define Pareto optimal adaptive solutions for general (non-linear) two-stage ARO problems, which for linear problems is equivalent to our definition of PARO. They subsequently define Pareto optimal affine adaptive solutions, which is equivalent to the definition of PRO after using LDRs, and focus on finding the latter type of solutions. In their numerical studies, \citet{Iancu14} also consider two-stage problems and find PRO solutions after plugging in LDRs.

Secondly, we can also solely relax the requirement that the solution is ARO. For example, often LDRs do not guarantee an ARO solution but do exhibit good practical performance \citep{Kuhn09}. Suppose these yield a worst-case objective value $p$ ($>$ OPT). Then we can define p-PARO solutions as those solutions $(\bm{x},\bm{y}(\cdot))$ that yield an objective value of at most $p$ in each scenario, and are not dominated by another solution $(\bm{\bar{x}},\bm{\bar{y}}(\cdot))$ that yields an objective value of at most $p$ in each scenario. 

Thirdly, PARO may also be defined in terms of slack variables, analogous to the extension of PRO to constraint slacks in \citet[Section 4.1]{Iancu14}. In that paper, a slack value vector is introduced that quantifies the relative importance of slack in each constraint. This scalarization allows the computation of the total slack value of a solution in any scenario. Subsequently PRO (and also PARO) can be defined on this total slack value instead of the objective value. This may be useful in applications where ARO is mainly used for maintaining feasibility, such as immunizing against uncertain renewable energy source output \citep{Jabr13} and adjusting to disturbances in railway timetabling \citep{Polinder19}.

\section{Properties of PARO Solutions} \label{sec: properties}
In this section, we prove existence of PARO solutions for two-stage ARO problems of form \eqref{eq: P}. First, we use FME to prove that a PARO Stage-1 (here-and-now) solution is equivalent to a PRO solution of a PWL convex static RO problem, and use that to prove the existence of a PARO Stage-1 solution. Subsequently, we prove that there exists a PWL decision rule that is PARO to \eqref{eq: P}.

\subsection{Existence of a PARO Stage-1 solution}
We prove existence of PARO Stage-1 solutions in three steps. First, we prove that \eqref{eq: P} is equivalent to a static RO problem with a convex PWL objective. Subsequently, we prove that a PRO solution to this static RO problem is equivalent to a PARO solution to \eqref{eq: P}. Lastly, we prove that PRO solutions to such problems always exist.

\begin{lemma} \label{lemma: P-FME}
If $(\bm{x}^{\ast},\bm{y}^{\ast}(\cdot))$ is ARO to \eqref{eq: P}, $\bm{y}^{\ast}(\cdot)$ satisfies 
\begin{align} \label{appeq: dy = max-h}
\bm{d}^{\top}\bm{y}^{\ast}(\bm{z}) = \max_{(S,T)\in M} \{h_{S,T}(\bm{x}^{\ast},\bm{z}) \},~~\forall \bm{z}\in U,
\end{align} 
and $\bm{x}^{\ast}$ is optimal to 
\begin{align} \label{appeq: P-FME}
\min_{\bm{x} \in \mathcal{X}_{\text{FME}}}~&~\max_{\bm{z} \in U} \bm{c}(\bm{z})^{\top} \bm{x} + \max_{(S,T)\in M} \{h_{S,T}(\bm{x},\bm{z}) \},
\end{align}
with 
\begin{align*} 
M = \big\{(S,T)~|~ \exists k=1,\dotsc,n_y \text{ s.t. }& S \in C_k^{-}, T\in C_k^{+}, \\
& \beta(S,l) = \beta(T,l),~\forall l>k,~ 0\in S \cup T \big\} ,
\end{align*}
and linear functions
\begin{align*}\small
h_{S,T}(\bm{x},\bm{z}) = \sum_{p \in S, p>0} \frac{\alpha(S,p)}{\alpha(T,0) - \alpha(S,0)} \varphi_p(\bm{x},\bm{z}) - \sum_{q \in T, q>0} \frac{\alpha(T,q)}{\alpha(T,0) - \alpha(S,0)} \varphi_q(\bm{x},\bm{z}),
\end{align*}
and sets $C^{-}$, $C^{+}$, functions $\varphi(\cdot)$ and coefficients $\alpha$ and $\beta$ defined as in \Cref{lemma: y-representation}. Conversely, if $\bm{x}^{\ast}$ is optimal to \eqref{appeq: P-FME}, there exists a $\bm{y}^{\ast}(\cdot)$ such that $(\bm{x}^{\ast},\bm{y}^{\ast}(\cdot))$ is ARO to \eqref{eq: P}, and any such $\bm{y}^{\ast}(\cdot)$ satisfies \eqref{appeq: dy = max-h}.
\end{lemma}
\begin{proof}
See \Cref{app: proof-P-FME}.
\end{proof}

The result of \Cref{lemma: P-FME} is also illustrated in \Cref{ex: toy-problem-RT-FME}, where if auxiliary variable $t$ is eliminated the resulting problem has a convex PWL objective. If the number of adaptive variables in \eqref{eq: P} is small enough that full FME can be performed (order of magnitude: 20 adaptive variables \citep{Zhen18}), one can solve \eqref{eq: P-FME} via an epigraph formulation in order to obtain an ARO $\bm{x}$ to \eqref{eq: P}.

\begin{lemma} \label{lemma: PRO-FME}
A solution $\bm{x}^{\ast}$ is PARO to \eqref{eq: P} if and only if it is PRO to 
\begin{align} \label{eq: P-FME}
\min_{\bm{x} \in \mathcal{X}_{\text{FME}}}~&~\max_{\bm{z} \in U} \bm{c}(\bm{z})^{\top} \bm{x} + \max_{(S,T)\in M} \{h_{S,T}(\bm{x},\bm{z}) \},
\end{align}
where each element $(S,T)$ of set $M$ is a pair of sets of original constraints of \eqref{eq: P} and each function $h_{S,T}(\bm{x},\bm{z})$ is bilinear in $\bm{x}$ and $\bm{z}$.
\end{lemma}
\begin{proof}
See \Cref{app: proof-PRO-FME}.
\end{proof}
Thus, existence of a PARO solution to \eqref{eq: P} is now reduced to existence of a PRO solution to a static RO problem with a convex PWL objective in both $\bm{x}$ and $\bm{z}$. For problems without adaptive variables in the objective the following result immediately follows.
\begin{corollary} \label{cor: PRO-FME-d-zero}
If $\bm{d}=\bm{0}$, a solution $\bm{x}^{\ast}$ is PARO to \eqref{eq: P} if and only if it is PRO to
\begin{align*}
\min_{\bm{x} \in \mathcal{X}_{\text{FME}}} \max_{\bm{z} \in U} \bm{c}(\bm{z})^{\top}\bm{x}.
\end{align*}
\end{corollary}
\begin{proof}
This directly follows from plugging in $\bm{d}=\bm{0}$ in the proof of \Cref{lemma: PRO-FME}.
\end{proof}

We can now prove one of our main results: existence of a PARO $\bm{x}$ for any ARO problem of form \eqref{eq: P} with compact feasible region. Our proof uses \Cref{lemma: PRO-FME} and essentially proves existence of a PRO solution to \eqref{eq: P-FME}.
\begin{theorem}\label{thm: x-PARO}
If $\mathcal{X}$ is compact, there exists a PARO $\bm{x}$ to \eqref{eq: P}.
\end{theorem}
\begin{proof}
See \Cref{app: proof-x-PARO}.
\end{proof}
Note that the theorem also holds if $\mathcal{X}$ restricts (some elements of) $\bm{x}$ to be integer-valued. The boundedness assumption on $\mathcal{X}$ cannot be relaxed, because in that case a PRO solution to \eqref{eq: P-FME} need not exist. For example, consider the static RO problem $\max_{x\geq 0}\min_{z \in [0,1]} xz$. The worst-case scenario is $z=0$, and any $x\geq 0$ is worst-case optimal. In any other scenario $z >0$, higher $x$ is better. Any $x$ is dominated by $x+\epsilon$ with $\epsilon>0$, and there is no PRO solution.

\subsection{Existence of a PARO piecewise linear decision rule}
Now that existence of a PARO $\bm{x}$ is established, we investigate the structure of decision rule $\bm{y}(\cdot)$. We illustrate via an example that for any ARF $\bm{x}$ there exists a PWL PARO extension $\bm{y}(\cdot)$.
\begin{example} \label{ex: PWL-PARO-extension}
\small 
Consider the following ARO problem, a slight adaptation of \Cref{ex: hybrid}:
\begin{subequations} \label{eq: ex-PWL}
\begin{align}
\min_{x,\bm{y}(\cdot)}~&~ \max_{\bm{z} \in [0,1]^4} x - y_1(\bm{z}) + y_2(\bm{z}), \label{eq: ex-PWL-0}\\
\text{s.t.}~&~ x - y_2(\bm{z}) \leq -z_0 - \frac{1}{2}z_1,~~\forall (z_0,z_1) \in [0,1]^2, \label{eq: ex-PWL-1}\\
~&~ -x + y_1(\bm{z}) + y_2(\bm{z}) \leq z_0 + \frac{1}{2}z_2 + \frac{1}{2}z_3 + 2,~~\forall (z_0,z_2,z_3) \in [0,1]^3, \label{eq: ex-PWL-2}\\
~&~ 1 \leq y_1(\bm{z}) \leq 2,~~\forall \bm{z} \in U, \label{eq: ex-PWL-3}\\
~&~ \frac{3}{2} \leq y_2(\bm{z}) \leq 2,~~\forall \bm{z} \in U. \label{eq: ex-PWL-4}
\end{align}
\end{subequations}
We eliminate $y_1(\bm{z})$ and $y_2(\bm{z})$ in constraints \eqref{eq: ex-PWL-1}-\eqref{eq: ex-PWL-4} analogous to \Cref{ex: hybrid}, and find the ARF solution $x^{\ast} = \frac{1}{2}$ and the following bounds on $y_1(\bm{z})$ and $y_2(\bm{z})$:
\begin{subequations} \label{eq: ex-bounds-PWL}
\begin{align}
1 \leq & y_1(\bm{z}) \leq \min\{2, z_0 - y_2(\bm{z}) + \frac{5}{2}\}, \\
\max\{\frac{3}{2}, 1 + z_0 \} \leq & y_2(\bm{z}) \leq \min \{2, \frac{3}{2} + z_0 \}. \label{eq: ex-bounds-PWL-2}
\end{align}
\end{subequations}
Variables $y_1(\bm{z})$ and $y_2(\bm{z})$ have not been eliminated in the objective. Therefore, any decision rule satisfying \eqref{eq: ex-bounds-PWL} is ARF to \eqref{eq: ex-PWL} but need not be ARO or PARO. 

Variable $y_1(\bm{z})$ does not appear in the bounds of $y_2(\bm{z})$, so we can consider its individual contribution to the objective value. The objective coefficient of $y_1(\bm{z})$ is negative, so for any $\bm{z}$ (including the worst-case) the best possible contribution of $y_1(\bm{z})$ to the objective value is achieved if we set $y_1(\bm{z})$ equal to its upper bound. Therefore, for the given $x^{\ast}$, we have the following PWL PARO extension as a function of $y_2(\bm{z})$:
\begin{align*}
y_1^{\ast}(\bm{z}) = \min\{2, z_0 - y_2(\bm{z}) + \frac{5}{2}\}.
\end{align*} 
Now that $y_1(\bm{z})$ is eliminated in the objective value, it remains to find the optimal decision rule for $y_2(\bm{z})$. Variable $y_2(\bm{z})$ now appears directly in the objective~\eqref{eq: ex-PWL-0} and through its occurence in the decision rule $y_1^{\ast}(\bm{z})$. For fixed $\bm{z}$, the optimal $y_2(\bm{z})$ is determined by solving
\begin{align*} 
\min_{y_2(\bm{z})} &- \min\{2, z_0 - y_2(\bm{z}) + \frac{5}{2}\} + y_2(\bm{z}), \\
\text{s.t.}~&~ \max\{\frac{3}{2}, 1 + z_0 \} \leq y_2(\bm{z}) \leq \min \{2, \frac{3}{2} + z_0 \}.
\end{align*}
One can easily see that the objective is increasing in $y_2(\bm{z})$, so for any $\bm{z}$ the best possible contribution of $y_2(\bm{z})$ to the objective value is achieved if we set $y_2(\bm{z})$ equal to its lower bound. Thus, for the given $x^{\ast}$, we have the following PWL PARO extension:
\begin{align*}
y_2^{\ast}(\bm{z}) = \max\{\frac{3}{2}, 1 + z_0 \}.
\end{align*}
Note that plugging in a PWL argument in a PWL function retains the piecewise linear structure. Therefore, we also obtain the following PWL PARO extension for $y_1^{\ast}(\bm{z})$:
\begin{align*}
y_1^{\ast}(\bm{z}) = \min\{2, z_0 - \max\{\frac{3}{2}, 1 + z_0 \} + \frac{5}{2}\}.
\end{align*} 

Note that we did not move adaptive variables in the objective to the constraints using an epigraph variable, as was done in \Cref{ex: toy-problem-RT-FME}. Using an epigraph variable for the objective ensures that each decision rule satisfying the bounds is worst-case optimal, but prevents from comparing performance in other scenarios. Naturally, computationally it has the major advantage that it remains a linear program.
\end{example}

\citet{Bemporad03} show worst-case optimality of PWL decision rules for right-hand polyhedral uncertainty, i.e., ARO PWL decision rules in our terminology. \citet[Theorem 3]{Zhen18} generalize this to problems of form \eqref{eq: P} with particular assumptions on the uncertainty set. These decision rules are general PWL in $\bm{z}$ for all variables $y_j$, $j\neq l$, where $y_l$ is the last eliminated variable in the FME procedure. The decision rule is convex or concave PWL in $y_l$. These results solely consider the performance of PWL decision rules in the worst-case. \Cref{ex: PWL-PARO-extension} illustrates that for any ARF $\bm{x}$ there exists a PWL PARO extension $\bm{y}(\cdot)$. The lemma below formalizes this claim.
\begin{lemma} \label{lemma: PARO-extension-PWL}
For any $\bm{x}$ that is ARF to \eqref{eq: P} there exists a PARO extension $\bm{y}(\bm{z})$ that is PWL in $\bm{z}$.
\end{lemma}
We present two proofs to \Cref{lemma: PARO-extension-PWL}; one via FME using the idea of \Cref{ex: PWL-PARO-extension}, and one via basic solutions in linear optimization.
\begin{proof}[Proof of \Cref{lemma: PARO-extension-PWL} via FME.]
See \Cref{app: proof-PARO-extension-PWL-FME}.
\end{proof}
\begin{proof}[Proof of \Cref{lemma: PARO-extension-PWL} via linear optimization.]
See \Cref{app: proof-PARO-extension-PWL-LO}.
\end{proof}
In both proofs the constructed decision rule is in fact optimal for \emph{all} scenarios in the uncertainty set. As long as $\bm{x}$ is fixed, this is necessary for PARO solutions. The following theorem establishes the existence of PARO PWL decision rules.
\begin{theorem} \label{cor: PARO-PWL}
If $\mathcal{X}$ is compact, there exists a PARO $\bm{y}(\cdot)$ for \eqref{eq: P} such that $\bm{y}(\bm{z})$ is PWL in $\bm{z}$.
\end{theorem}
\begin{proof}
According to \Cref{thm: x-PARO} there exists a PARO $\bm{x}$, and according to \Cref{lemma: PARO-extension-PWL} there exists a PARO extension $\bm{y}(\cdot)$ for this $\bm{x}$ that is PWL in $\bm{z}$. It immediately follows that $\bm{y}(\cdot)$ is PARO.
\end{proof}

\section{Constructing PARO Solutions - Special Cases} \label{sec: constructive}
The methods used in the existence proofs of \Cref{sec: properties} are not computationally tractable, i.e., they provide little guidance for finding PARO solutions in practice. In this section, we discuss several practical methods for finding PARO solutions for special cases of problem \eqref{eq: P}. Several results assume that $\bm{d} = \bm{0}$, i.e., Stage-2 variables do not appear in the objective. This can be a limiting assumption, although it is satisfied in applications where ARO is used for maintaining robust feasibility \citep{Jabr13, Polinder19}. In the next section, a general methodology is presented that does not rely on this assumption.

First, we consider special cases of problem \eqref{eq: P} where particular decision rule structures guarantee PARO solutions in case $\bm{d} = \bm{0}$. Subsequently, we show how for fixed $\bm{x}$ we can check whether $\bm{y}(\cdot)$ is a PARO extension. After that, we consider an application of the finite subset approach of \citet{Hadjiyiannis11}. Lastly, we present a method for finding a PARO solution if $\bm{d} = \bm{0}$ and a convex hull description of the uncertainty set is available.

\subsection{Known worst-case optimal decision rules} \label{sec: PARO-decision-rule-structure}
In \Cref{app: decision-rules}, using an FME lens, we analyze various decision rule structures, considering special uncertainty set forms, including so-called constraintwise, hybrid and block forms. This enables us to prove that an ARF decision rule with a particular structure exists for \emph{every} ARF $\bm{x}$, instead of solely proving it is optimal for an ARO $\bm{x}$. For example, for ARO problems with hybrid uncertainty, for any ARF Stage-1 decision there exists an ARF decision that depends only on the non-constraintwise uncertain parameter. It turns out that, in case there are no adaptive variables in the objective, PRO solutions to the static problem obtained after plugging in that decision rule structure are PARO solutions to the original ARO problem. To formalize this, let $\bm{y}(\bm{z}) = f_{\bm{w}}(\bm{z})$ be a decision rule with known form $f$ (e.g., linear or quadratic) and finite number of parameters $\bm{w} \in \mathbb{R}^p$, such that $f_{\bm{w}}(\bm{z}) \in \mathcal{R}^{L,n_y}$ for any $\bm{w}$. 
\begin{theorem} \label{theorem: DR-PARO}
Let $P$ denote an ARO problem of form \eqref{eq: P} with $\bm{d} = \bm{0}$ and where for any ARF $\bm{x}$ there exists an ARF decision rule of form $\bm{y}^{\ast}(\bm{z}) = f_{\bm{w}}(\bm{z})$ for some $\bm{w}$. Then any $\bm{x}^{\ast}$ that is PRO to the static robust optimization problem obtained after plugging in decision rule structure $f_{\bm{w}}(\bm{z})$ is PARO to $P$.
\end{theorem}
\begin{proof}
See \Cref{app: proof-DR-PARO}.
\end{proof}

Due to \Cref{lemma: hybrid-ARF,lemma: block-ARF,lemma: simplex-LDR-ARF}, the following result immediately follows for hybrid, block and simplex uncertainty.
\begin{corollary} \label{cor: SDR-LDR-PARO}
\leavevmode
\begin{corlist}
\item Let $P_{\text{hybrid}}$ denote an ARO problem of form \eqref{eq: P} with $\bm{d}=\bm{0}$ and hybrid uncertainty. Let $Q$ denote the static robust optimization problem obtained from $P_{\text{hybrid}}$ by plugging in a decision rule structure that depends only on the non-constraintwise parameter. Any $\bm{x}^{\ast}$ that is PRO to $Q$ is PARO to $P_{\text{hybrid}}$.\label{cor: hybrid-PARO}
\item Let $P_{\text{block}}$ denote an ARO problem of form \eqref{eq: P} with $\bm{d}=\bm{0}$ and block uncertainty. Let $Q$ denote the static robust optimization problem obtained from $P_{\text{block}}$ by plugging in a decision rule structure where adaptive variable $\bm{y}_{(v)}^{\ast}(\cdot)$ depend only on $\bm{z}_{(v)}$ for all $v=1,\dotsc,V$. Then any $\bm{x}^{\ast}$ that is PRO to $Q$ is PARO to $P_{\text{block}}$. \label{cor: block-PARO}
\item Let $P_{\text{simplex}}$ denote an ARO problem of form \eqref{eq: P} with $\bm{d}=\bm{0}$ and a simplex uncertainty set, i.e., $U = \text{Conv}(\bm{z}^1,\dotsc,\bm{z}^{L+1})$, with $\bm{z}^{j} \in \mathbb{R}^{L}$ such that $\bm{z}^1,\dotsc,\bm{z}^{L+1}$ are affinely independent. Let $Q$ denote the static robust optimization problem obtained from $P_{\text{simplex}}$ by plugging in an LDR structure. Then any $\bm{x}^{\ast}$ that is PRO to $Q$ is PARO to $P_{\text{simplex}}$. \label{cor: simplex-LDR-PARO}
\end{corlist}
\end{corollary}
\begin{proof}
See \Cref{app: proof-SDR-LDR-PARO}.
\end{proof}
The case with constraintwise uncertainty is a special case of \Cref{cor: hybrid-PARO}. The case with one uncertain parameter is a special case of \Cref{cor: simplex-LDR-PARO}. Note that, unlike for worst-case optimization in \Cref{app: decision-rules}, it is necessary that $\bm{d}=\bm{0}$, because our definition of PRO involves the term $\bm{d}$. If $\bm{d} \neq \bm{0}$, the results above do not hold. This is also illustrated in \Cref{ex: toy-problem-RT} in \Cref{sec: introduction}.

The results of \Cref{cor: SDR-LDR-PARO} can be combined. For example, for problems with both simplex uncertainty and hybrid uncertainty,  \Cref{cor: hybrid-PARO} and \Cref{cor: simplex-LDR-PARO} together imply that one needs to consider only decision rules that are affine in the non-constraintwise parameter, if there are no adaptive variables in the objective. Simplex uncertainty sets arise in a variety of applications and can be used to approximate other uncertainty sets \citep{BenTal20}.

\subsection{Check whether a decision rule is a PARO extension}
If the Stage-1 decision $\bm{x}$ is fixed, one can verify whether the decision rule $\bm{y}$ is a PARO extension (\Cref{def: PARO-ext-y}) as follows.
\begin{lemma} \label{lemma: PARO-x-unique}
Let $(\bm{x}^{\ast}, \bm{y}^{\ast}(\cdot))$ be an ARO solution to \eqref{eq: P}. Consider the problem
\begin{subequations} \label{eq: check-PARO-extension}
\begin{align}
\max_{\bm{z},\bm{y}} ~&~ \bm{d}^{\top} (\bm{y}^{\ast}(\bm{z}) - \bm{y}), \\
\text{s.t.}~&~ \bm{A}(\bm{z}) \bm{x}^{\ast} + \bm{B}\bm{y} \leq \bm{r}(\bm{z}), \\
~&~ \bm{z} \in U.
\end{align}
\end{subequations}
If the optimal objective value is zero, $\bm{y}^{\ast}(\cdot)$ is a PARO extension of $\bm{x}^{\ast}$.  If the objective value is positive, then $\bm{y}^{\ast}(\cdot)$ is not a PARO extension of $\bm{x}^{\ast}$ and the suboptimality of $\bm{y}^{\ast}(\cdot)$ is bounded by the optimal objective value. 
\end{lemma}
\begin{proof}
See \Cref{app: proof-PARO-x-unique}.
\end{proof}
If the optimal value to \eqref{eq: check-PARO-extension} is positive and $(\bm{\bar{z}},\bm{\bar{y}})$ denotes an optimal solution to \eqref{eq: check-PARO-extension}, then $\bm{\bar{z}}$ is a scenario where the suboptimality bound is attained, and $\bm{\bar{y}}$ is an optimal decision for this scenario. Also, note that if the optimal value of \eqref{eq: check-PARO-extension} equals zero, the pair $(\bm{x}^{\ast}, \bm{y}^{\ast}(\cdot))$ need not be PARO; there may be a different pair $(\bm{\hat{x}}, \bm{\hat{y}}(\cdot))$ that dominates the current pair.

\subsection{Unique ARO solution on finite subset of scenarios is PARO}
The finite subset approach of \citet{Hadjiyiannis11} can be used in a PARO setting as well. If the lower bound problem has a unique solution and this solution is feasible to the original problem, it is a PARO solution to the original problem. This is formalized in \Cref{lemma: x-unique-Hadjiyiannis}. 
\begin{lemma} \label{lemma: x-unique-Hadjiyiannis}
Let $S = \{\bm{z}^1,\dotsc,\bm{z}^N\}$ denote a finite set of scenarios, $S \subseteq U$. Let $\bm{x}^{\ast}$ be the unique ARO here-and-now decision for which there exist $\bm{y}^{1\ast},\dotsc,\bm{y}^{N\ast}$ such that $(\bm{x}^{\ast},\bm{y}^{1\ast},\dotsc,\bm{y}^{N\ast})$ is an optimal solution to
\begin{subequations} \label{eq: x-unique-Hadjiyiannis}
\begin{align}
\min_{\bm{x},\bm{y}^1,\dotsc,\bm{y}^N} ~&~ \max_{i=1,\dotsc,N} \{\bm{c}(\bm{z}^i)^{\top} \bm{x} + \bm{d}^{\top} \bm{y}^i\}, \label{eq: x-unique-Hadjiyiannis-1}\\
\text{s.t.}~&~ \bm{A}(\bm{z}^i) \bm{x} + \bm{B} \bm{y}^i \leq \bm{r(\bm{z}^i)},~\forall i=1,\dotsc,N. \label{eq: x-unique-Hadjiyiannis-2}
\end{align}
\end{subequations}
Then $\bm{x}^{\ast}$ is PARO to \eqref{eq: P}.
\end{lemma}
\begin{proof}
Let $(\bm{\bar{x}}, \bm{\bar{y}}(\cdot))$ be ARO to \eqref{eq: P} with $\bm{\bar{x}}$ unequal to $\bm{x}^{\ast}$. Then the solution $(\bm{\bar{x}},\bm{\bar{y}}(\bm{z}^1),\dotsc,\bm{\bar{y}}(\bm{z}^N))$ is feasible to \eqref{eq: x-unique-Hadjiyiannis}. Because $\bm{x}^{\ast}$ is the unique here-and-now ARO decision that can be extended to an optimal solution of \eqref{eq: x-unique-Hadjiyiannis}, it holds that
\begin{align*}
\bm{c}(\bm{z}^i)^{\top} \bm{x}^{\ast} + \bm{d}^{\top} \bm{y}^{i\ast}  <  \bm{c}(\bm{z}^i)^{\top} \bm{\bar{x}} + \bm{d}^{\top} \bm{\bar{y}}(\bm{z}^i) \text{ for some } \bm{z}^i \in S.
\end{align*} 
That is, for each $\bm{\bar{x}}$ that is ARO to \eqref{eq: P} and unequal to $\bm{x}^{\ast}$ there is at least one scenario $\bm{z}^i$ in $U$ for which $\bm{x}^{\ast}$ outperforms $\bm{\bar{x}}$. This implies that $\bm{x}^{\ast}$ is PARO to \eqref{eq: P}.
\end{proof}
Objective \eqref{eq: x-unique-Hadjiyiannis-1} optimizes for the worst-case scenario $\bm{z}^i$ in set $S$, and \eqref{eq: x-unique-Hadjiyiannis-2} ensures that $\bm{x}^{\ast}$ is feasible for all scenarios in $S$. It should be noted that requiring $\bm{x}^{\ast}$ to be both ARO to \eqref{eq: P} and a unique optimal solution to \eqref{eq: x-unique-Hadjiyiannis} is quite restrictive.

\subsection{Convex hull description of scenario set} \label{sec: PARO-convex-hull}
Next, consider the case where the uncertainty set is given by the convex hull of a finite set of points, i.e., $U = \text{Conv}(\bm{z}^1,\dotsc,\bm{z}^N)$. Additionally assume that there are no adaptive variables in the objective. Then \eqref{eq: P} is equivalent to
\begin{subequations} \label{eq: P-finite}
\begin{align}
\min_{\bm{x},\bm{y}^1,\dotsc,\bm{y}^N} ~&~ \max_{i=1,\dotsc,N} \bm{c}(\bm{z}^i)^{\top} \bm{x}, \\
\text{s.t.}~&~ \bm{A}(\bm{z}^i)\bm{x} + \bm{B}\bm{y}^i \leq \bm{r}(\bm{z}^i),~\forall i=1,\dotsc,N.
\end{align}
\end{subequations}
Analogous to \citet{Iancu14}, after finding an ARO solution we can perform an additional step by optimizing the set of ARO solutions over a scenario in the relative interior (denoted $\text{ri}(\cdot)$) of the convex hull of our finite set of scenarios. This produces a PARO here-and-now solution to \eqref{eq: P}.
\begin{lemma} \label{lemma: PARO-extreme-points-d0}
Let $\bm{d}=\bm{0}$. Let $U = \text{Conv}(\bm{z}^1,\dotsc,\bm{z}^N)$, $\bm{\bar{z}} \in \text{ri}(U)$ and let $(\bm{x}^{\ast},\bm{y}^{1\ast},\dotsc,\bm{y}^{N\ast})$ denote an optimal solution to 
\begin{subequations}\label{eq: aux-d0}
\begin{align}
\min_{\bm{x},\bm{y}^1,\dotsc,\bm{y}^N} ~&~ \bm{c}(\bm{\bar{z}})^{\top} \bm{x},\\
\text{s.t.}~&~ \bm{A}(\bm{z}^i) \bm{x} + \bm{B} \bm{y}^i \leq \bm{r(\bm{z}^i)},~\forall i=1,\dotsc,N, \label{eq: aux-d0-2}\\
~&~ \bm{c}(\bm{z^i})^{\top} \bm{x} + \bm{d}^{\top} \bm{y}^i \leq \text{OPT},~\forall i=1,\dotsc,N, \label{eq: aux-d0-3}
\end{align}
\end{subequations}
where OPT denotes the optimal objective value of \eqref{eq: P-finite}. Then $\bm{x}^{\ast}$ is PARO to \eqref{eq: P}.
\end{lemma}
\begin{proof}
See \Cref{app: proof-PARO-extreme-points-d0}.
\end{proof}

\section{Constructing PARO Solutions - General Methodology} \label{sec: constructive-general}
Adaptive robust optimization problems of form \eqref{eq: P} are in general NP-hard \citep{Guslitser02}, and finding ARO solutions is still the focus of ongoing research \citep{Yanikoglu19}. Thus, finding a general methodology that, given an ARO solution to \eqref{eq: P}, can produce a PARO solution is not an easy task either. In this section we present a general methodology for finding PARO solutions to ARO problems of form \eqref{eq: P},  which involves solving a sequence of bilinear optimization problems; by solving these heuristically, the method can be used to find `approximate' PARO solutions in practice. First we describe the methodology, after that we discuss how to solve the bilinear optimization problems.

\subsection{Constraint \& column generation procedure for finding a PARO solution} \label{sec: CCG}
The starting point is the constraint-and-column generation (C\&CG) method of \citep{Zeng13} for finding ARO solutions; we first briefly describe this method. Let $M \subseteq{U}$ denote a finite set of scenarios, e.g., solely the nominal scenario. Consider the following LP:
\begin{subequations} \label{eq: P1}
\begin{align}
(P_1(M))~~~~~~\min_{\bm{x},\bm{y}^i,\mu} ~&~ \mu, \label{eq: P1-1}\\
\text{s.t.} ~&~ \mu \geq \bm{c}(\bm{z}^i)^{\top} \bm{x} + \bm{d}^{\top} \bm{y}^i,~~\forall i=1,\dotsc,|M|, \label{eq: P1-2}\\ 
~&~ \bm{A}(\bm{z}^i)\bm{x} + \bm{B}\bm{y}^i \leq \bm{r}(\bm{z}^i),~~\forall i=1,\dotsc,|M|. \label{eq: P1-3}
\end{align}
\end{subequations}
The Stage-1 solution $\bm{\hat{x}}$ to $P_1(M)$ is ARO for the set $M$, but not neccesarily ARO (or ARF) for the true uncertainty set $U$. Define the subproblem
\begin{align} \label{eq: Q}
(Q(\bm{\hat{x}}) ~~~~ \max_{\bm{z} \in U}&~ \bm{c}(\bm{z})^{\top}\bm{\hat{x}} + \min_{\bm{y}} \{ \bm{d}^{\top} \bm{y} ~|~ \bm{A}(\bm{z})\bm{\hat{x}} + \bm{B}\bm{y} \leq \bm{r}(\bm{z}) \},
\end{align}
and let $q(\bm{\hat{x}})$ denote the objective value. Either $q(\bm{\hat{x}})$ is finite with solution $\bm{\hat{z}}$, or a scenario $\bm{\hat{z}}$ is identified where the inner minimization problem is infeasible. In that case set $q(\bm{\hat{x}})$ to $+\infty$ by convention. In either case, the solution $\bm{\hat{z}}$ is a scenario in $U$ where the Stage-1 solution $\bm{\hat{x}}$ to $P_1(M)$ performs worst. If the objective value of $Q(\bm{\hat{x}})$ is higher than the objective value of $P_1(M)$, then $M$ did not contain the worst-case scenario, and $\bm{\hat{x}}$ is not ARO. That is, $\bm{\hat{z}}$ provides either a feasibility or (worst-case) optimality cut. The scenario $\bm{\hat{z}}$ is added to $M$, and the procedure is repeated until both objective values are sufficiently close. \Cref{alg: CCG} provides the pseudocode, see \citet{Zeng13} for more details.

\begin{algorithm}[htb!]
\Begin{
Set $k=0$, set $\epsilon>0$, choose an initial $M^k \subseteq U$ and set boolean \textit{ARO} to \textit{False}\;
\While{not \textit{ARO}}{
Solve problem $P_1(M^k)$. Denote the stage-1 solution by $\bm{x}^k$ and the objective value by $\mu^k$\;
Solve the subproblem $Q(\bm{x}^k)$\;
Denote the objective value by $q(\bm{x}^k)$, denote the solution by $\bm{z}^k$\;
\eIf{$q(\bm{x}^k) \leq \mu^k +\epsilon$}{
	Set \emph{ARO} to \emph{True}\;
	}{
	Set $M^{k+1} =  M^k \cup \{\bm{z}^k\}$\;
	Set $k \leftarrow k+1$\;
	}
}
Return $\bm{x}^{\text{ARO}} := \bm{x}^k$\;
}
\caption{\small C\&CG method of \citep{Zeng13} \label{alg: CCG}}
\end{algorithm}
The result of termination of \Cref{alg: CCG} after $k$ iterations is an ARO solution $\bm{x}^{\text{ARO}} := \bm{x}^k$; additionally let $M^{\text{ARO}}:= M^k$ denote the resulting set of scenarios and let $OPT := \mu^k$ denote the worst-case optimal objective value.

A PARO solution can be found by appending another C\&CG procedure to \Cref{alg: CCG}. In each iteration $k$ of the second C\&CG procedure, we find a candidate Stage-1 solution $\bm{x}^c$ that satisfies three conditions:
\begin{condition}
\leavevmode
\begin{conlist}
\item $\bm{x}^c$ is feasible for all scenarios $\bm{z}^l \in M^k$.  \label{con: CCG-1}
\item $\bm{x}^c$ results in an objective value of at most $v^l(\bm{x}^k,\bm{z}^l)$ for all scenarios $\bm{z}^l \in M^k$, with 
\begin{align}
\hspace*{-0.4cm}
\small 
v(\bm{x}^k,\bm{z}^l) = 
\begin{cases}
\text{OPT} & \forall \bm{z}^l \in M^{\text{ARO}} \\
\bm{c}(\bm{z}^l)^{\top}\bm{x}^k + \min_{\bm{y}} \{\bm{d}^{\top} \bm{y} | \bm{A}(\bm{z}^l)\bm{x}^k + \bm{B}\bm{y} \leq \bm{r}(\bm{z}^l) \}
& \forall  \bm{z}^l \in M^k \backslash M^{\text{ARO}},
\end{cases}
\end{align} 
i.e., $\bm{x}^c$ is feasible on the original scenario set $M^{\text{ARO}}$, and performs at least as good as $\bm{x}^k$ on any scenario subsequently added to $M^k$. \label{con: CCG-2}
\item $\bm{x}^c$ results in a strictly lower objective value than the current Stage 1 solution $\bm{x}^k$ on a new (to be determined) scenario $\bm{z}^c$.  \label{con: CCG-3}
\end{conlist}
\end{condition}
This is achieved by solving the following optimization problem:\footnote{To limit notational burden, we overload the notation $\bm{x}^c$ and $\bm{z}^c$ and use these for both the optimization variables in \eqref{eq: P2} and their values in the optimal solution.}
\begin{subequations} \label{eq: P2}
\small
\begin{align}
\min_{\substack{\bm{z}^c,\bm{x}^c,\bm{y}^c, \\ \{\bm{y}^l\}_{l=1}^{|M^k|}}} ~&~  \max_{\bm{y}^k:\bm{A} \bm{x}^k + \bm{B} \bm{y^k} \leq \bm{r}(\bm{z}^c)} (\bm{c}(\bm{z}^c)^{\top}\bm{x}^c + \bm{d}^{\top}\bm{y}^c) - (\bm{c}(\bm{z}^c)^{\top}\bm{x}^k + \bm{d}^{\top}\bm{y}^k), \\
\text{s.t.}~&~ \bm{A}(\bm{z}^c) \bm{x}^c + \bm{B} \bm{y}^c \leq \bm{r}(\bm{z}^c), \label{eq: P2-2}\\
(P_2(\bm{x}^k,M^k))\hspace{1cm}		   ~&~ \bm{A}(\bm{z}^l) \bm{x}^c + \bm{B} \bm{y}^l \leq \bm{r}(\bm{z}^l),~~\forall l=1,\dotsc,|M^k|, \label{eq: P2-3}\\
	   		   ~&~ \bm{c}(\bm{z}^l)^{\top} \bm{x}^c + \bm{d}^{\top} \bm{y}^l \leq v(\bm{x}^k,\bm{z}^l),~~\forall l=1,\dotsc,|M^k|, \label{eq: P2-4}\\
		   ~&~ \bm{z}^c \in U. \label{eq: P2-5}
\end{align}
\end{subequations}
Constraints \eqref{eq: P2-3} and \eqref{eq: P2-4} enforce \Cref{con: CCG-1,con: CCG-2}, respectively. The objective aims to find the scenario $\bm{z}^c$ where the candidate solution $\bm{x}^c$ most outperforms $\bm{x}^k$. The inner maximization problem finds the optimal recourse decision $\bm{y}^k$ for the current solution $\bm{x}^k$ at $\bm{z}^c$. Constraint \eqref{eq: P2-2} ensures that $\bm{x}^c$ is feasible for the new scenario $\bm{z}^c$. Note that setting $\bm{x}^c = \bm{x}^k$ is feasible and yields objective value $0$, so $P_2(\bm{x}^k,M^k)$ is always feasible and has a nonpositive objective value.

If the objective value of \eqref{eq: P2} is strictly negative, then candidate solution $\bm{x}^c$ satisfies \Cref{con: CCG-1,con: CCG-2,con: CCG-3}, i.e., it dominates $\bm{x}^k$ on set $M^k$. However, because $M^k$ is only a subset of $U$, a candidate solution $\bm{x}^c$ is not neccesarily ARO (or even ARF). Thus, we again solve problem $Q(\bm{x}^c)$; let $q(\bm{x}^c)$ denote its objective value and $\bm{z}^q$ the optimizer. We distinguish two cases:

\begin{enumerate}[leftmargin=*,label=\textbullet]
\item $q(\bm{x}^c) \leq \text{OPT}$. Candidate solution $\bm{x}^c$ is an ARO solution that dominates $\bm{x}^k$ on set $M^k \cup \{\bm{z}^c\}$. For the next iteration of the C\&CG procedure we update the current best solution: $\bm{x}^{k+1} = \bm{x}^c$. The new set of scenarios is $M^{k+1} = M^k \cup \{\bm{z}^c\}$. Further improvements over $\bm{x}^c$ may be possible, so it is not neccesarily PARO. In preparation of the next iteration, compute $v(\bm{x}^{k+1},\bm{z}^l)$ for all $l=1,\dotsc,M^{k+1}$.

\item $q(\bm{x}^c) > \text{OPT}$. Candidate solution $\bm{x}^c$ is not ARO, so we set $\bm{x}^{k+1} = \bm{x}^k$. For the next iteration we set $M^{k+1} = M^k \cup \{\bm{z}^q\}$, to render solution $\bm{x}^c$ infeasible in the next iteration. In preparation of the next iteration, compute $v(\bm{x}^{k+1},\bm{z}^c)$.
\end{enumerate}
The above procedure is repeated until the objective value of \eqref{eq: P2} is nonnegative. \Cref{alg: CCG-2} describes the resulting C\&CG algorithm and \Cref{lemma: PARO-CCG} proves that it yields a PARO Stage-1 solution.

\begin{algorithm}[htb!]
\Begin{
Run \Cref{alg: CCG}\;
Initialize $\bm{x}^k = \bm{x}^{\text{ARO}}$ and $M^k = M^{\text{ARO}}$\;
Set $v^l(\bm{x}^k,\bm{z}^l) = \text{OPT}$ for all $l=1,\dotsc,|M^k|$\;
Set boolean \emph{PARO} to \emph{False}\;
\While{not \emph{PARO}}{
	Solve problem $P_2(\bm{x}^k,M^k)$ and denote the objective value $p_2^k$\;
	\eIf {$p_2^k \geq 0$}{
	Set  \emph{PARO} to \emph{True}\; 
	}
	{Denote the stage-1 solution of $P_2(\bm{x}^k,M^k)$ by $\bm{x}^c$ and the new scenario by $\bm{z}^c$\;
	Solve subproblem $Q(\bm{x}^c)$\;
	\eIf{$q(\bm{x}^c) \leq \text{OPT}$}
	{
		Set $M^{k+1} \leftarrow M^k \cup \{\bm{z}^c\}$\;
		Set $\bm{x}^{k+1} \leftarrow \bm{x}^c$\;
		Compute	$v(\bm{x}^{k+1},\bm{z}^l)$ for all $l=1,\dotsc,|M^{k+1}|$\;
	}
	{
		Denote the optimizer of $Q(\bm{x}^c)$ by $\bm{z}^q$\;
		Set $M^{k+1} =  M^k \cup \{\bm{z}^q \}$\;
		Set $\bm{x}^{k+1} \leftarrow \bm{x}^k$\;
		Compute	$v(\bm{x}^{k+1},\bm{z}^c)$\;
	}
	Set $k \leftarrow k+1$\;
	}
}
Return $\bm{x}^{\text{PARO}} := \bm{x}^k$\;
}
\caption{\small C\&CG method for finding a PARO solution \label{alg: CCG-2}}
\end{algorithm}

\begin{lemma} \label{lemma: PARO-CCG}
A solution $\bm{x}^{\text{PARO}}$ obtained from \Cref{alg: CCG-2} is PARO to \eqref{eq: P}.
\end{lemma}
\begin{proof}
See \Cref{app: proof-PARO-CCG}.
\end{proof}

In the special case that the vertices of $U$ can be explicitly enumerated, the vertex set can be used to initialize $M^k$. In this case, any candidate solution $\bm{x}^c$ is feasible on all vertices on $U$, so it is guaranteed to be ARO. Consequently, it is not needed to solve subproblem $Q(\bm{x}^c)$ in each iteration, and \Cref{alg: CCG-2} can be simplified. Note that this special case is the same as the case discussed in \Cref{sec: PARO-convex-hull}, but with $\bm{d} \neq \bm{0}$.

\subsection{Hints for solving the bilinear optimization problems}
\Cref{alg: CCG-2} requires solving both problems $P_2$ and $Q$ in each iteration; $Q$ is also solved in each iteration of \Cref{alg: CCG}. Unfortunately, both are intractable in general. The reason is that for both problems the feasible region of the inner optimization problem depend on the variables of the outer optimization problem. For $Q$, dualizing the inner minimization problem results in the following bilinear problem:
\begin{subequations} \label{eq: Q-dual}
\begin{align}
\max_{\bm{z} \in U,~\bm{\lambda}\leq \bm{0}} &~ \bm{c}(\bm{z})^{\top}\bm{x} + \bm{\lambda}^{\top}\big(\bm{r}(\bm{z}) - \bm{A}(\bm{z})\bm{x}\big) \\
\text{s.t.}~~~&~ \bm{B}^{\top}\bm{\lambda} = \bm{d}.
\end{align}
\end{subequations}

Similarly, for $P_2$ dualizing the inner maximization problem also results in a bilinear optimization problem:
\begin{subequations} \label{eq: P2-dual}
\begin{align}
\min_{\substack{\bm{z}^c,\bm{x}^c,\bm{y}^c, \bm{\lambda} \\ \{\bm{y}^l\}_{l=1}^{|M^k|}}} ~&~  \big(\bm{c}(\bm{z}^c)^{\top}\bm{x}^c + \bm{d}^{\top}\bm{y}^c\big) - \Big(\bm{c}(\bm{z}^c)^{\top}\bm{x}^k + \bm{\lambda}^{\top} (\bm{r}(\bm{z}^c) - \bm{A}(\bm{z}^c)\bm{x}^k)\Big), \\
\text{s.t.}~&~ \bm{A}(\bm{z}^c) \bm{x}^c + \bm{B} \bm{y}^c \leq \bm{r}(\bm{z}^c), \\
		   ~&~ \bm{A}(\bm{z}^c) \bm{x}^c + \bm{B} \bm{y}^l \leq \bm{r}(\bm{z}^l),~~\forall l=1,\dotsc,|M^k|, \\
	   		   ~&~ \bm{c}(\bm{z}^c)^{\top} \bm{x}^c + \bm{d}^{\top} \bm{y}^l \leq v(\bm{x}^k,\bm{z}^l),~~\forall l=1,\dotsc,|M^k|, \\
		   ~&~ \bm{z}^c \in U, \\
		   ~&~ \bm{B}^{\top}\bm{\lambda} = \bm{d},~~\bm{\lambda} \leq \bm{0}. 
\end{align}
\end{subequations}
For both \eqref{eq: Q-dual} and \eqref{eq: P2-dual}, we propose to use a simple alternating direction heuristic, also known as mountain climbing, which guarantees a local optimum \citep{Konno76}.

We describe this approach in more detail for \eqref{eq: P2-dual}. For some initial $\bm{z}^c$ one can determine the optimal $\{\bm{x}^c,\bm{y}^c,\bm{y}^1,\dotsc,\bm{y}^{|M^k|}\}$ by solving an LP. Subsequently, we alternate between optimizing for $\bm{z}^c$ and $\{\bm{\lambda},\bm{x}^c,\bm{y}^c,\bm{y}^1,\dotsc,\bm{y}^{|M^k|}\}$ while keeping the other set of variables at their current value. For either set of variables, the problem is an LP. This is continued until two consecutive LP problems yield the same objective value.

The solution quality of the mountain climbing procedure depends on the starting value for $\bm{z}^c$. One option is to simply pick the nominal scenario, if it is defined. A better starting solution can be obtained by plugging in an LDR for $\bm{y}^k$ in $P_2$, i.e., by solving
\begin{subequations}
\begin{align}
\max_{\bm{w},\bm{W}}~\min_{\substack{\bm{z}^c,\bm{x}^c,\bm{y}^c, \\  \{\bm{y}^l\}_{l=1}^{|M^k|}}} ~&~ (\bm{c}(\bm{z}^c)^{\top}\bm{x}^c + \bm{d}^{\top}\bm{y}^c) - (\bm{c}(\bm{z}^c)^{\top}\bm{x}^k + \bm{d}^{\top}(\bm{w} + \bm{W}\bm{z}^c)) \\
\text{s.t.}~&~ \bm{A}(\bm{z}^c) \bm{x}^k + \bm{B} (\bm{w} + \bm{W}\bm{z}^c) \leq \bm{r}(\bm{z}^c),
\end{align}
\end{subequations}
additionally subject to \eqref{eq: P2-2}-\eqref{eq: P2-5}. This is a static linear robust optimization problem.

This mountain climbing procedure can also be used to find a local optimum to \eqref{eq: Q-dual}. A starting solution for $\bm{z}$ can be found by plugging in an LDR for $\bm{y}$ in \eqref{eq: Q}, and subsequently one can alternate between solving for $\bm{z}$ and $\bm{\lambda}$ in \eqref{eq: Q-dual}.

By using a heuristic approach to solving \eqref{eq: P2}, it possible that at a certain iteration $k$ in \Cref{alg: CCG-2} we obtain an estimate $\hat{p}_k \geq 0$, while the true $p_k<0$, so the algorithm terminates without finding a PARO solution. Nevertheless, also solutions obtained this way that are not proven to be PARO can improve upon the original ARO solution. Similarly, if \eqref{eq: Q-dual} is not solved to optimality, we might falsely conclude that a solution is ARO. In the numerical experiments this risk is reduced by using multiple starting points for $\bm{z}^c$.

There are many different approaches to (approximately) solving problems \eqref{eq: Q-dual} and \eqref{eq: P2-dual}, or equivalently, problems $Q$ and $P_2$; we conclude this section by providing some examples and references. In case of RHS uncertainty, the following approach gives an exact solution to \eqref{eq: P2}. The inner minimization problem is an LP for which we can write down the optimality conditions. Subsequently, the complementary slackness conditions can be linearized using big-M constraints and auxiliary binary variables. This results in an exact reformulation to a mixed binary convex reformulation (mixed binary linear if $U$ is polyhedral). This reformulation was previously used in the C\&CG method of \citet{Zeng13} for ARO problems with a polyhedral uncertainty set, and to solve bilinear optimization problems with a disjoint uncertainty set \citep{Zhen22}. One can verify that in case of RHS uncertainty, \eqref{eq: P2-dual} indeed has disjoint polyhedral feasible regions for $\lambda$ and $\{\bm{z}^c,\bm{x}^c,\bm{y}^c,\bm{y}^1,\dotsc,\bm{y}^{|M^k|}\}$.

Another possible approach is the Reformulation-Perspectification-Technique proposed by \citet{Zhen21}. The advantage of this approach is that one gets both an upper and lower bound for the optimal objective value and this method can be easily extended to a Branch \& Bound framework to find the global optimal solution.

For alternative approaches to solving bilinear optimization problems, we refer to \citet{Konno76}, \citet{Nahapetyan09} and \citet{Zhen22}.

\section{Numerical Experiments} \label{sec: num-exp}
To demonstrate the value of PARO solutions in practice, we focus on an example problem in which (adaptive) robust optimization has been successfully applied: a facility location problem. 

\subsection{Problem description}
Consider a strategic decision-making problem where a number of facilities are to be opened, in order to satisfy the demand of a number of customers. The goal is to choose the locations for opening a facility such that the cost for opening the facilities plus the transportation cost for satisfying demand is minimized. We consider this problem in a two-stage setting with uncertain demand. Thus, facility opening decisions need to be made in Stage 1, before Stage-2 demand is known. 

Suppose there are $n$ locations where a facility can be opened, and $m$ demand locations. Let $\bm{x} \in \{0,1\}^n$ be a binary Stage-1 decision variable denoting the facility opening decisions. Opening facility costs $f_i$ and yields a capacity $s_i$, $i=1,\dotsc,n$. Let $\bm{y} \in \mathcal{R}^{m, m\times n}$ be the Stage-2 decision variable denoting transport from facility $i$ to demand location $j$; let $c_{ij}$ denote the associated costs, $i=1,\dotsc,n$, $j=1,\dotsc,m$. Let $z_j$ denote the uncertain demand in location $j$. The two-stage facility location model reads
\begin{subequations}\label{eq: FL}
\begin{align} 
\min_{\bm{x},\bm{y}(\cdot)} ~&~ \max_{\bm{z} \in U}~ \sum_{i=1}^n \sum_{j =1}^m c_{ij} y_{ij}(\bm{z}) + \sum_{i=1}^n f_i x_i,  \\
\text{s.t.}~&~ \sum_{i =1}^n y_{ij}(\bm{z}) \geq z_j,~~\forall \bm{z} \in U,~\forall j=1,\dotsc,m, \\
~&~ \sum_{j=1}^m y_{ij}(\bm{z}) \leq s_i x_i,~~\forall \bm{z} \in U,~\forall i =1,\dotsc,m, \\
~&~y_{ij}(\bm{z}) \geq 0,~\forall \bm{z} \in U,~~\forall i=1,\dotsc,n,~j=1,\dotsc,m, \\
~&~\bm{x} \in \{0,1\}^n,
\end{align}
\end{subequations}
with uncertainty set
\begin{align*}
U = \{\bm{z} : \sum_{j =1}^m z_j \leq \Gamma,~ l_j \leq z_j \leq u_j,~\forall j=1,\dotsc,m\}.
\end{align*}

\subsection{Setup}
Model \eqref{eq: FL} can be written in form \eqref{eq: P} with right-hand side uncertainty. In order to find an (approximate) PARO solution $\bm{x}_{\text{PARO}}$, \Cref{alg: CCG-2} is used, including C\&CG \Cref{alg: CCG} to find an ARO solution. In order to improve the likelihood of finding an ARO solution, problem \eqref{eq: Q-dual} is solved (with the mountain climbing approach) using multiple starting points for $\bm{z}$: the LDR-based starting point and 5 randomly sampled starting points.

For comparison purposes, we also compute a PRO solution to \eqref{eq: P} using the methodology of \citet[Theorem 1]{Iancu14}, which we repeat for convenience. Specifically, we plug in LDR $\bm{y}(\bm{z}) = \bm{w}+\bm{W}\bm{z}$, and obtain solution $(\bm{x}_1,\bm{w}_1,\bm{W}_1)$. Subsequently, we optimize for the nominal scenario $\bm{\bar{z}}$ whilst ensuring that performance in other scenarios does not deteriorate, and feasibility is maintained:
\begin{subequations}\label{eq: Thm1-PRO}
\begin{align}
\min_{\substack{\bm{x},\bm{w},\bm{W}\\\bm{x_2},\bm{w_2},\bm{W_2}}} ~&~ \bm{c}(\bm{\bar{z}})^{\top} \bm{x_2} + \bm{d}^{\top}(\bm{w_2} + \bm{W_2}\bm{\bar{z}}), \\
\text{s.t.} ~&~ \bm{c}(\bm{z})^{\top} \bm{x_2} + \bm{d}^{\top}(\bm{w_2} + \bm{W_2}\bm{z}) \leq 0, ~~\forall \bm{z} \in U, \label{eq: Thm1-PRO-2} \\
~&~ \bm{A}(\bm{z})\bm{x} + \bm{B}(\bm{w} + \bm{W}\bm{z}) \leq \bm{r}(\bm{z}),~~\forall \bm{z} \in U, \label{eq: Thm1-PRO-3}\\
~&~\bm{x} = \bm{x_1}+\bm{x_2}, \bm{w} = \bm{w_1}+\bm{w_2}, \bm{W} = \bm{W_1}+\bm{W_2}.  \label{eq: Thm1-PRO-4}
\end{align}
\end{subequations}
Constraint \eqref{eq: Thm1-PRO-4} states that the new solution equals the original solution (variables with subscript 1) plus an adaptation (variables with subscript 2). Constraint \eqref{eq: Thm1-PRO-2} ensures that the adaptation does not deteriorate performance in any scenario, and the objective is to optimize performance in scenario $\bm{\bar{z}}$. According to \citet[Theorem 1]{Iancu14}, the optimal solution for $(\bm{x},\bm{w},\bm{W})$ is PRO to \eqref{eq: P}. Let $\bm{\bm{x}}_{\text{PRO}}$ denote the optimal solution for $\bm{x}$.

We compare the performance of the Stage-1 (here-and-now) solutions $\bm{x}_{\text{PARO}}$, $\bm{x}_{\text{ARO}}$ and $\bm{x}_{\text{PRO}}$. For solutions $\bm{x}_{\text{PARO}}$ and $\bm{x}_{\text{ARO}}$ we use the optimal recourse decision. For $\bm{x}_{\text{PRO}}$ we report the results for two decision rules: (i) the optimal recourse decision, (ii) the LDR.  We refer to the four objective values as \textit{PARO}, \textit{ARO}, \textit{PRO} and \textit{PRO(LDR)}. We compute the relative improvement (in $\%$) of \textit{PARO} over the other three objective values for three different cases:
\begin{enumerate}[leftmargin=* ,labelindent=4.4em]
\item[\textbf{Nominal}:] Relative improvement in nominal scenario $\bm{\bar{z}}$.
\item[\textbf{Average}:] Average relative improvement over $10$ uniform randomly sampled scenarios. 
\item[\textbf{Maximum}:] Relative improvement in the scenario with the maximum performance difference between $\bm{x}_{\text{ARO}}$ and $\bm{x}_{\text{PARO}}$. This scenario, which we denote $\bm{z}^{\ast}$, is found by solving \eqref{eq: P2} with fixed $\bm{x}^k = \bm{x}_{\text{ARO}}$ and $\bm{x}^c = \bm{x}_{\text{PARO}}$.
\end{enumerate}
All optimization problems are solved using Gurobi 9.0 \citep{Gurobi20} with the dual simplex algorithm selected. We note that the influence of different solvers may also be investigated, but this is beyond the scope of this paper. All computations are performed on a Intel-Core i7-8565U PC with 16GB RAM, using all 8 threads.

During our numerical studies we found examples where \Cref{alg: CCG-2} was not able to improve upon the initial Stage-1 solution $\bm{x}_{\text{ARO}}$. This could occur if the initial $\bm{x}_{\text{ARO}}$ happens to be PARO. Or, it could occur if there is a unique ARO solution - after all, not every ARO instance has multiple worst-case optimal Stage-1 solutions. The latter has been reported before in literature. \Citet{DeRuiter16} show that the multi-stage production-inventory model of \citet{BenTal04} has unique here-and-now decisions in almost all time periods, if LDRs are used. In that example, the reported multiplicity of solutions is mainly due to non-PRO decision rule coefficients. We find that multiplicity of Stage-1 solutions appears in particular when problem data is integer.

\subsection{Data} \label{sec: data}
We consider 250 instances with $m=20$ and $n=40$. Facility capacity $s_i$ is set at 15 for each $i$. Other parameters are independently drawn from a discrete uniform distribution. Construction costs $\bm{f} \in \mathbb{R}^n$ are drawn between 4 and 22. Entries of transportation cost matrix $\bm{C} \in \mathbb{R}^{n\times m}$ are drawn between $2$ and $12$.

We set lower and upper bound $l_j=8$ and $u_j=12$ for each demand location $j=1,\dotsc,m$. Maximum total demand is set at $\Gamma = 200$. The nominal demand scenario is $\bar{z}_j = 10$ for all $j$. Note that $\bm{\bar{z}} \in \text{ri}(U)$.

\Cref{app: num-exp-small} provides additional results for smaller instances, where the vertices of the uncertainty set $U$ can be enumerated, thus simplifying \Cref{alg: CCG-2} (see \Cref{sec: CCG}).

\subsection{Results} 
Computing an (approximate) ARO solution using the C\&CG method requires at most $6$ iterations, together taking 251 seconds, on average. Subsequently, \Cref{alg: CCG-2} performs at most $6$ iterations to find an (approximate) PARO solution, taking in total 78 seconds, on average.

For the worst-case scenario, \textit{ARO} does not necessarily find the worst-case scenario and may thus underestimate the true worst-case cost. On the other hand, \textit{PRO} uses an LDR and may find an objective value that is higher than the worst-case optimum. On average, the reported worst-case objective value for \textit{PRO} is at most $1.72\%$ higher than that of \textit{ARO} (median $0.75\%$). In $31\%$ of instances the stage-1 solution of $\bm{x}^{\text{PARO}}$ differs from $\bm{x}^{\text{ARO}}$. \Cref{table: FL-large-x} reports the median and maximum difference in $\ell_1$-norm for these instances, representing the number of different facilities that are opened. For example, an $\ell_1$-norm of 2 indicates that one solution opened facility $i$ and another solution opened facility $j$, or one solution opened both facilities $i$ and $j$ and the other solution opened neither. The total number of considered facility locations is $n=40$, so the differences in Stage-1 facility openings are meaningful.

\begin{table}[htb]
\centering
\begin{tabular}{c c c c}\toprule
 & $\|\bm{x}_{\text{PARO}} - \bm{x}_{\text{ARO}} \|_1$ & $\|\bm{x}_{\text{PARO}} - \bm{x}_{\text{PRO}} \|_1$ & $\|\bm{x}_{\text{ARO}} - \bm{x}_{\text{PRO}} \|_1$ \\ \midrule
 median & 4 & 3 & 4 \\
 max & 10  & 10 &  9 \\ \bottomrule
\end{tabular}
\caption{\small Total differences in Stage-1 facility openings. \label{table: FL-large-x}}
\end{table}

\begin{figure}[htb]
\hspace*{-2cm}
\begin{subfigure}{0.5\textwidth}
  \includegraphics[scale=0.65]{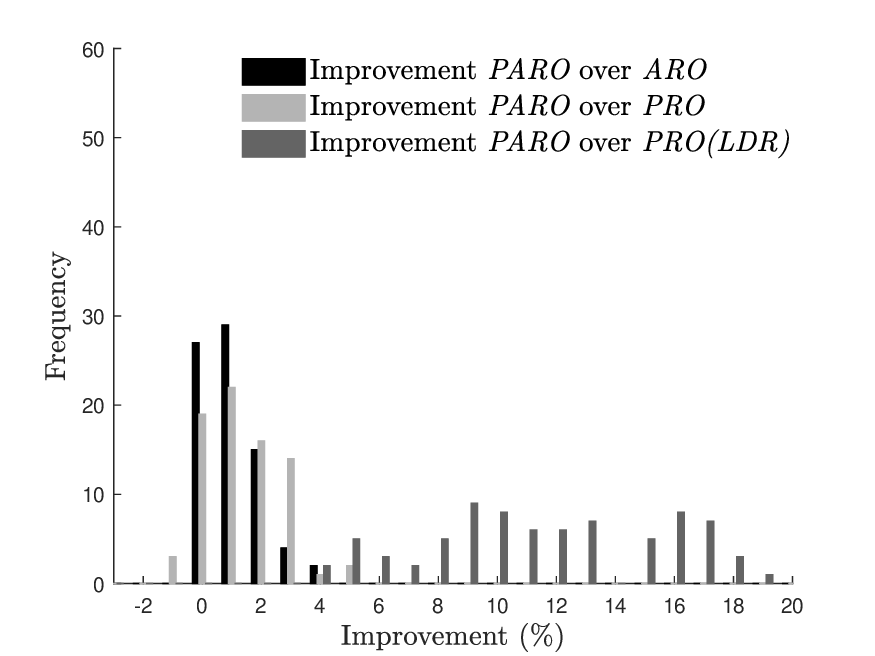}
  \caption{\small Scenario $\bm{z^{\ast}}$ \label{fig: histograms-FL-large-max}}
\end{subfigure}
\begin{subfigure}{0.5\textwidth}
  \includegraphics[scale=0.65]{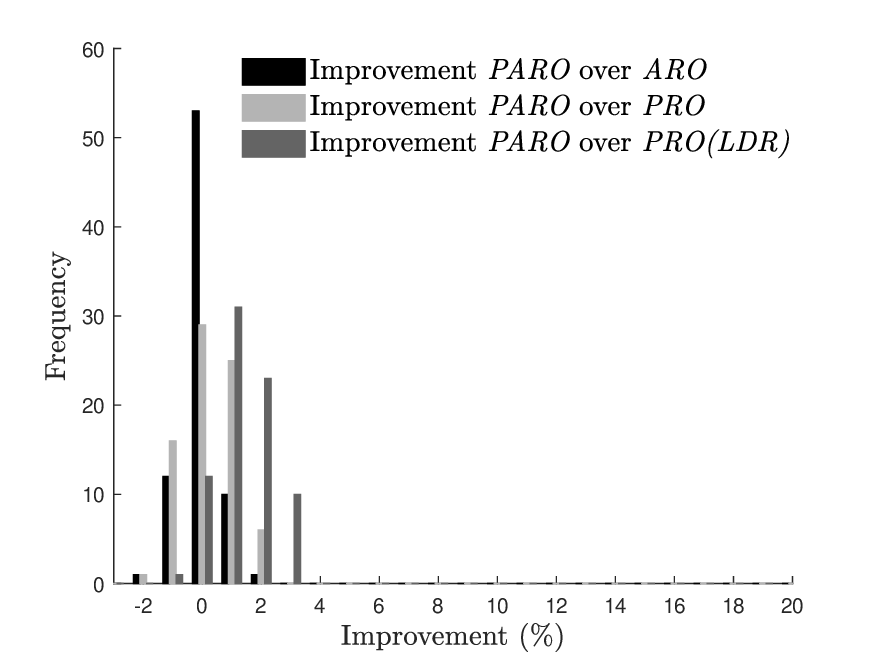}
  \caption{\small Scenario $\bm{\bar{z}}$  \label{fig: histograms-FL-large-nom}}
\end{subfigure}
\begin{subfigure}{\textwidth}
\centering
  \includegraphics[scale=0.65]{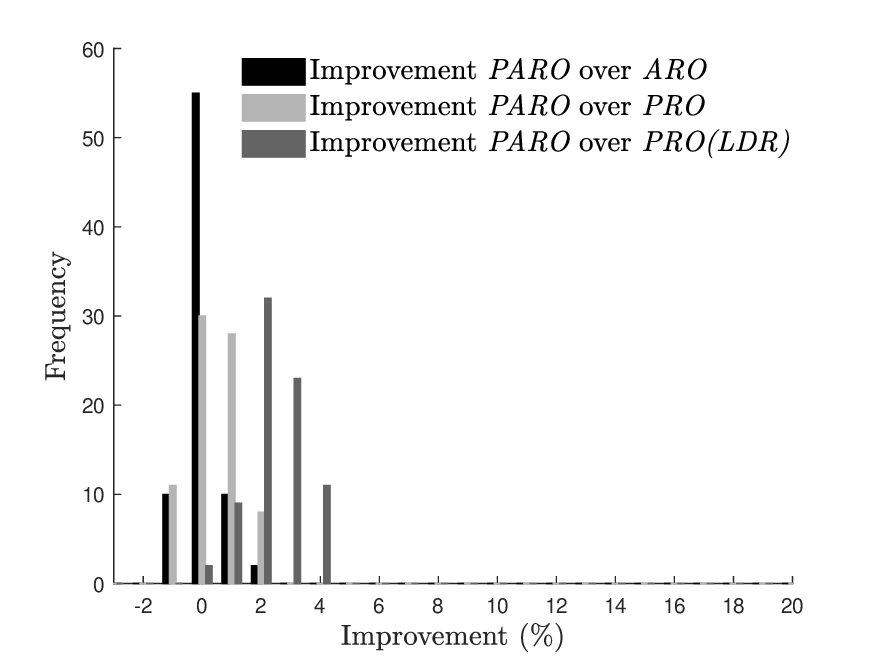}
  \caption{\small 10 Random scenarios \label{fig: histograms-FL-large-mean}}
\end{subfigure}
\caption{\small Histograms with relative improvement of PARO solution over alternative solutions. \label{fig: histograms-FL-large}}
\end{figure}

\Cref{fig: histograms-FL-large} shows the histograms of the relative objective value improvement of \textit{PARO} over \textit{ARO}, \textit{PRO} and \textit{PRO(LDR)} for the $31\%$ of instances with different $\bm{x}^{\text{PARO}}$ and $\bm{x}^{\text{ARO}}$. \Cref{fig: histograms-FL-large-max} show the improvement for maximum difference scenario $\bm{z^{\ast}}$, \Cref{fig: histograms-FL-large-nom} show the improvement for nominal scenario $\bm{\bar{z}}$ and \Cref{fig: histograms-FL-large-mean} show the improvement for 10 random scenarios in the uncertainty set. \Cref{table: percentages-FL_large} details the minimum, median and maximum relative improvement.

\begin{table}[htb!]
\centering
\begin{adjustbox}{width=\columnwidth,center}
\begin{tabular}{c c c c c}
\toprule 
& & \multicolumn{3}{c}{Relative improvement ($\%$)} \\
&  & \textit{PARO} over \textit{ARO} & \textit{PARO} over \textit{PRO} & \textit{PARO} over \textit{PRO(LDR)} \\\midrule
\multirow{3}{*}{Scenario $\bm{z^{\ast}}$}&  minimum & 	0.21 & -0.40 & 4.22 \\
& median  & 1.52 & 1.84 & 11.5 \\
& maximum & 4.51 & 5.45 & 19.1 \\[0.5em]
\multirow{3}{*}{Scenario $\bar{\bm{z}}$}&  minimum & 	-1.14& -1.03 & -0.20 \\
& median  & 0.54 & 0.73 & 1.83 \\
& maximum & 2.15 & 2.36 & 3.90 \\[0.5em] 
\multirow{3}{*}{\shortstack{10 Random \\ scenarios}}&  minimum & 	-0.78 & -0.90 & 0.67 \\
& median  & 0.57 & 0.82 & 2.92 \\
& maximum & 2.03 & 2.83 & 4.53 \\ \bottomrule
\end{tabular}
\end{adjustbox}
\caption{\small Relative improvement of PARO solution over alternative solutions. \label{table: percentages-FL_large}}
\end{table}

The magnitude of differences is larger for scenario $\bm{z}^{\ast}$ than for the other two measures. In all cases the maximum relative improvement is substantial, but the median relative improvement is only minor in most cases. However, if the Stage-1 solution represents a decision that is to be implemented in practice, even the possibility to get an improvement of a few percentage points warrants the extra effort to obtain an (approximate) PARO solution. We note that for ARO we use the first found ARO solution $\bm{x}^{\text{ARO}}$; it is possible that there exists yet another ARO solution, for which the improvement percentages of PARO over ARO are larger than those reported in \Cref{table: percentages-FL_large} and \Cref{fig: histograms-FL-large}.

The relative improvement of PARO over PRO smaller than that of PARO over PRO(LDR). Thus, reported differences are partially due to the different Stage-1 decision, and partially due to the Stage-2 decision rule. In some scenarios the relative improvement is negative (i.e., PARO has a worse objective value), although \Cref{fig: histograms-FL-large} shows that for the large majority of instances the relative improvement is positive. For many instances the relative difference between PARO and ARO is zero, i.e., a different Stage-1 solution $\bm{x}$ does not always translate to a different performance on the three reported measures. Nevertheless, also compared to ARO, the maximum relative improvement of PARO can be substantial. Lastly, the results indicate a larger spread in objective value for PRO and PRO(LDR) than for ARO.

\section{Conclusion}\label{sec: concluding-remarks}
In this paper, we dealt with Pareto efficiency in two-stage adaptive robust optimization problems. Similar to static robust optimization, the large majority of solution techniques focus only on worst-case optimality, and may yield solutions that are not Pareto efficient. To alleviate this, the concept of Pareto Adaptive Robustly Optimal (PARO) solutions has been introduced in the literature, and is formalized in a general framework in this paper.\footnote{For ARO problems, every non-PARO solution is dominated by a PARO solution, even if the former is Pareto Robustly Optimal (PRO), as defined by \citet{Iancu14}.}

Using FME as the predominant technique, we have analyzed the relation between PRO and PARO and investigated optimality of various decision rule structures in both worst-case and non-worst-case scenarios. We have shown the existence of PARO here-and-now decisions and shown that there exists a PWL decision rule that is PARO.

Moreover, we have provided several practical approaches to generate or approximate PARO solutions. Numerical experiments on a facility location example demonstrate that PARO solutions can significantly improve performance in non-worst-case scenarios over ARO and PRO solutions. 

A potential direction for future research would be to further investigate constructive approaches to find or approximate PARO solutions. In particular, it would be valuable to have tractable algorithms for larger instances and/or more general classes of problems than the ones that can be tackled using the currently presented approaches.

\bibliographystyle{apalike}
\small
\bibliography{References} 

\appendix
\renewcommand{\theequation}{\thesection.\arabic{equation}}

\setcounter{equation}{0}
\sectionfont{\large}
\subsectionfont{\normalsize}

\section{Optimality of Decision Rule Structures via an FME Lens} \label{app: decision-rules}
\small
As in \citet{Zhen18}, we use FME as a proof technique for ARO, and we analyze various decision rule structures. Moreover, the results in this section show that FME not only provides more general results, but also leads to more concise (and perhaps more intuitive) proofs to known results on optimal decision structures. \citet{Zhen18} also use FME as a lens to derive certain optimality properties of decision rule structures. However, the results obtained in this section are either new or more general. 

The FME lens enables us to prove that an ARF decision rule with a particular structure exists for \emph{every} ARF $\bm{x}$, instead of solely proving it is optimal for an ARO $\bm{x}$. These results are crucial for one of our main results in \Cref{sec: constructive}. 

\subsection{Eliminating adaptive variables using Fourier-Motzkin Elimination}
FME \citep{Fourier27,Motzkin36} is an algorithm for solving systems of linear inequalities. We refer to \citet{Bertsimas97} for an introduction to FME in linear optimization. Its usefulness in ARO is due to the fact that it can be used to eliminate adaptive variables, as proposed by \citet{Zhen18}. FME leads to an exponential increase in number of constraints. \citet{Zhen18} introduce a redundant constraint identification scheme, which helps to reduce the number of redundant constraints, although the number of constraints remains exponential. \citet{Zhen18} also propose to use FME to eliminate only part of the variables and using LDRs for remaining adaptive variables. Next to this, they use FME to prove (worst-case) optimality of PWL decision rules. Furthermore, they consider optimal decision rules for the adaptive variable in the dual problem: they prove (worst-case) optimality of LDRs in case of simplex uncertainty and (two-)piecewise linear decision rules in case of box uncertainty. \Citet{Zhen18b} use a combination of FME and ARO techniques to compute the maximum volume inscribed ellipsoid of a polytopic projection. The following example illustrates the use of FME to eliminate an adaptive variable.

\begin{example} \label{ex: toy-problem-RT-FME}
\small 
We use FME to eliminate adaptive variable $y$ from \eqref{eq: RT-model} in \Cref{ex: toy-problem-RT}. We move the uncertain objective to the constraints using an epigraph variable $t\in \mathbb{R}$, and rewrite the constraints to obtain:
\begin{subequations} \label{eq: RT-model-FME}
\begin{align}
\min_{x,t,y(d_1,d_2)} ~&~ t, \label{eq: RT-model-FME-1} \\
\text{s.t.}~&~ 20 \leq x \leq 40, \label{eq: RT-model-FME-2}\\
	~&~ \hspace*{1.15cm} y(d_1,d_2) \leq t/\delta - x, \hspace*{0.71cm} \forall  (d_1,d_2) \in U, \label{eq: RT-model-FME-3}\\
			~&~ d_1 - x  \leq y(d_1,d_2), \hspace*{2cm} \forall  (d_1,d_2) \in U, \\
		   ~&~ d_2 - x  \leq y(d_1,d_2), \hspace*{2cm} \forall  (d_1,d_2) \in U, \\
		   ~&~ \hspace*{0.51cm} 20 \leq y(d_1,d_2) \leq 40, \hspace*{1.34cm} \forall  (d_1,d_2) \in U. \label{eq: RT-model-FME-6}   
\end{align}
\end{subequations}
For fixed $(d_1,d_2)$, Constraints~\eqref{eq: RT-model-FME-3}-\eqref{eq: RT-model-FME-6} specify lower and/or upper bounds on $y$. By combining each pair of lower and upper bounds on $y$ into a new constraint, we find the following problem in terms of $(x,t)$:
\begin{subequations} \label{eq: RT-model-xt}
\begin{align}
\min_{x,t} ~&~ t, \label{eq: RT-model-xt-1} \\
\text{s.t.}~&~ 20 \leq x \leq 40, \label{eq: RT-model-xt-2}\\
	~&~ d_1 \leq t/\delta, ~~\forall  (d_1,d_2) \in U, \label{eq: RT-model-xt-3}\\
		~&~ d_2 \leq t/\delta, ~~\forall  (d_1,d_2) \in U, \label{eq: RT-model-xt-4}\\
		~&~ 20 \leq t/\delta - x, ~~\forall  (d_1,d_2) \in U, \label{eq: RT-model-xt-5}\\
	~&~ d_1 - x \leq 40, ~~\forall  (d_1,d_2) \in U, \label{eq: RT-model-xt-6}\\
		~&~ d_2 - x \leq 40, ~~\forall  (d_1,d_2) \in U, \label{eq: RT-model-xt-7}
\end{align}
\end{subequations}
where we have removed the trivial new constraint $20 \leq 40$. Any solution $(x,t)$ sets the following bounds on $y$: 
\begin{align*}
\max\{d_1-x,d_2-x,20 \} \leq y(d_1,d_2) \leq \min\{ t/\delta - x, 40\},~~\forall (d_1,d_2) \in U,
\end{align*}
and any decision rule satisfying these inequalities is ARO to \eqref{eq: RT-model}. Thus, two-stage problem \eqref{eq: RT-model} has been reduced to static linear RO problem \eqref{eq: RT-model-xt}. Auxiliary variable $t$ can be eliminated, but this transforms \eqref{eq: RT-model-xt} to an RO problem with a PWL objective. 
\end{example}

We focus on applying FME as a proof technique. Through the ``lens'' of FME we first consider (worst-case) optimality of decision rule structures, and subsequenly consider Pareto optimality. In the remainder of the paper, if FME is applied, w.l.o.g. it is applied on the adaptive variables in the order $y_1(\bm{z}),\dotsc,y_{n_y}(\bm{z})$, i.e., according to their index. We first state some frequently used definitions. If FME is performed on $\mathcal{X}$ until all adaptive variables are eliminated, the feasible region can be written as
\begin{align*}
\mathcal{X}_{\text{FME}} = \{\bm{x} \in \mathbb{R}^{n_x}~|~\bm{G}(\bm{z}) \bm{x} \leq \bm{f}(\bm{z}), ~~\forall \bm{z} \in U\},
\end{align*} 
for some matrix $\bm{G}(\bm{z})$ and vector $\bm{f}(\bm{z})$ depending affinely on $\bm{z}$. \citet{Zhen18} show that $\mathcal{X} = \mathcal{X}_{\text{FME}}$. For the analysis of particular decision rule structures, it is crucial to keep track of the original constraints during the FME procedure. A frequently used technical result on this is provided in the Electronic Companion, namely \Cref{lemma: y-representation} in \Cref{app: lemma-y-representation}.

\subsection{Optimality of decision rule structures}
In this section, we consider several special cases of problem \eqref{eq: P} for which particular decision rule structures are known to be optimal. We use FME to prove generalizations of these results for linear two-stage ARO problems. In particular, using FME as a proof technique enables us to show that the particular decision rule structure is not only ARO (i.e., worst-case optimal), but is ARF for each ARF Stage-1 decision $\bm{x}$. These results are used in analyzing PARO in \Cref{sec: constructive}.

We consider the cases where uncertainty appears (i) constraintwise, (ii) in a hybrid structure (part constraintwise, part non-constraintwise), (iii) in a block structure, and we consider (iv) the case with a simplex uncertainty set and the case with only one uncertain parameter.

\subsubsection*{(i) Constraintwise uncertainty}
Constraintwise uncertainty is formally defined as follows.
\begin{definition} \label{def: constraintwise}
ARO problem \eqref{eq: P} has constraintwise uncertainty if there is a partition 
\begin{align*}
\bm{z} = (\bm{z}_{(0)},\bm{z}_{(1)},\dotsc,\bm{z}_{(m)}),
\end{align*}
such that $\bm{z}_{(0)},\dotsc,\bm{z}_{(m)}$ are disjoint, the objective depends only on $\bm{z}_{(0)}$ and constraint $i$ depends only on $\bm{z}_{(i)}$, $i=1,\dotsc,m$. Additionally, $U = \{(\bm{z}_{(0)},\dotsc,\bm{z}_{(m)})~|$ $\bm{z}_{(i)} \in U^i, i=0,\dotsc,m\}$, with $U^i \subseteq \mathbb{R}^{|\bm{z}_{(i)}|}$ for all $i=0,\dotsc,m$.
\hfill $\blacksquare$ \end{definition}
\citet{BenTal04} show that for constraintwise uncertainty\footnote{Note that constraintwise uncertainty in ARO differs from constraintwise uncertainty in static RO problems. Whereas in the latter uncertainty can always be treated `constraintwise' \citep[p11]{BenTal09}, in ARO this is not the case in general, due to the adaptive variables. Constraintwise uncertainty in ARO refer to problems that satisfy the construction of \Cref{def: constraintwise}.} the objective values of the static and adaptive problem are equal, i.e., there exists an optimal static decision rule. Using FME, a generalization of their result can be easily proved. We first provide an example.
\begin{example} \label{ex: constraintwise}
\small 
Consider the following ARO problem with constraintwise uncertainty:
\begin{align*}
\min_{x,\bm{y}(\cdot)}~&~ x, \\
\text{s.t.}~&~ x - y_2(\bm{z}) \leq - \frac{1}{2}z_1,~~\forall z_1 \in [0,1], \\
~&~ -x + y_1(\bm{z}) + y_2(\bm{z}) \leq \frac{1}{2}z_2 + \frac{1}{2}z_3 + 2,~~\forall (z_2,z_3) \in [0,1]^2, \\
~&~ 1 \leq y_1(\bm{z}),~~\forall \bm{z} \in U, \\
~&~ \frac{3}{2} \leq y_2(\bm{z}) \leq 2,~~\forall \bm{z} \in U 
\end{align*}
with $U = [0,1]^3$. Uncertain parameter $z_1$ occurs only in the first constraint and $(z_2,z_3)$ occur only in the second constraint. Using FME, we first eliminate $y_1(\bm{z})$ and subsequently eliminate $y_2(\bm{z})$:
\begin{align*}
1 \leq & y_1(\bm{z}) \leq - y_2(\bm{z}) +x + 2 + \frac{1}{2}z_2 + \frac{1}{2}z_3,~~\forall \bm{z} \in U,\\
\max\{\frac{3}{2}, x + \frac{1}{2}z_1 \} \leq & y_2(\bm{z}) \leq \min \{2, x + 1 + \frac{1}{2}z_2 + \frac{1}{2}z_3  \},~~\forall \bm{z} \in U.
\end{align*}
From the bounds on $y_2(\bm{z})$ four linear constraints can be derived, two of which depend on $x$ and are non-trivial:
\begin{align*}
x + \frac{1}{2}z_1 \leq 2,~~\forall \bm{z} \in U. \\
\frac{3}{2} \leq x + 1 + \frac{1}{2}z_2 + \frac{1}{2}z_3,~~\forall \bm{z} \in U.
\end{align*}
One can verify that the (unique) ARO solution is $x^{\ast} = \frac{1}{2}$. Additionally, note that the term $\frac{1}{2}z_2 + \frac{1}{2}z_3$ appears in both upper bounds with a positive sign. As this is the only term that depends on $(z_2,z_3)$, it can be replaced by its worst-case value $0$. Similarly, the term $-\frac{1}{2}z_1$ appears in the lower bound on $y_2(\bm{z})$ with a negative sign, and can be replaced by its worst-case value $-\frac{1}{2}$. This gives the following bounds on $y_1(\bm{z})$ and $y_2(\bm{z})$:
\begin{align*}
1 \leq & y_1(\bm{z}) \leq - y_2(\bm{z}) + \frac{5}{2},~~\forall \bm{z} \in U,\\
\max\{\frac{3}{2}, 1 \} \leq & y_2(\bm{z}) \leq \min \{2, \frac{3}{2} \},~~\forall \bm{z} \in U.
\end{align*}
For $y_2(\bm{z})$, the only feasible (and hence ARO) decision rule is $y_2(\bm{z}) = \frac{3}{2}$. This implies $y_1(\bm{z}) = 1$, and we find that for both adaptive variables the optimal decision rule is static.
\end{example}
According to \Cref{lemma: y-representation}, any term such as $\frac{1}{2}z_2 + \frac{1}{2}z_3$ in \Cref{ex: hybrid} appears in all upper bounds with a positive sign and all lower bounds with a negative sign, or vice versa. Hence, if this is the only term depending on $z_2$ and $z_3$, these uncertain parameters can be eliminated by replacing them with their worst-case value. The resulting bounds on adaptive variables are independent of uncertain parameters. 

Instead of directly providing a formal proof of the result for constraintwise uncertainty, it follows as a corollary from our analysis of hybrid uncertainty, which is considered next. 

\subsubsection*{(ii) Hybrid uncertainty}
Hybrid uncertainty is a generalization of constraintwise uncertainty, where part of the uncertain parameters appear constraintwise, and part does not appear constraintwise. This uncertainty structure has previously been considered in \citet{Marandi18}. Hybrid uncertainty is defined as follows.
\begin{definition}
ARO problem \eqref{eq: P} has hybrid uncertainty if there is a partition 
\begin{align*}
\bm{z} = (\bm{\hat{z}},\bm{z}_{(0)},\bm{z}_{(1)},\dotsc,\bm{z}_{(m)}),
\end{align*}
such that $\bm{\hat{z}},\bm{z}_{(0)},\dotsc,\bm{z}_{(m)}$ are disjoint, the objective depends only on $\bm{\hat{z}}$ and $\bm{z}_{(0)}$ and constraint $i$ depends only on $\bm{\hat{z}}$ and $\bm{z}_{(i)}$, $i=1,\dotsc,m$. Additionally, $U = \{(\bm{\hat{z}},\bm{z}_{(0)},\dotsc,\bm{z}_{(m)})~|~\bm{\hat{z}} \in \hat{U}, \bm{z}_{(i)} \in U^i, i=0,\dotsc,m\}$, with $\hat{U} \subseteq \mathbb{R}^{|\bm{\hat{z}}|}$ and $U^i \subseteq \mathbb{R}^{|\bm{z}_{(i)}|}$ for all $i=0,\dotsc,m$.
\hfill $\blacksquare$\end{definition}

In case of hybrid uncertainty, there exist ARO decision rules that do not depend on the constraintwise uncertain parameters. We illustrate this with a toy example.
\begin{example} \label{ex: hybrid}
\small 
We extend \Cref{ex: constraintwise} to a problem with hybrid uncertainty by introducing a non-constraintwise uncertain parameter $\hat{z}$:
\begin{subequations} \label{eq: ex-hybrid}
\begin{align}
\min_{x,\bm{y}(\cdot)}~&~ x, \\
\text{s.t.}~&~ x - y_2(\bm{z}) \leq -\hat{z} - \frac{1}{2}z_1,~~\forall (\hat{z},z_1) \in [0,1]^2,  \label{eq: ex-hybrid-2}\\
~&~ -x + y_1(\bm{z}) + y_2(\bm{z}) \leq \hat{z} + \frac{1}{2}z_2 + \frac{1}{2}z_3 + 2,~~\forall (\hat{z},z_2,z_3) \in [0,1]^3, \\
~&~ 1 \leq y_1(\bm{z}),~~\forall \bm{z} \in U,\\
~&~ \frac{3}{2} \leq y_2(\bm{z}) \leq 2,~~\forall \bm{z} \in U,  \label{eq: ex-hybrid-5}
\end{align}
\end{subequations}
with $U = [0,1]^4$. Uncertain parameter $\hat{z}$ occurs in both constraints, $z_1$ occurs only in the first constraint and $(z_2,z_3)$ occur only in the second constraint. Using FME, we again first eliminate $y_1(\bm{z})$ and subsequently eliminate $y_2(\bm{z})$:
\begin{align*}
1 \leq & y_1(\bm{z}) \leq \hat{z} - y_2(\bm{z}) + x + 2 + \frac{1}{2}z_2 + \frac{1}{2}z_3, ~~\forall \bm{z} \in U,\\
\max\{\frac{3}{2}, x + \hat{z} + \frac{1}{2}z_1 \} \leq & y_2(\bm{z}) \leq \min \{2, x+1+ \hat{z} + \frac{1}{2}z_2 + \frac{1}{2}z_3  \},~~\forall \bm{z} \in U.
\end{align*}
From the bounds on $y_2(\bm{z})$ again four linear constraints for $x$ can be derived. The new parameter $\hat{z}$ does not break robustness of solution $x^{\ast} = \frac{1}{2}$, so this is still the unique ARO solution. Similar to \Cref{ex: constraintwise}, we can replace both occurrences of the term $\frac{1}{2}z_2 + \frac{1}{2}z_3$ by its worst-case value $0$, and $-\frac{1}{2}z_1$ can be replaced by its worst-case value $-\frac{1}{2}$. This yields the following bounds on $y_1(\bm{z})$ and $y_2(\bm{z})$:
\begin{align*}
1 \leq & y_1(\bm{z}) \leq \hat{z} - y_2(\bm{z}) + \frac{5}{2}, ~~\forall \bm{z} \in U,\\
\max\{\frac{3}{2}, 1 + \hat{z} \} \leq & y_2(\bm{z}) \leq \min \{2, \frac{3}{2} + \hat{z} \},~~\forall \bm{z} \in U.
\end{align*}
For $y_2(\bm{z})$, the only feasible (and hence ARO) LDR is $y_2(\hat{z}) = \frac{3}{2} + \frac{1}{2}\hat{z}$. This implies $ 1 \leq y_1(\bm{z}) \leq 1 + \frac{1}{2}\hat{z}$, and any decision rule that satisfies these bounds is ARO. Note that both decision rules do not depend on the constraintwise uncertain parameters. One can also pick a PWL decision rule for $y_2(\bm{z})$, such as its lower or upper bound. Also in this case the decision rules for $y_1$ and $y_2$ do not depend on $z_1$, $z_2$ or $z_3$.
\end{example}

To formally prove our claim that there exist ARO decision rules that do not depend on the constraintwise uncertain parameters, we first need a result on feasibility.
\begin{lemma} \label{lemma: hybrid-ARF} 
Let $P_{\text{hybrid}}$ denote an ARO problem of form \eqref{eq: P} with hybrid uncertainty and let $\bm{x}^{\ast}$ be ARF to $P_{\text{hybrid}}$. Then, there exists a decision rule $\bm{y}^{\ast}(\cdot)$ that depends only on $\bm{\hat{z}}$ such that $(\bm{x}^{\ast},\bm{y}^{\ast}(\cdot))$ is ARF to $P_{\text{hybrid}}$.
\end{lemma}
\begin{proof}
See \Cref{app: proof-hybrid-ARF}. 
\end{proof}
The following result is an immediate consequence of \Cref{lemma: hybrid-ARF} for ARO decisions.
\begin{corollary}\label{cor: hybrid-ARO} 
Let $P_{\text{hybrid}}$ denote an ARO problem of form \eqref{eq: P} with hybrid uncertainty. For each $\bm{x}^{\ast}$ that is ARO to $P_{\text{hybrid}}$ there exists a decision rule $\bm{y}^{\ast}(\cdot)$ depending only on $\bm{\hat{z}}$ such that the pair $(\bm{x}^{\ast},\bm{y}^{\ast}(\cdot))$ is ARO to $P_{\text{hybrid}}$.
\end{corollary}
\begin{proof}
See \Cref{app: proof-hybrid-ARO}.
\end{proof}
In case of pure constraintwise uncertainty ($U^0 = \emptyset$) \Cref{lemma: hybrid-ARF} shows that for each ARF $\bm{x}$ there exists a static $\bm{y}$ such that $(\bm{x},\bm{y})$ is ARF. Additionally, \Cref{cor: hybrid-ARO} shows that for each ARO $\bm{x}^{\ast}$ there exists a static $\bm{y}^{\ast}$ such that $(\bm{x}^{\ast},\bm{y}^{\ast})$ is ARO. 

\citet{Marandi18} prove a similar result to \Cref{cor: hybrid-ARO} for non-linear problems. More precisely, they prove that for problems with hybrid uncertainty there exists an optimal decision rule that is a function of only the non-constraintwise uncertain parameters if the problem is convex in the decision variables, concave in uncertain parameters, has a convex compact uncertainty set and a convex compact feasible region for the adaptive variables.

\subsubsection*{(iii) Block uncertainty}
Suppose we can split the constraints into blocks, where each block has its own uncertain parameters and adaptive variables, and the uncertainty set is a Cartesian product of the block-wise uncertainty sets, then there exists an optimal decision rule for each adaptive variable that depends only on the uncertain parameters in its own block. 

The formal definition of block uncertainty is as follows. Recall that constraints are indexed $1,\dotsc,m$. Let index $0$ refer to the objective.
\begin{definition}
ARO problem \eqref{eq: P} has block uncertainty if there exist partitions $\bm{z} = (\bm{z}_{(1)},\dotsc,\bm{z}_{(V)})$, $\bm{y}(\cdot) = (\bm{y}_{(1)}(\cdot),\dotsc,\bm{y}_{(V)}(\cdot))$ and $\{0,\dotsc,m\} =$ \linebreak$\{K_{(1)},\dotsc,K_{(V)}\}$ such that
\begin{itemize}
\item $U = \{(\bm{z}_{(1)},\dotsc,\bm{z}_{(V)})~|~\bm{z}_{(v)} \in U^v, v=1,\dotsc,V\}$, with $U^v \subseteq \mathbb{R}^{|\bm{z}_{(v)}|}$ for all blocks $v=1,\dotsc,V$.
\item A constraint or objective with index in set $K_{(v)}$ is independent of uncertain parameters $\bm{z}_{(w)}$ and adaptive variables $\bm{y}_{(w)}$ if block $w\neq v$. \hfill $\blacksquare$
\end{itemize}
\end{definition}
We first provide an example to develop some intuition for block uncertainty.
\begin{example} \label{ex: block}
\small 
Consider again \Cref{ex: hybrid}. Add the following constraints to \eqref{eq: ex-hybrid}:
\begin{subequations} \label{eq: ex-block}
\begin{align}
y_3(\bm{z}) + x \leq -\frac{1}{2}z_4 + \frac{3}{2},~~\forall z_4 \in [0,1], \\
y_3(\bm{z}) + 2x \geq \frac{1}{2}z_5 + 1,~~\forall z_5 \in [0,1],
\end{align}
\end{subequations}
and let $U = [0,1]^6$ denote the new uncertainty set. Then the first block consists of constraints \eqref{eq: ex-hybrid-2}-\eqref{eq: ex-hybrid-5}, adaptive variables $y_1(\bm{z}), y_2(\bm{z})$ and uncertain parameters $z_0,\dotsc,z_3$. The second block consists of constraints \eqref{eq: ex-block}, adaptive variable $y_3(\bm{z})$ and uncertain parameters $z_4$ and $z_5$. One can verify that the unique ARO solution remains $x^{\ast} = \frac{1}{2}$. The following bounds on $y_3(\bm{z})$ are obtained:
\begin{align*}
\frac{1}{2}z_5 \leq y_3(\bm{z}) \leq 1 - \frac{1}{2}z_4,~~\forall \bm{z} \in U.
\end{align*}
One feasible (and hence ARO) decision rule is $y_3(z_4,z_5) = \frac{1}{2}(1+z_5-z_4)$. The decision rules for $y_1$ and $y_2$ remain unchanged. It follows that for each adaptive variable the optimal decision rule is a function of only the uncertain parameters in its own block.
\end{example}

In order to prove the claim that there exists an optimal decision rule for each adaptive variable that depends only on the uncertain parameters in its own block, we again first consider feasibility.
\begin{lemma} \label{lemma: block-ARF} 
Let $P_{\text{block}}$ denote an ARO problem of form \eqref{eq: P} with block uncertainty and let $\bm{x}$ be ARF to $P_{\text{block}}$. Then there exists a decision rule $\bm{y}(\cdot)$ with $\bm{y}_{(v)}(\cdot)$ depending only on $\bm{z}_{(v)}$, for all $v=1,\dotsc,V$, such that $(\bm{x},\bm{y}(\cdot))$ is ARF to $P_{\text{block}}$.
\end{lemma}
\begin{proof}
See \Cref{app: proof-block-ARF}.
\end{proof}
\begin{corollary}\label{cor: block-ARO} 
Let $P_{\text{block}}$ denote an ARO problem of form \eqref{eq: P} with block uncertainty. For each $\bm{x}$ that is ARO to $P_{\text{block}}$ there exists a decision rule $\bm{y}(\cdot)$ with $\bm{y}_{(v)}(\cdot)$ depending only on $\bm{z}_{(v)}$, for all $v=1,\dotsc,V$, such that the pair $(\bm{x},\bm{y}(\cdot))$ is ARO to $P_{\text{block}}$. \end{corollary}
\begin{proof}
Follows from \Cref{lemma: block-ARF} analogous to the proof of \Cref{cor: hybrid-ARO}.
\end{proof}

\subsubsection*{(iv) Simplex uncertainty or one uncertain parameter}
\citet{Bertsimas12} prove optimality of LDRs for right-hand side uncertainty and a simplex uncertainty set. \citet{Zhen18} generalize this to both left- and right-hand side uncertainty, their proof uses FME on the dual problem. We use FME on the primal problem, which leads to a more intuitive proof; the following example illustrates the main idea. We note that the case with one uncertain parameter is a special case of simplex uncertainty, so the results of this section also hold for that case.
\begin{example}
\small 
Consider the problem
\begin{align*}
\min~&~ x, \\
\text{s.t.}~&~ x - y_2 \leq -z_1 - \frac{1}{2}z_2 - \frac{1}{2},~~ \forall \bm{z} \in U, \\
~&~ -x + y_1 + y_2 \leq z_1 + z_3 + 2,~~ \forall \bm{z} \in U, \\
~&~ 0 \leq y_1(\bm{z}),  ~~ \forall \bm{z} \in U,\\
~&~\frac{3}{2} \leq y_2(\bm{z}) \leq 2, ~~ \forall \bm{z} \in U,
\end{align*}
with standard simplex uncertainty set $U = \{(z_1,z_2,z_3): z_1+z_2+z_3 \leq 1, z_1,z_2,z_3\geq 0\}$. Similar to \Cref{ex: hybrid}, we first eliminate $y_1(\bm{z})$ and then $y_2(\bm{z})$. This results in the following bounds on the adaptive variables:
\begin{subequations} \label{eq: simplex-bounds}
\begin{align}
0 \leq &y_1(\bm{z}) \leq z_1 + z_3 + 2 + x - y_2(\bm{z}),~~\forall \bm{z} \in U,\\
\max\{\frac{3}{2},\frac{1}{2} +x + z_1 + \frac{1}{2}z_2 \} \leq &y_2(\bm{z}) \leq \min \{2, z_1 + z_3 + 1 + x \},~~\forall \bm{z} \in U.
\end{align}
\end{subequations}
Equivalently, these bounds have to be satisfied for each point in $\text{ext}(U)$. One can verify that $x^{\ast} = \frac{1}{2}$ is an ARO solution. Plugging this in \eqref{eq: simplex-bounds}, we get the following bounds for each extreme point:
\begin{align} \label{eq: extreme-points-y2}
\begin{aligned}
(0,0,0):\hspace*{1cm}	& 0 \leq y_1 \leq \frac{5}{2} -y_2,~ & \frac{3}{2} \leq y_2 \leq \frac{3}{2},\\
(1,0,0):\hspace*{1cm}	& 0 \leq y_1 \leq \frac{7}{2} -y_2,~ & 2 \leq y_2 \leq 2,\\
(0,1,0):\hspace*{1cm}	& 0 \leq y_1 \leq \frac{5}{2} -y_2,~ & \frac{3}{2} \leq y_2 \leq \frac{3}{2}, \\
(0,0,1):\hspace*{1cm}	& 0 \leq y_1 \leq \frac{7}{2} -y_2,~ & \frac{3}{2} \leq y_2 \leq 2.
\end{aligned}
\end{align}
Because $U$ is a simplex, the four extreme points are affinely independent. Therefore, there is a unique LDR such that the upper bound on $y_2(\cdot)$ holds with equality for each extreme point. This is also the case for the lower bound, and any convex combination of both decision rules also satisfies the bounds for $y_2$ in \eqref{eq: extreme-points-y2}. The LDR corresponding with the upper bounds is $y_2(z_1,z_3) = \frac{1}{2}(3 + z_1+z_3)$, and plugging this in the bounds on $y_1$ yields a similar system as \eqref{eq: extreme-points-y2} for $y_1$. This guarantees existence of an LDR for $y_1$; for the upper bound we find $y_1(z_1,z_3) =  \frac{1}{2}(2+z_1 + z_3)$. Note that this does not generalize to uncertainty sets described by more than $L+1$ extreme points.
\end{example}
Similar to the cases for hybrid and block uncertainty, we first prove feasibility for each ARF $\bm{x}$, and subsequently prove optimality.
\begin{lemma} \label{lemma: simplex-LDR-ARF}
Let $P_{\text{simplex}}$ denote an ARO problem of form \eqref{eq: P} with a simplex uncertainty set, i.e., $U = \text{Conv}(\bm{z}^1,\dotsc,\bm{z}^{L+1})$, with $\bm{z}^{j} \in \mathbb{R}^{L}$ such that $\bm{z}^1,\dotsc,\bm{z}^{L+1}$ are affinely independent. Let $\bm{x}$ be ARF to $P_{\text{simplex}}$. Then there exists an LDR $\bm{y}(\cdot)$ such that $(\bm{x},\bm{y})$ is ARF to $P_{\text{simplex}}$.
\end{lemma}
\begin{proof}
See \Cref{app: proof-simplex-LDR-ARF}.
\end{proof}
Similar to \Cref{cor: hybrid-ARO}, we have the following result for ARO decisions.
\begin{corollary} \label{cor: simplex-LDR-ARO}
Let $P_{\text{simplex}}$ denote an ARO problem of form \eqref{eq: P} with a simplex uncertainty set, i.e., $U = \text{Conv}(\bm{z}^1,\dotsc,\bm{z}^{L+1})$, with $\bm{z}^{j} \in \mathbb{R}^{L}$ such that $\bm{z}^1,\dotsc,\bm{z}^{L+1}$ are affinely independent. For each $\bm{x}$ that is ARO to $P_{\text{simplex}}$ there exists an LDR $\bm{y}(\cdot)$ such that the pair $(\bm{x},\bm{y}(\cdot))$ is ARO to $P_{\text{simplex}}$. 
\end{corollary}
\begin{proof}
Follows from \Cref{lemma: simplex-LDR-ARF} analogous to the proof of \Cref{cor: hybrid-ARO}.
\end{proof}
Because the case with one uncertain parameter is a special case of simplex uncertainty, the results of \Cref{lemma: simplex-LDR-ARF} and \Cref{cor: simplex-LDR-ARO} also hold for that case.

The results on PARO in \Cref{sec: PARO-decision-rule-structure} make use of the fact that an ARF decision rule with a particular structure exists for \emph{every} ARF $\bm{x}$, i.e., \Cref{lemma: hybrid-ARF,lemma: block-ARF,lemma: simplex-LDR-ARF}.

\section{Numerical Experiments - Enumeration of Uncertainty Set Vertices} \label{app: num-exp-small}
\small
For small instances of the facility location problem of \Cref{sec: num-exp}, the vertices of the uncertainty set $U$ can be enumerated. Thus, we can obtain an ARO solution $\bm{x}_{\text{ARO}}$ by defining a separate recourse variable for each vertex of the uncertainty set. To find an (approximate) PARO solution, \Cref{alg: CCG-2} can be slightly simplified, see the end of \Cref{sec: CCG} for more details. We consider 1,000 instances with $m=8$ demand locations and $n=20$ possible facility locations. Maximum total demand is set at $\Gamma = 90$. All other settings are identical to \Cref{sec: data}.

Computing an ARO solution takes on average 10 seconds. Subsequently, \Cref{alg: CCG-2} performs 1 or 2 iterations (average 51 seconds) to find an (approximate) PARO solution. For the worst-case scenario, $\textit{PRO(LDR)}$ and $\textit{PRO}$ are both within $0.72\%$ of the optimum for all instances. In $28\%$ of the instances the Stage-1 solution $\bm{x}_{\text{PARO}}$ differs from $\bm{x}_{\text{ARO}}$ and/or $\bm{x}_{\text{PRO}}$. \Cref{table: FL-x} again reports the median and maximum difference in $\ell_1$-norm for these instances, representing the number of different facilities that are opened. The total number of considered facility locations is $n=20$, so the differences reported in \Cref{table: FL-x} are substantial.
\begin{table}[htb]
\centering
\begin{tabular}{c c c c}\toprule
 & $\|\bm{x}_{\text{PARO}} - \bm{x}_{\text{ARO}} \|_1$ & $\|\bm{x}_{\text{PARO}} - \bm{x}_{\text{PRO}} \|_1$ & $\|\bm{x}_{\text{ARO}} - \bm{x}_{\text{PRO}} \|_1$ \\ \midrule
 median & 0 & 1 & 1 \\
 max & 7  &  8 &  9 \\ \bottomrule
\end{tabular}
\caption{\small Total differences in Stage-1 facility openings (for the small instances). \label{table: FL-x}}
\end{table}

\begin{figure}[htb]
\hspace*{-2cm}
\begin{subfigure}{0.5\textwidth}
  \includegraphics[scale=0.65]{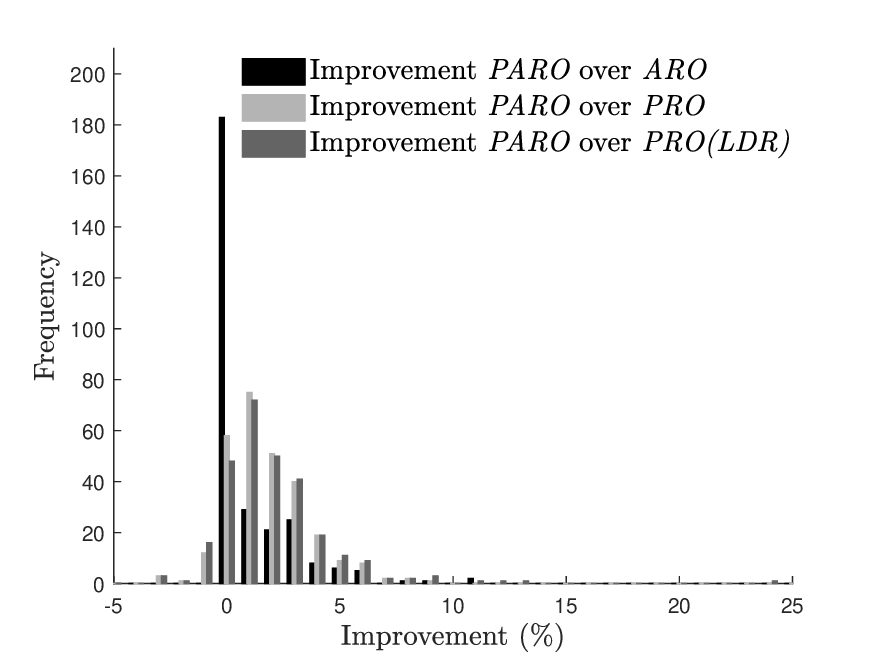}
  \caption{\small Scenario $\bm{z^{\ast}}$ \label{fig: histograms-FL-max}}
\end{subfigure}
\begin{subfigure}{0.5\textwidth}
  \includegraphics[scale=0.65]{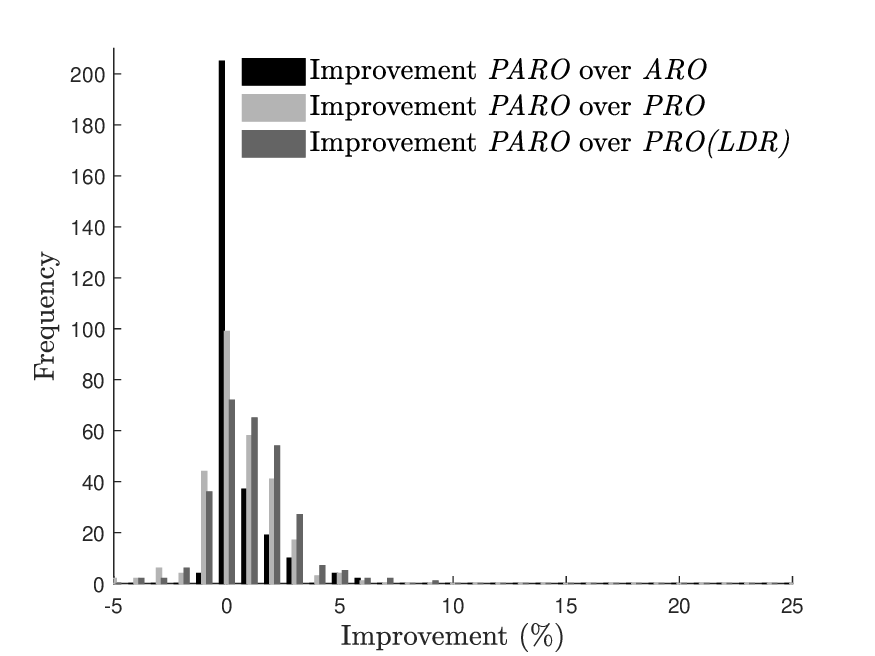}
  \caption{\small Scenario $\bm{\bar{z}}$  \label{fig: histograms-FL-nom}}
\end{subfigure}
\begin{subfigure}{\textwidth}
\centering
  \includegraphics[scale=0.65]{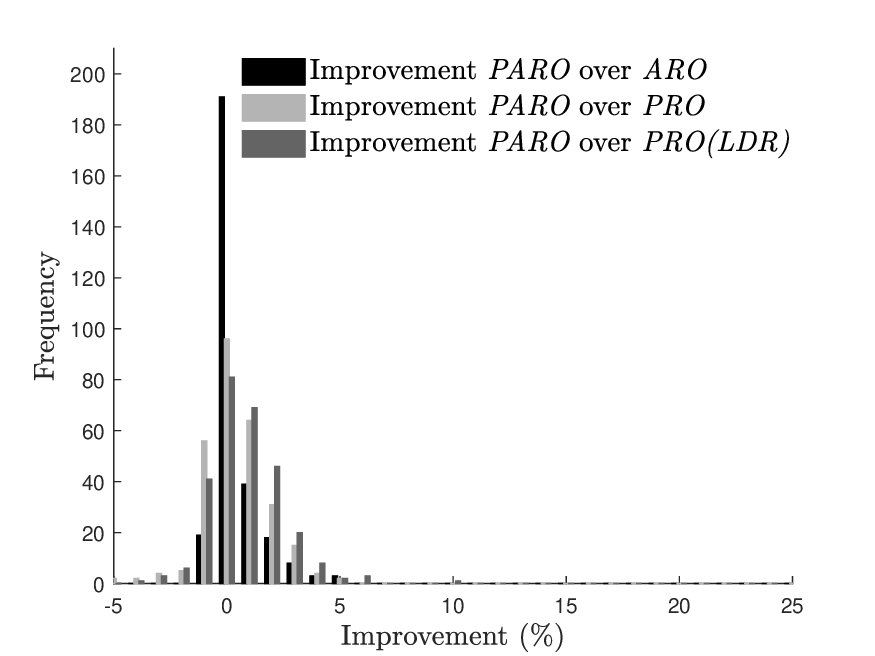}
  \caption{\small 10 Random scenarios \label{fig: histograms-FL-mean}}
\end{subfigure}
\caption{\small Histograms with relative improvement of PARO solution over alternative solutions (for the small instances). \label{fig: histograms-FL}}
\end{figure}
Analogous to \Cref{fig: histograms-FL-large}, \Cref{fig: histograms-FL} shows histograms of the relative objective value improvement of \textit{PARO} over \textit{ARO}, \textit{PRO} and \textit{PRO(LDR)} for $28\%$ of instances with different Stage-1 decisions, and \Cref{table: percentages-FL} details the minimum, median and maximum relative improvement.

\begin{table}[htb!]
\centering
\begin{adjustbox}{width=\columnwidth,center}
\begin{tabular}{c c c c c}
\toprule 
& & \multicolumn{3}{c}{Relative improvement ($\%$)} \\
&  & \textit{PARO} over \textit{ARO} & \textit{PARO} over \textit{PRO} & \textit{PARO} over \textit{PRO(LDR)} \\\midrule
\multirow{3}{*}{Scenario $\bm{z^{\ast}}$}&  minimum & 	0& -2.62 & -2.62 \\
& median  & 0 & 1.89 & 2.00 \\
& maximum & 11.2 & 9.48 & 24.2 \\[0.5em]
\multirow{3}{*}{Scenario $\bar{\bm{z}}$}&  minimum & 	-0.39& -4.89 &-3.51 \\
& median  & 0 & 0.77 & 1.35 \\
& maximum & 6.37 & 6.37 & 9.97 \\[0.5em] 
\multirow{3}{*}{\shortstack{10 Random \\ scenarios}}&  minimum & 	-0.58& -4.65 &-3.58 \\
& median  & 0 & 0.71 & 1.15 \\
& maximum & 5.48 & 5.35 & 10.9 \\ \bottomrule
\end{tabular}
\end{adjustbox}
\caption{\small Relative improvement of PARO solution over alternative solutions (for the small instances).\label{table: percentages-FL}}
\end{table}
 
The results are largely similar to the results of the larger instances of \Cref{sec: num-exp}, indicating that the findings are not limited to specific instances or instance sizes. A notable difference is that for these smaller instances the improvement of \textit{PARO} over \textit{PRO(LDR)} in scenario $\bm{z}^{\ast}$ is closer to the improvement of \textit{PARO} over \textit{PRO}. This indicates that using linear decision rules for the Stage-2 decisions is particularly suboptimal for the larger instances, whereas this suboptimality is less pronounced for the smaller instances.

\section{Technical Lemmas and Proofs} \label{app: proofs}
\small
\subsection{Bounds on eliminated adaptive variables} \label{app: lemma-y-representation}
\begin{lemma} \label{lemma: y-representation}
Let $\bm{x}$ be ARF to \eqref{eq: P}. Let $\varphi_i(\bm{x},\bm{z}) = \bm{r}_i(\bm{z}) - \bm{a}_i(\bm{z})^{\top}\bm{x}$ for each constraint $i=1,\dotsc,m$ of \eqref{eq: P-2}. Consider the system of inequalities $\bm{b}_i^{\top} \bm{y}(\bm{z}) \leq \varphi_i(\bm{x},\bm{z})$, $i=1,\dotsc,m$ and use FME to eliminate all variables. For all $k=1,\dotsc,n_y$ we can write the bounds after elimination of variable $y_k(\bm{z})$ as 
\begin{align} \label{eq: set-inequalities}
\begin{aligned}
&\max_{S_k \in C_{k}^{-}} \Big\{ \sum_{p \in S_k} \alpha(S_k,p) \varphi_p(\bm{x},\bm{z})  - \sum_{l=k+1}^{n_y} \beta(S_k,l) y_l(\bm{z}) \Big \} \leq y_k(\bm{z}) \\
&\hspace*{0cm} \leq \min_{T_k \in C_{k}^{+}} \Big\{ \sum_{q \in T} \alpha(T_k,q) \varphi_q(\bm{x},\bm{z})  - \sum_{l=k+1}^{n_y} \beta(T_k,l) y_l(\bm{z}) \Big \},~~\forall \bm{z} \in U,
\end{aligned}
\end{align}
for some coefficients $\alpha$ and $\beta$ independent of $\bm{z}$, and $C_k^{-},C_k^{+} \subseteq P(\{1,\dotsc,m\})$, with $P(\{1,\dotsc,m\})$ the power set of $\{1,\dotsc,m\}$. Additionally, if $S_k \in C_k^{-}$ for some $k$, then $\alpha(S_k,p) <0$ for all $p\in S_k$. If $T_k \in C_k^{+}$ for some $k$, then $\alpha(T_k,q) >0$ for all $q\in T_k$.
\end{lemma}
\begin{proof}
Proof by induction.\\

\noindent \emph{Base case:}\\
Elimination of variable $y_1(\bm{z})$ yields
\begin{align} \label{appeq: set-inequalities-1}
\max_{\{p:  b_{p,1} <0 \}} \Big\{ \frac{\varphi_p(\bm{x},\bm{z})}{b_{p,1}} - \frac{\sum_{l=2}^{n_y} b_{p,l} y_l(\bm{z})}{b_{p,1}} \Big\} \leq y_1(\bm{z}) \leq \min_{\{q:  b_{q,1} >0 \}} \Big\{ \frac{\varphi_q(\bm{x},\bm{z})}{b_{q,1}} - \frac{\sum_{l=2}^{n_y} b_{q,l} y_l(\bm{z})}{b_{q,1}}  \Big\}.
\end{align}
Define
\begin{align*}
C_1^{-} = \{p~|~ b_{p,1} <0 \},~~C_1^{+} = \{q~|~ b_{q,1} >0 \},
\end{align*} 
then each constraint in $C_1^{-}$ defines a lower bound on $y_1(\bm{z})$ and each constraint in $C_1^{+}$ defines an upper bound on $y_1(\bm{z})$. Each element of $C_1^{-}$ and $C_1^{+}$ is an individual `original' constraint index and not a set of constraints indices. For all $S_1 = \{p\} \in C_1^{-}$ set $\alpha(S,p) = b_{p,1}^{ -1}$, and for all $T_1 = \{q\} \in C_1^{+}$ set $\alpha(T,q) = b_{q,1}^{ -1}$. Furthermore, set $\beta(S_1,l) = b_{p,l}B_{p,1}^{-1}$ for all $S_1 = \{p\} \in C_1^{-} \cup C_1^{+}$ and all $l=2,\dotsc,n_y$. With these definitions, \eqref{appeq: set-inequalities-1} is reformulated in form \eqref{eq: set-inequalities}. Additionally, by construction, $\alpha(S_1,p) <0$ if  $p\in S_1$, $S_1\in C_1^{-}$ and  $\alpha(T_1,q) >0$ if $q\in T_1$, $T_1\in C_1^{+}$.\\

\noindent \emph{Induction step:}\\
Suppose the result holds for some $k-1$ (i.e., after elimination of variable $y_{k-1}(\bm{z})$). Variable $y_{k}(\bm{z})$ can occur in two types of constraints: (i) original constraints $i=1,\dotsc,m$ that do not depend on $y_1(\bm{z}),\dotsc,y_{k-1}(\bm{z}))$ and (ii) the new constraints acquired after elimination of $y_1(\bm{z}),\dotsc,y_{k-1}(\bm{z}))$. For case (i), define
\begin{align*}
I_k^{-} &= \{p~|~ b_{p,k} < 0,~b_{p,l} = 0,~\forall l=1,\dotsc,k-1 \},\\
I_k^{+} &= \{p~|~ b_{p,k} > 0,~b_{p,l} = 0,~\forall l=1,\dotsc,k-1 \},
\end{align*}
then each constraint in $I_k^{-}$ defines a lower bound on $y_k(\bm{z})$ and each constraint in $I_k^{+}$ provides an upper bound on $y_k(\bm{z})$. Reformulation to form \eqref{eq: set-inequalities} is similar to the case $k=1$. Thus, $\alpha(S_k,p) <0$ if $p\in S_k$, $S_k\in I_k^{-}$ and $\alpha(T_k,p) >0$ if $p\in T_k$, $T_k\in I_k^{+}$.  

For case (ii), $y_k(\bm{z})$ can occur in constraints resulting from picking linear lower and upper bounds on $y_l(\bm{z})$ from \eqref{eq: set-inequalities}. If these bounds are independent of $y_{l+1}(\bm{z}),\dotsc,y_{k-1}(\bm{z})$, for $l=1,\dotsc,k-1$, they are used directly to eliminate $y_k(\bm{z})$. For any such pair of constraints $S_l \in C_l^{-}$ and $T_l \in C_l^{+}$, FME yields the following bound on $y_k(\bm{z})$ (due to the induction assumption):
\begin{align} \label{eq: inequality-k}
\begin{aligned}
\sum_{p \in S_l} \alpha(S_l,p) \varphi_p(\bm{x},\bm{z}) - \sum_{q \in T_l} \alpha(T_l,q) \varphi_q(\bm{x},\bm{z}) -  \sum_{l=k+1}^{n_y} y_l(\bm{z}) \big(\beta(S_l,l)  - \beta(T_l,l) \big) \\ 
\hspace*{5cm} \leq y_k(\bm{z})  \big(\beta(S_l, k)  - \beta(T_l, k) \big).
\end{aligned}
\end{align}
We proceed by dividing by the coefficient of $y_k(\bm{z})$. If $\beta(S_l,k)  > \beta(T_l,k)$, inequality \eqref{eq: inequality-k} defines a lower bound for $y_{k}(\bm{z})$; if $\beta(S_l,k)  < \beta(T_l,k)$, inequality \eqref{eq: inequality-k} defines an upper bound for $y_{k}(\bm{z})$. Define
\begin{align*}
\begin{split}
J_k^{-} = \{S_k ~|~ \exists l=1,\dotsc,k-1 \text{ s.t. }&S_k = S_l \cup T_l,~ S_l \in C_l^{-},~T_l\in C_l^{+},\\
&  \beta(S_l,j) = \beta(T_l,j),~\forall j<l,~\beta(S_l,k) > \beta(T_l,k) \}, \\
\end{split}\\
\begin{split}
J_k^{+} = \{T_k ~|~ \exists l=1,\dotsc,k-1 \text{ s.t. }&T_k = S_l \cup T_l,~ S_l \in C_l^{-},~T_l\in C_l^{+}, \\
& \beta(S_l,j) = \beta(T_l,j),~\forall j<l,~\beta(S_l,k) < \beta(T_l,k) \} \\
\end{split}
\end{align*} 
so each element $S_k$ in $J_k^{-}$ (or $T_k$ in $J_k^{+}$) is a union of the indices of a lower bound constraint (set $S_l$) and an upper bound constraint (set $T_l$) on $y_l(\bm{z})$. The condition $\beta(S_l,j) = \beta(T_l,j),~\forall j<l$ on the second line ensures that these lower and upper bound constraints on $y_l(\bm{z})$ do not specify a constraint on $y_{l+1}(\bm{z}),\dotsc,y_{k-1}(\bm{z})$.

Set the coefficients for the not yet eliminated variables $y_{k+1}(\bm{z}),\dotsc,y_{n_y}(\bm{z})$ for form \eqref{eq: set-inequalities} as
\begin{align*}
\beta(S_k,j) &= \frac{\beta(S_l,j)  - \beta(T_l,j)}{\beta(S_l,k)  - \beta(T_l,k)},~\forall j=k+1,\dotsc,n_y.
\end{align*}
If $S_k \in J_k^{-}$, with $S_k = S_l \cup T_l$ for some $S_l \in C_l^{-}$ and $T_l \in C_l^{+}$, $l=1,\dotsc,k-1$, then set 
\begin{align} \label{eq: def-alpha-S-k}
\alpha(S_k, p) = 
\begin{cases} 
\displaystyle \frac{\alpha(S_l,p)}{\beta(S_l,k)  - \beta(T_l,k)} & \text{ if }p \in S_l,~ p \notin T_l, \\[1em]
\displaystyle \frac{\alpha(S_l,p) - \alpha(T_l,p)}{\beta(S_l,k)  - \beta(T_l,k)} & \text{ if } p \in S_l \cap T_l, \\[1em]
\displaystyle \frac{-\alpha(T_l,p)}{\beta(S_l,k)  - \beta(T_l,k)} & \text{ if }p \notin S_l,~ p \in T_l.
\end{cases}
\end{align}
Similarly, if $T_k \in J_k^{+}$, with $T_k = S_l \cup T_l$ for some $S_l \in C_l^{-}$ and $T_l \in C_l^{+}$ for some $l=1,\dotsc,k-1$, then set
\begin{align} \label{eq: def-alpha-T-k}
\alpha(T_k, p) &=
\begin{cases} 
\displaystyle \frac{\alpha(S_l,p)}{\beta(S_l,k)  - \beta(T_l,k)} & \text{ if }p \in S_l,~ p \notin T_l, \\[1em]
\displaystyle \frac{\alpha(S_l,p) - \alpha(T_l,p)}{\beta(S_l,k)  - \beta(T_l,k)} & \text{ if } p \in S_l, \cap T_l \\[1em]
\displaystyle \frac{-\alpha(T_l,p)}{\beta(S_l,k)  - \beta(T_l,k)} & \text{ if }p \notin S_l,~ p \in T_l.
\end{cases}
\end{align}
Due to the induction hypothesis, $\alpha(S_l,p)<0$ if $S_l \in C_l^{-}$ and $\alpha(T_l,p) >0$ if $T_l \in C_l^{+}$ for $l<k$. The denominator in both lines of \eqref{eq: def-alpha-S-k} is positive, so in that case $\alpha(S_k, p)<0$. The denominator in both lines of \eqref{eq: def-alpha-T-k} is negative, so in that case $\alpha(T_k, p)>0$. With the new coefficients chosen as above, \eqref{eq: inequality-k} provides a lower or upper bound on $y_{\hat{k}}(\bm{z})$ of the form inside the maximum or minimum operator in \eqref{eq: set-inequalities}, respectively. 

Finally, define $C_k^{-} = I_k^{-} \cup J_k^{-}$ and $C_k^{+} = I_k^{+} \cup J_k^{+}$. Each constraint in $C_k^{-}$ defines a lower bound on $y_k(\bm{z})$ and each constraint in $C_k^{+}$ defines an upper bound on $y_k(\bm{z})$. Moreover, set $C_k = C_k^{-} \cup C_k^{+}$ contains all constraints after elimination of $y_1(\bm{z}),\dotsc,y_{k-1}(\bm{z})$ that have $y_k(\bm{z})$ as lowest indexed adaptive variable. This completes the induction step.
\end{proof}

\subsection{\texorpdfstring{Proof \Cref{lemma: P-FME}}{}} \label{app: proof-P-FME}
Consider problem \eqref{eq: P}, with the objective moved to the constraints using epigraph variable $t\in \mathbb{R}$:
\begin{subequations} \label{appeq: P-epi}
\begin{align}
\min_{t,\bm{x},\bm{y}(\cdot)} ~&~ t, \\
\text{s.t.} ~&~ t \geq \bm{c}(\bm{z})^{\top} \bm{x} + \bm{d}^{\top} \bm{y}(\bm{z}),~~\forall \bm{z} \in U, \label{appeq: P-epi-2}\\
~&~ \bm{A}(\bm{z})\bm{x} + \bm{B}\bm{y}(\bm{z}) \leq \bm{r}(\bm{z}),~~\forall \bm{z} \in U. \label{appeq: P-epi-3}
\end{align}
\end{subequations}
Eliminate all adaptive variables in \eqref{appeq: P-epi-2}-\eqref{appeq: P-epi-3} via FME. Let $\varphi_0(\bm{x},t,\bm{z}) = t - \bm{c}(\bm{z})^{\top} \bm{x}$. In notation of \Cref{lemma: y-representation}, FME is performed on 
\begin{subequations} \label{appeq: FME-subset}
\begin{align} 
\bm{d}^{\top}\bm{y}(\bm{z}) &\leq \varphi_0(\bm{x},t,\bm{z}), \label{appeq: FME-subset-1}\\
\bm{b}_i^{\top} \bm{y}(\bm{z}) &\leq \varphi_i(\bm{x},t,\bm{z}),~\forall  i=1,\dotsc,m,\label{appeq: FME-subset-2}
\end{align}
\end{subequations}
where the coefficient for $t$ is zero in $\varphi_i$, $i=1,\dotsc,m$. According to \Cref{lemma: y-representation}, after elimination of variable $k$, inequalities \eqref{eq: set-inequalities} hold. Suppose for some $S_k \in C_{k}^{-}$, $T_k \in C_{k}^{+}$ the upper and lower bounds on $y_k(\bm{z})$ do not depend on $y_{k+1}(\bm{z}),\dotsc,y_{n_y}(\bm{z})$. Then the following constraint is derived for the static robust optimization problem after completing the full FME procedure:
\begin{align} \label{appeq: FME-constraint}
\sum_{p \in S_k} \alpha(S_k,p) \varphi_p(\bm{x},t,\bm{z}) \leq  \sum_{q \in T_k} \alpha(T_k,q) \varphi_q(\bm{x},t,\bm{z}),~~\forall \bm{z} \in U,
\end{align}
where $\varphi_p(\cdot)$ is a function of $t$ only if $p=0$. Constraints of the original system \eqref{appeq: P-epi-3} that are independent of adaptive variables can also be represented in form \eqref{appeq: FME-constraint}. Original constraints \eqref{appeq: P-epi-2} are part of a particular constraint in form \eqref{appeq: FME-constraint} if and only if $0 \in S_k \cup T_k$ for some $S_k \in C_{k}^{-}$, $T_k \in C_{k}^{+}$, $k=1,\dotsc,n_y$. Thus, problem \eqref{appeq: P-epi} after FME can be written as
\begin{subequations} \label{appeq: P-separate}
\begin{align}
\min_{t,\bm{x}} ~&~ t, \label{appeq: P-separate-1}\\
\text{s.t.} ~&~ \sum_{p \in S} \alpha(S,p) \varphi_p(\bm{x},t,\bm{z}) \leq  \sum_{q \in T} \alpha(T,q) \varphi_q(\bm{x},t,\bm{z}),~~\forall (S,T) \in M,~~\forall \bm{z} \in U, \label{appeq: P-separate-2}\\
~&~ \sum_{p \in S} \alpha(S,p) \varphi_p(\bm{x},t,\bm{z}) \leq  \sum_{q \in T} \alpha(T,q) \varphi_q(\bm{x},t,\bm{z}),~~\forall (S,T) \in N,~~\forall \bm{z} \in U, \label{appeq: P-separate-3}
\end{align}
\end{subequations}
with 
\begin{subequations}
\begin{align} 
\hspace*{-0.3cm} M &= \{(S,T)~|~ \exists k=1,\dotsc,n_y \text{ s.t. } S \in C_k^{-}, T\in C_k^{+}, \beta(S,l) = \beta(T,l),~\forall l>k,~ 0\in S \cup T \},  \\
\hspace*{-0.3cm} N &= \{(S,T)~|~ \exists k=1,\dotsc,n_y \text{ s.t. } S \in C_k^{-}, T\in C_k^{+},\beta(S,l) = \beta(T,l),~\forall l>k,~ 0\notin S \cup T \}. 
\end{align}
\end{subequations}
In other words, we separated the constraints depending on $t$ from the constraints not depending on $t$. From \Cref{lemma: y-representation} one can see that \eqref{appeq: P-separate-3} is the result of performing FME on the set of constraints \eqref{appeq: FME-subset-2}, which are the constraints defining set $\mathcal{X}$. Thus, \eqref{appeq: P-separate-3} describes set $\mathcal{X}_{\text{FME}}$. Furthermore, if we define $\alpha(S,0) = 0$ if $ 0\notin S$ and $\gamma(T,0) = 0$ if $0 \notin T$, constraint \eqref{appeq: P-separate-2} can be rewritten to
\begin{align} \label{appeq: t-pieces}
\begin{aligned}
\displaystyle t \geq \bm{c}(\bm{z})^{\top} \bm{x} + \sum_{p \in S, p>0} \frac{\alpha(S,p)}{\alpha(T,0) - \alpha(S,0)} \varphi_p(\bm{x},t,\bm{z}) - \sum_{q \in T, q>0} \frac{\alpha(T,q)}{\alpha(T,0) - \alpha(S,0)} \varphi_q(\bm{x},t,\bm{z}) \\
\hspace*{5cm} \forall (S,T) \in M,~~\forall \bm{z}\in U,
\end{aligned}
\end{align}
because $\alpha(T,0) > \alpha(S,0)$ according to \Cref{lemma: y-representation}. Note that the coefficient for $t$ is zero for all functions $\varphi$ on the RHS. Thus, for fixed $\bm{z} \in U$, constraint \eqref{appeq: t-pieces} defines a lower bound on epigraph variable $t$ that is convex PWL in $\bm{x}$. Subsequently, we eliminate $t$ and define
\begin{align} \label{appeq: h-ST}
h_{S,T}(\bm{x},\bm{z}) = \sum_{p \in S, p>0} \frac{\alpha(S,p)}{\alpha(T,0) - \alpha(S,0)} \varphi_p(\bm{x},\bm{z}) - \sum_{q \in T, q>0} \frac{\alpha(T,q)}{\alpha(T,0) - \alpha(S,0)} \varphi_q(\bm{x},\bm{z}).
\end{align}
This yields the following problem equivalent to \eqref{appeq: P-separate}:
\begin{align} \label{appeq: P-FME2}
\min_{\bm{x} \in \mathcal{X}_{\text{FME}}}~&~\max_{\bm{z} \in U} \bm{c}(\bm{z})^{\top} \bm{x} + \max_{(S,T)\in M} \{h_{S,T}(\bm{x},\bm{z}) \}.
\end{align}
If $(\bm{x}^{\ast},t^{\ast},\bm{y}^{\ast}(\cdot))$ is optimal to \eqref{appeq: P-epi}, $\bm{x}^{\ast}$ is optimal to \eqref{appeq: P-FME2} with equal objective value. This implies that $\bm{y}^{\ast}(\cdot)$ satisfies
\begin{align} \label{appeq: dy = max-h2}
\bm{d}^{\top}\bm{y}^{\ast}(\bm{z}) = \max_{(S,T)\in M} \{h_{S,T}(\bm{x}^{\ast},\bm{z}) \},~~\forall \bm{z}\in U.
\end{align}
Conversely, if $\bm{x}^{\ast}$ is optimal to \eqref{appeq: P-FME2}, there exists a $(t^{\ast},\bm{y}^{\ast}(\cdot))$ such that $(\bm{x}^{\ast},t^{\ast},\bm{y}^{\ast}(\cdot))$ is optimal to \eqref{appeq: P-epi} with equal objective value. This implies that any such $\bm{y}^{\ast}(\cdot)$ satisfies \eqref{appeq: dy = max-h2}. Lastly, note that $\bm{x}^{\ast}$ is optimal to \eqref{eq: P} if and only if there exists a $t^{\ast}\in \mathbb{R}$ such that $(t^{\ast},\bm{x}^{\ast})$ is optimal to \eqref{appeq: P-epi}. This completes the proof.

\subsection{\texorpdfstring{Proof \Cref{lemma: PRO-FME}}{}} \label{app: proof-PRO-FME}
By \Cref{def: PARO-x} a solution $\bm{x}^{\ast}$ is PARO to \eqref{eq: P} if and only if 
\begin{itemize}
\item There exists a $\bm{y}^{\ast} \in \mathcal{R}^{L,n_y}$ such that $(\bm{x}^{\ast}, \bm{y}^{\ast}(\cdot))$ is ARO to \eqref{eq: P} and there does not exist a pair $(\bm{\bar{x}},\bm{\bar{y}}(\cdot))$ that is ARO to \eqref{eq: P} and the following conditions hold:
\begin{align} \label{eq: PARO-inequalities}
\begin{aligned}
\bm{c}(\bm{z})^{\top}\bm{\bar{x}} + \bm{d}^{\top}\bm{\bar{y}}(\bm{z}) &\leq \bm{c}(\bm{z})^{\top}\bm{x^{\ast}} + \bm{d}^{\top}\bm{y}^{\ast}(\bm{z}),~~ \forall \bm{z} \in U, \\
\bm{c}(\bm{\bar{z}})^{\top}\bm{\bar{x}} + \bm{d}^{\top}\bm{\bar{y}}(\bm{\bar{z}}) &< \bm{c}(\bm{\bar{z}})^{\top}\bm{x^{\ast}} + \bm{d}^{\top}\bm{y}^{\ast}(\bm{\bar{z}}),~~ \text{for some } \bm{\bar{z}} \in U.
\end{aligned}
\end{align}
\end{itemize}
By \Cref{lemma: P-FME}, this holds if and only if
\begin{itemize}
\item $\bm{x}^{\ast}$ is optimal to \eqref{eq: P-FME} and there exists a $\bm{y}^{\ast} \in \mathcal{R}^{L,n_y}$ such that
\begin{align} \label{eq: dy = max-h-proof}
\bm{d}^{\top}\bm{y}^{\ast}(\bm{z}) = \max_{(S,T)\in M} \{h_{S,T}(\bm{x}^{\ast},\bm{z}) \}~~~\forall \bm{z}\in U,
\end{align} 
and there does not exist a $(\bm{\bar{x}},\bm{\bar{y}})$ such that $\bm{\bar{x}}$ is optimal to \eqref{eq: P-FME} and $(\bm{\bar{x}},\bm{\bar{y}}(\cdot))$ satisfies \eqref{eq: dy = max-h-proof} and \eqref{eq: PARO-inequalities} holds. 
\end{itemize}
Substituting \eqref{eq: dy = max-h-proof} in \eqref{eq: PARO-inequalities} yields the following set of equivalent conditions:
\begin{itemize}
\item $\bm{x}^{\ast}$ is optimal to \eqref{eq: P-FME} and there does not exist another $\bm{\bar{x}}$ optimal to \eqref{eq: P-FME} such that
\begin{align*} 
\bm{c}(\bm{z})^{\top}\bm{\bar{x}} + \max_{(S,T)\in M} \{h_{S,T}(\bm{\bar{x}},\bm{z}) \} &\leq \bm{c}(\bm{z})^{\top}\bm{x^{\ast}} + \max_{(S,T)\in M} \{h_{S,T}(\bm{x^{\ast}},\bm{z}) \},~~ \forall \bm{z} \in U, \\
\bm{c}(\bm{\bar{z}})^{\top}\bm{\bar{x}} + \max_{(S,T)\in M} \{h_{S,T}(\bm{\bar{x}},\bm{\bar{z}}) \} &< \bm{c}(\bm{\bar{z}})^{\top}\bm{x^{\ast}} +\max_{(S,T)\in M} \{h_{S,T}(\bm{x^{\ast}},\bm{\bar{z}}) \},~~ \text{for some } \bm{\bar{z}} \in U.
\end{align*}
\end{itemize}
This statement holds if and only if $\bm{x}^{\ast}$ is PRO to \eqref{eq: P-FME}, by \Cref{def: PRO}.

\subsection{\texorpdfstring{Proof \Cref{thm: x-PARO}}{}} \label{app: proof-x-PARO}
First, we prove the existence of PRO solutions to a general class of static RO problems, with bounded feasible region $\mathcal{X}$.
\begin{lemma} \label{lemma: PRO-existence}
Let $f: \mathbb{R}^n \times \mathbb{R}^L \mapsto \mathbb{R}$, with $f(\bm{x},\bm{z})$ continuous in $\bm{z}$. Consider the static RO problem
\begin{align}\label{eq: PRO-general}
\min_{\bm{x} \in \mathcal{X}} \max_{\bm{z} \in U}~&~f(\bm{x},\bm{z}).
\end{align}
Let $U \subseteq \mathbb{R}^L$ be closed, convex with a nonempty relative interior. If (i) $\mathcal{X}$ is compact and $f(\bm{x},\bm{z})$ continuous in $\bm{x}$ and/or (ii) $\mathcal{X}$ is a finite set, and additionally there exists an RO solution to \eqref{eq: PRO-general}, there also exists a PRO solution to \eqref{eq: PRO-general}.
\end{lemma}
\begin{proof}[Proof of \Cref{lemma: PRO-existence}.]
Let $(\mathbb{R}^L,\mathcal{B}(\mathbb{R}^L))$ be a measurable space, with $\mathcal{B}(\mathbb{R}^L)$ the Borel $\sigma$-algebra. For fixed $\bm{x}$, function $f(\bm{x},\bm{z})$ is continuous in $\bm{z}$, so it is measurable on closed subsets of $\mathbb{R}^L$, in particular set $U$. Define function $g: \mathbb{R}^n \mapsto \mathbb{R}$ with
\begin{align}\label{eq: integral}
g(\bm{x}) := \int_{U} f(\bm{x},\bm{z}) dP(\bm{z}),
\end{align}
where $P$ denotes a strictly positive probability measure on $\mathbb{R}^L$, such as the Gaussian measure. Because $0\leq P(U) \leq P(\mathbb{R}^L)=1$, the Lebesgue integral \eqref{eq: integral} assumes finite values for any $\bm{x}$. Hence, $f(\bm{x},\bm{z})$ is Lebesgue-integrable in its second argument on measured space $(\mathbb{R}^L,\mathcal{B}(\mathbb{R}^L),P)$ for any $\bm{x}$ and $g$ is well-defined.

We proceed by showing that an optimal solution to the following optimization problem is PRO to \eqref{eq: PRO-general}:
\begin{align} \label{eq: min-integral}
\min_{\bm{x} \in \mathcal{X}^{\text{RO}}} g(\bm{x}).
\end{align}
The remainder of the proof consists of two parts. First, we show that an optimal solution to \eqref{eq: min-integral} is always attained. Subsequently, we show that such an optimal solution is PRO to \eqref{eq: PRO-general}.\\

\emph{Part 1 (The optimum is attained)}:\\
We treat the two cases for $\mathcal{X}$ separately. 

Case (i):  Set $\mathcal{X}$ is compact and $f(\bm{x},\bm{z})$ continuous in $\bm{x}$. We show that $g$ is continuous. Consider a sequence $\{\bm{x}_n\}_{n\in \mathbb{N}}$ converging to $\bm{x}$. By continuity of $f$ in $\bm{x}$, $\lim_{n \rightarrow \infty} f(\bm{x}_n, \bm{z}) = f(\bm{x},\bm{z})$. Thus,
\begin{align} \label{eq: g-cont1}
g(\bm{x}) = \int_{U} f(\bm{x},\bm{z}) dP(\bm{z}) = \int_{U} \lim_{n \rightarrow \infty} f(\bm{x}_n, \bm{z}) dP(\bm{z}).
\end{align}
Let $M>0$ be such that $|f(\bm{x},\bm{z})| < M$, and define $h:\mathbb{R}^L \mapsto \mathbb{R}$ with $h(\bm{z}) =M$ for all $\bm{z}$. Then $h$ is Lebesgue-integrable, and we can apply the dominated convergence theorem to switch the order of the limit and integration in \eqref{eq: g-cont1} to obtain 
\begin{align*} 
g(\bm{x}) = \lim_{n \rightarrow \infty}  \int_{U} f(\bm{x}_n, \bm{z}) dP(\bm{z}) = \lim_{n \rightarrow \infty}  g(\bm{x}_n),
\end{align*}
Hence, $g(\bm{x})$ is continuous for each $\bm{x} \in \mathbb{R}^n$. Let $\mathcal{X}^{\text{RO}}$ denote the set of robustly (worst-case) optimal solutions to \eqref{eq:  PRO-general}. Then $\mathcal{X}^{\text{RO}}$ is compact if $\mathcal{X}$ is compact. Problem \eqref{eq: min-integral} minimizes a continuous function over a compact domain, so, by the extreme value theorem, a minimum is always attained. 

Case (ii): Set $\mathcal{X}$ is a finite set. Problem \eqref{eq: min-integral} minimizes $g(\bm{x})$ over a finite set, so the minimum is attained.\\

\emph{Part 2 (An optimal solution is PRO)}:\\
Let $\bm{\hat{x}}$ denote an optimal solution to \eqref{eq: min-integral}. We proceed by showing via proof by contradiction that $\bm{\hat{x}}$ is PRO to \eqref{eq: PRO-general}. Suppose $\bm{\hat{x}}$ is not PRO to \eqref{eq: PRO-general}. Then there exists an $\bm{\bar{x}} \in \mathcal{X}^{\text{RO}}$ such that 
\begin{align*}
f(\bm{\bar{x}},\bm{z}) &\leq f(\bm{\hat{x}},\bm{z}),~~\forall \bm{z} \in U, \\
f(\bm{\bar{x}},\bm{\bar{z}}) &< f(\bm{\hat{x}},\bm{\bar{z}}),~~\text{ for some } \bm{\bar{z}} \in U.
\end{align*}
We proceed by showing that there must exist a ball contained in $U$ with strictly positive measure where strict inequality holds. Let $\bar{B}$ denote the ball with radius $\delta$ centered at $\bm{\bar{z}}$:
\begin{align*}
\bar{B} = \{\bm{z} \in \mathbb{R}^L : \|\bm{z} - \bm{\bar{z}}\|_2 \leq \delta \}.
\end{align*}
By continuity of $f(\bm{\bar{x}},\bm{z}) - f(\bm{\hat{x}},\bm{z})$ w.r.t. $\bm{z}$, there exists a $\delta >0$ such that for each $\bm{z} \in \bar{B}$ it holds that $f(\bm{\bar{x}},\bm{z}) - f(\bm{\hat{x}},\bm{z}) < 0$. Note that $\bm{\bar{z}}$ need not be in the relative interior of $U$. Hence, the ball $\bar{B}$ need not be contained in $U$. Let $\bm{\tilde{z}} \in \text{ri}(U)$. We construct a new scenario $\bm{z}^{\ast} = \theta \bm{\tilde{z}} + (1-\theta) \bm{\bar{z}}$. Because $U$ is convex, $\bm{z}^{\ast} \in \text{ri}(U)$ if $0 \leq \theta <1$ according to \citet[Theorem 6.1]{Rockafellar70}. Choosing $1 - \delta \|\bm{\tilde{z}} -  \bm{\bar{z}}\|_2^{-1} < \theta < 1$ ensures that $\bm{z}^{\ast} \in \text{int}(\bar{B}) \cap \text{ri}(U) = \text{ri}(U \cap \bar{B})$. Consider the ball $B^{\ast}$ with radius $\epsilon>0$ centered at $\bm{z}^{\ast}$:
\begin{align*}
B^{\ast} = \{\bm{z} \in \mathbb{R}^L : \|\bm{z} - \bm{z}^{\ast} \|_2 \leq \epsilon \}.
\end{align*}
For sufficiently small $\epsilon > 0$, it holds that $\bm{z} \in B^{\ast} \Rightarrow \bm{z} \in U \cap \bar{B}$. In other words, for such an $\epsilon$, each point $\bm{z} \in B^{\ast}$ is in the uncertainty set $U$ and is such that $f(\bm{\bar{x}},\bm{z}) < f(\bm{\hat{x}},\bm{z})$. 

Finally, we consider the difference between $g(\bm{\bar{x}})$ and $g(\bm{\hat{x}})$ on $U$. Note that $|g(\bm{x})|<\infty$ for all $\bm{x}$. The following holds:
\begin{align*}
g(\bm{\bar{x}}) - g(\bm{\hat{x}}) = \int_{U\backslash B^{\ast}} f(\bm{\bar{x}},\bm{z}) - f(\bm{\hat{x}},\bm{z}) d
 P(\bm{z}) + \int_{B^{\ast}} f(\bm{\bar{x}},\bm{z}) - f(\bm{\hat{x}},\bm{z}) dP(\bm{z}). 
\end{align*}
The first integral is nonpositive since $f(\bm{\bar{x}},\bm{z}) \leq f(\bm{\hat{z}},\bm{z})$ for each $\bm{z} \in U\backslash B^{\ast}$. The second integral is strictly negative since $f(\bm{\bar{x}},\bm{z}) < f(\bm{\hat{z}},\bm{z})$ for $\bm{z}\in B^{\ast}$ and measure $P$ is strictly positive, i.e., $P(B^{\ast}) >0$. Hence, $g(\bm{\bar{x}}) < g(\bm{\hat{x}})$, contradicting the fact that $\bm{\hat{x}}$ is optimal to \eqref{eq: min-integral}.
\end{proof}

The result of \Cref{thm: x-PARO} immediately follows.
\begin{proof}[Proof of \Cref{thm: x-PARO}.]
By \Cref{lemma: PRO-FME}, it suffices to prove existence of a PRO solution to \eqref{eq: P-FME}. Because $\mathcal{X}=\mathcal{X}_{\text{FME}}$, set $\mathcal{X}_{\text{FME}}$ is compact. By construction of \eqref{eq: P}, uncertainty set $U$ is assumed to be convex, compact with a nonempty relative interior. Lastly, the objective function of \eqref{eq: P-FME} is continuous in $\bm{x}$ and $\bm{z}$. Hence, all conditions of \Cref{lemma: PRO-existence} are satisfied, and existence of a PARO solution to \eqref{eq: P} is guaranteed.
\end{proof}

\subsection{\texorpdfstring{Proof \Cref{lemma: PARO-extension-PWL} via FME}{}} \label{app: proof-PARO-extension-PWL-FME}
Let $\bm{x}$ be ARF to \eqref{eq: P}. W.l.o.g., suppose in the FME procedure the adaptive variables are eliminated in the order $y_1,\dotsc,y_{n_y}$, i.e., according to their index. Let $F_k(y_{k+1}(\bm{z}),\dotsc,y_{n_y}(\bm{z}),\bm{z})$ denote the optimal decision rule for $y_k$ as a function of the decision rules for the adaptive variables with higher index and the uncertain parameter $\bm{z}$. We prove by induction on $k=1,\dotsc,n_y$ that $F_k(y_{k+1}(\bm{z}),\dotsc,y_{n_y}(\bm{z}),\bm{z})$ is jointly PWL in $y_{k+1},\dotsc,y_{n_y}$ and $\bm{z}$. 

According to \Cref{lemma: y-representation}, we can write the bounds after elimination of variable $y_1(\bm{z})$ as 
\begin{align*} 
&\max_{S \in C_{1}^{-}} \Big\{ \sum_{p \in S} \alpha(S,p) \varphi_p(\bm{z})  - \sum_{l=2}^{n_y} \beta(S,l) y_l(\bm{z}) \Big \} \leq y_1(\bm{z}) \\
&\hspace*{3cm} \leq \min_{T \in C_{1}^{+}} \Big\{ \sum_{q \in T} \alpha(T,q) \varphi_q(\bm{z})  - \sum_{l=2}^{n_y} \beta(T,l) y_l(\bm{z}) \Big \},~~\forall \bm{z} \in U,
\end{align*}
for some coefficients $\alpha$ and $\beta$ independent of $\bm{z}$. For fixed $y_2,\dotsc,y_{n_y}$, $\bm{z}$ and $\bm{x}$, the highest possible contribution of $y_1$ to the objective value is achieved by setting $y_1$ equal to its upper bound if $d_1 <0$, and equal to its lower bound if $d_1 >0$. Thus, $F_1(y_2(\bm{z}),\dotsc,y_{n_y}(\bm{z}),\bm{z})$ is equal to either the upper or the lower bound on $y_1$. Both the upper and lower bound are jointly PWL in $y_i$, $i=2,\dotsc,n_y$ and $\bm{z}$. 

Now, suppose that for each $i=1,\dotsc,k-1$, after elimination of variable $y_{i}(\bm{z})$ the optimal decision rule $F_{i}(y_{i+1}(\bm{z}),\dotsc,y_{n_y}(\bm{z}),\bm{z})$ is jointly PWL in $y_{i+1},\dotsc,y_{n_y}$.
 
After elimination of $y_k(\bm{z})$ we can again write the bounds according to \Cref{lemma: y-representation}. For fixed $y_{k+1},\dotsc,y_{n_y}$, $\bm{z}$ and $\bm{x}$, the highest possible contribution of $y_k$ to the objective value is achieved by minimizing $\bm{d}^{\top}\bm{y}$, i.e., solving 
\begin{subequations} \label{eq: yk-problem}
\begin{align}
\begin{split} 
\min_{y_k} ~&~ \sum_{i=1}^{k-1} d_i F_i(F_{i+1}(\dotsc),\dotsc,F_{k-1}(y_k(\bm{z}),\dotsc,y_{n_y}(\bm{z}),\bm{z}), y_k(\bm{z}),\dotsc,y_{n_y}(\bm{z}),\bm{z}) \\
& + d_k y_k(\bm{z}) + \sum_{i=k+1}^{n_y}d_i y_i(\bm{z}), 
\end{split}   \label{eq: yk-problem-1}\\
\text{s.t.} ~&~ \max_{S \in C_{k}^{-}} \Big\{ \sum_{p \in S} \alpha(S,p) \varphi_p(\bm{z})  - \sum_{l=k+1}^{n_y} \beta(S,l) y_l(\bm{z}) \Big \} \leq y_k(\bm{z}), \label{eq: yk-problem-2}\\
~&~ \min_{T \in C_{k}^{+}} \Big\{ \sum_{q \in T} \alpha(T,q) \varphi_q(\bm{z})  - \sum_{l=k+1}^{n_y} \beta(T,l) y_l(\bm{z}) \Big \} \geq y_k(\bm{z}),\label{eq: yk-problem-3}
\end{align}
\end{subequations}
where the last term in the objective (the last summation) may be dropped because it does not depend on $y_k$. In the objective each decision rule $F_i$, $i=1,\dotsc,k-1$, is a function of the decision rules $F_{i+1},\dotsc, F_{k-1}$, variables $y_k(\bm{z}),\dotsc,y_{n_y}(\bm{z})$ and $\bm{z}$. Plugging in a PWL argument in a PWL function retains the piecewise linear structure. Thus, \eqref{eq: yk-problem} asks to minimize a univariate PWL function on a closed interval. The optimum is attained at either an interior point or a boundary point; we consider these cases separately.
\begin{itemize}
\item Problem \eqref{eq: yk-problem} has a boundary minimum. The minimum is attained at either the lower or upper bounds provided by \eqref{eq: yk-problem-2} and \eqref{eq: yk-problem-3}. In this case, $F_k(y_{k+1}(\bm{z}),\dotsc,y_{n_y}(\bm{z}),\bm{z})$ is clearly jointly PWL in $y_{k+1}(\bm{z}),\dotsc,y_{n_y}(\bm{z})$ and $\bm{z}$. 
\item Problem \eqref{eq: yk-problem} has an interior minimum. The unrestricted minimum of \eqref{eq: yk-problem-1} is at the intersection of two functions that are jointly linear in $y_{k},\dotsc,y_{n_y}$ and $\bm{z}$. Any intersection point can be expressed as
\begin{align*}
s_0(\bm{z}) + \sum_{i=k}^{n_y} s_i y_i(\bm{z}) = t_0(\bm{z}) + \sum_{i=k}^{n_y} t_i y_i(\bm{z}), 
\end{align*}
for some scalars $s_0(\bm{z})$ and $t_0(\bm{z})$ depending linearly on $\bm{z}$ and some vectors $\bm{s},\bm{t} \in \mathbb{R}^{n_y-k}$. This is equivalent to
\begin{align*}
y_k(\bm{z}) = \frac{s_0(\bm{z}) - t_0(\bm{z}) + \sum_{i=k+1}^{n_y}(s_i-t_i) y_i(\bm{z}) }{t_k - s_k},
\end{align*}
and this is jointly linear in $y_{k},\dotsc,y_{n_y}$ and $\bm{z}$. The pair $\{(s_0(\bm{z}),\bm{s}) ,(t_0(\bm{z}),\bm{t})\}$ that defines the interior minimum intersection point depends on $y_{k},\dotsc,y_{n_y}$ and $\bm{z}$. Thus, the optimal decision rule $F_k(y_{k+1}(\bm{z}),\dotsc,y_{n_y}(\bm{z}),\bm{z})$ is a PWL function of $y_{k+1},\dotsc,y_{n_y}$ and $\bm{z}$.
\end{itemize}
This completes the induction step. Lastly, note that $F_{n_y}(\bm{z})$ is PWL in $\bm{z}$ and that plugging in a PWL argument in a PWL function retains the piecewise linear structure. Thus, going from $k=n_y$ to $k=1$ and for each $k$ plugging in $F_{k}(y_{k+1}(\bm{z}),\dotsc,y_{n_y}(\bm{z}),\bm{z})$ in $F_{k-1}(y_{k}(\bm{z}),\dotsc,y_{n_y}(\bm{z}),\bm{z})$ yields decision rules that are PWL in $\bm{z}$ for all variables $y_1,\dotsc,y_{n_y}$.

\subsection{\texorpdfstring{Proof \Cref{lemma: PARO-extension-PWL} via linear optimization}{}} \label{app: proof-PARO-extension-PWL-LO}
Let $\bm{x}$ be ARF to \eqref{eq: P}. We make use of the concept of basic solutions in linear optimization \citep{Bertsimas97}. In standard form the remaining problem for $\bm{y}$ for fixed $\bm{z}$, reads:
\begin{subequations} \label{eq: LP-v}
\begin{align}
\min_{\bm{y}^{+},\bm{y}^{-},\bm{s}} ~&~ \bm{d}^{\top} \big(\bm{y}^{+}- \bm{y}^{-}\big), \\
\text{s.t.} ~&~ \bm{B}\big(\bm{y}^{+} - \bm{y}^{-} \big) + \bm{s} = \bm{r}(\bm{z}) - \bm{A}(\bm{z})\bm{x}, \\
~&~ \bm{y}^{+},\bm{y}^{-},\bm{s} \geq \bm{0},
\end{align}
\end{subequations}
where $\bm{s}$ is a slack variable and $\bm{y}$ is represented by the difference of two nonnegative variables. Let $\bm{v} \in \mathbb{R}^{2n_y+m}$, $\bm{M} \in \mathbb{R}^{m\times(2n_y+m)}$ and $\bm{f} \in \mathbb{R}^{2n_y+m}$ denote the vector of decision variables, the equality constraint matrix and the objective vector of \eqref{eq: LP-v}, respectively:
\begin{align}
\bm{v} = [\bm{y}^{+} ~ \bm{y}^{-} ~ \bm{s}]^{\top},~\bm{M} = [\bm{B}~ \minus\bm{B}~ \bm{I}],~\bm{f} = [\bm{d} ~ \minus\bm{d}~ \bm{0}]^{\top}.
\end{align}
Each basis is represented by $m$ linearly independent columns of $\bm{M}$. Let $\bm{W} \in \mathbb{R}^{m\times m}$ denote a basis matrix, and let $\bm{v}_{\bm{W}}$ and $\bm{f}_{\text{W}}$ denote the components of $\bm{v}$ and $\bm{f}$ corresponding to the basic variables. For any basic solution $\bm{v}$ it holds that
\begin{align} \label{eq: LP-v-components}
\bm{v}_{\bm{W}} = \bm{W}^{-1}\big(\bm{r}(\bm{z}) - \bm{A}(\bm{z})\bm{x}\big),
\end{align}
and the remaining non-basic components of $\bm{v}$ are equal to zero. Denote the basic solution by $(\bm{v}_{\bm{W}},\bm{0}_{\backslash\bm{W}})$; it is a basic feasible solution (BFS) to \eqref{eq: LP-v} if and only if $\bm{v}_{\bm{W}} \geq \bm{0}$. For optimality of $(\bm{v}_{\bm{W}},\bm{0}_{\backslash\bm{W}})$ it is additionally required that the reduced costs are nonnegative. Nonnegativity of the reduced costs (i.e., optimality of $(\bm{v}_{\bm{W}},\bm{0}_{\backslash\bm{W}})$) reads
\begin{align} \label{eq: LP-v-optimality}
\bm{f} - \bm{f}_{\text{W}}^{\top} \bm{W}^{-1}\bm{M} \geq \bm{0}.
\end{align}

We restrict ourselves to those basic solutions for which optimality condition \eqref{eq: LP-v-optimality} holds, note that this condition is independent of $\bm{z}$. It follows that for each basis matrix $\bm{W}$ that satisfies \eqref{eq: LP-v-optimality}, it associated basic solution $(\bm{v}_{\bm{W}},\bm{0}_{\backslash\bm{W}})$ is feasible (and optimal) if and only if $\bm{z}$ is in the following subset of $U$:
\begin{align*}
U_{\bm{W}}(\bm{x}) = \{\bm{z}\in U ~:~ \bm{W}^{-1}\big(\bm{r}(\bm{z}) - \bm{A}(\bm{z})\bm{x}\big) \geq \bm{0} \}.
\end{align*}
Let $\bm{y}(\bm{x},\bm{z},\bm{W})$ denote the basic solution corresponding to $\bm{W}$ in terms of the original variables $\bm{y}$. From \eqref{eq: LP-v-components} it follows that $\bm{y}(\bm{x},\bm{z},\bm{W})$ is linear in $\bm{z}$.

Any basic solution to \eqref{eq: LP-v} corresponds with at least one basis, and each basis is represented by $m$ linearly independent columns of $\bm{M}$. Thus, there are at most $\beta=\binom{2n_y+m}{m}$ bases (i.e., matrices $\bm{W}$) to \eqref{eq: LP-v} that satisfy \eqref{eq: LP-v-optimality}, independent of $\bm{z}$. Number the matrices $\bm{W}_1,\dotsc,\bm{W}_{\beta}$. Each of these matrices $\bm{W}_j$ has its own LDR $\bm{y}(\bm{x},\bm{z},\bm{W}_j)$ that is optimal for all $\bm{z} \in U_{\bm{W}_j}(\bm{x})$.

Because $\bm{x}$ is ARF to \eqref{eq: P} and \eqref{eq: P} has a finite optimal objective value, problem \eqref{eq: LP-v} is feasible and has a finite optimum for all $\bm{z} \in U$. Therefore, there exists an optimal basic feasible solution for all $\bm{z} \in U$, and the union of all $U_{\bm{W}_i}$ equals $U$ itself. This implies that, for the given $\bm{x}$, the following PWL decision rule is optimal for each $\bm{z} \in U$:
\begin{align*}
\bm{y}(\bm{z}) = \bm{y}(\bm{x},\bm{z},\bm{W}_{i^{\ast}}) \text{ if } i^{\ast} = \min\{i: \bm{z}\in U_{\bm{W}_i}(\bm{x})\}.
\end{align*}
Note that a different numbering of the matrices gives a (possibly) different optimal PWL decision rule. In essence, the proof performs sensitivity analysis on the right-hand side vectors of \eqref{eq: LP-v}, which is the only term in \eqref{eq: LP-v} that depends on $\bm{z}$.

\subsection{\texorpdfstring{Proof \Cref{theorem: DR-PARO}}{}} \label{app: proof-DR-PARO}
Let $\text{OPT}$ denote the optimal (worst-case) objective value of $P$. By \Cref{def: PARO-x}, and using that $\bm{d} = \bm{0}$, a solution $\bm{x}^{\ast}$ is PARO to $P$ if and only if the following statement holds:
\begin{itemize}
\item There exists a $\bm{y}^{\ast} \in \mathcal{R}^{L,n_y}$ such that $(\bm{x}^{\ast}, \bm{y}^{\ast}(\cdot))$ is ARO to $P$ and there does not exist a pair $(\bm{\bar{x}},\bm{\bar{y}}(\cdot))$ that is ARO to $P$ and 
\begin{align} \label{eq: x-PRO-inequalities}
\begin{aligned}
\bm{c}(\bm{z})^{\top}\bm{\bar{x}} &\leq \bm{c}(\bm{z})^{\top}\bm{x^{\ast}}, ~~ \forall \bm{z} \in U,\\
\bm{c}(\bm{\bar{z}})^{\top}\bm{\bar{x}} &< \bm{c}(\bm{\bar{z}})^{\top}\bm{x^{\ast}}, ~~ \text{for some } \bm{\bar{z}} \in U.
\end{aligned}
\end{align}
\end{itemize}
By definition of set $\mathcal{X}$, this holds if and only if 
\begin{itemize}
\item $\bm{x}^{\ast} \in \mathcal{X}$, $\text{OPT} =  \max_{\bm{z} \in U} \bm{c}(\bm{z})^{\top}\bm{x}^{\ast}$ and there does not exist an $\bm{\bar{x}} \in \mathcal{X}$ such that $\text{OPT} =  \max_{\bm{z} \in U} \bm{c}(\bm{z})^{\top}\bm{\bar{x}}$ and \eqref{eq: x-PRO-inequalities} holds. 
\end{itemize}
Because for any ARF $\bm{x}$ there exists an ARF decision rule $\bm{y}(\cdot)$ such that $\bm{y}(\bm{z}) = f_{\bm{w}}(\bm{z})$ for some $\bm{w}$, it follows that $\mathcal{X}$ is equal to
\begin{align*}
\mathcal{X}_f = \{\bm{x} \in \mathbb{R}^{n_x}~|~\exists \bm{w} \in \mathbb{R}^p  : \bm{A}(\bm{z})\bm{x} + \bm{B}f_{\bm{w}}(\bm{z}) \leq \bm{r}(\bm{z}),~~\forall \bm{z} \in U \},
\end{align*} 
which is the set of feasible $\bm{x}$ when Stage-2 decision rules are restricted to be of form $f_{\bm{w}}(\bm{z})$. Hence, the previous set of conditions holds if and only if 
\begin{itemize}
\item $\bm{x}^{\ast} \in \mathcal{X}_f$, $\text{OPT} =  \max_{\bm{z} \in U} \bm{c}(\bm{z})^{\top}\bm{x}^{\ast}$ and there does not exist an $\bm{\bar{x}} \in \mathcal{X}$ such that $\text{OPT} =  \max_{\bm{z} \in U} \bm{c}(\bm{z})^{\top}\bm{\bar{x}}$ and \eqref{eq: x-PRO-inequalities} holds.
\end{itemize}
Parameters $\bm{w}$ are now Stage-1 decision variables, so $\mathcal{X}_f$ does not contain adaptive variables. The set of conditions describes a PRO solution to the static robust optimization problem obtained after plugging in decision rule structure $f_{\bm{w}}(\cdot)$.

\subsection{\texorpdfstring{Proof \Cref{cor: SDR-LDR-PARO}}{}} \label{app: proof-SDR-LDR-PARO}
\emph{\Cref{cor: hybrid-PARO}}: For any vector of parameters $\bm{w} \in \mathbb{R}^p$, let $f_{\bm{w}}(\bm{\hat{z}})$ denote a decision rule that depends only on $\bm{\hat{z}} \in \hat{U}$, the non-constraintwise component of uncertain parameter $\bm{z}$. From \Cref{lemma: hybrid-ARF} it follows that $\mathcal{X}$ is equal to
\begin{align*}
\mathcal{X}_{\text{hybrid}} = \{\bm{x} \in \mathbb{R}^{n_x}~|~\exists \bm{w} \in \mathbb{R}^p : \bm{A}(\bm{z})\bm{x} + \bm{B} f_{\bm{w}}(\bm{\hat{z}}) \leq \bm{r}(\bm{z}),~~\forall \bm{z} \in U \},
\end{align*} 
i.e., the feasible region for $\bm{x}$ remains unchanged if all adaptive variables are restricted to depend only on the non-constraintwise component of $\bm{z}$. Hence, setting $X_f = \mathcal{X}_{\text{hybrid}}$ in the proof of \Cref{theorem: DR-PARO} yields the result.\\

\emph{\Cref{cor: block-PARO}}: For each block $v=1,\dotsc,V$, let $\bm{w}(v) \in \mathbb{R}^{p(v)}$ denote a vector of parameters and let $f^{v}_{\bm{w}(v)}(\bm{z}_{(v)})$ denote a decision rule that depends only on $\bm{z}_{(v)}$, the uncertain parameters in block $v$.  From \Cref{lemma: block-ARF} it follows that $\mathcal{X}$ is equal to 
\begin{align*}
\mathcal{X}_{\text{block}} &= \{\bm{x} \in \mathbb{R}^{n_x}~|~ \forall v=1,\dotsc,V,~\exists \bm{w}(v) \in \mathbb{R}^{p(v)} : \bm{a}_i(\bm{z}_{(v)})^{\top}\bm{x} + \bm{b}_i^{\top}f^{v}_{\bm{w}(v)}(\bm{z}_{(v)}) \leq r_i(\bm{z}_{(v)}),\\
& \hspace*{6cm} ~~\forall \bm{z} \in U^v,~\forall i \in K(v) \}, 
\end{align*}
i.e., the feasible region for $\bm{x}$ remains unchanged if all adaptive variables are restricted to depend only on uncertain parameters in their own block. Hence, setting $X_f = \mathcal{X}_{\text{block}}$ in the proof of \Cref{theorem: DR-PARO} yields the result.\\

\emph{\Cref{cor: simplex-LDR-PARO}}: From \Cref{lemma: simplex-LDR-ARF} it follows that for simplex uncertainty $\mathcal{X}$ is equal to 
\begin{align*}
\mathcal{X}_{\text{simplex}} &= \{\bm{x} \in \mathbb{R}^{n_x}~|~\exists \bm{u} \in \mathbb{R}^{n_y},~\bm{V} \in \mathbb{R}^{n_y\times L} : \bm{A}(\bm{z})\bm{x} + \bm{B}(\bm{u} + \bm{V} \bm{z}) \leq \bm{r}(\bm{z}),~~\forall \bm{z} \in U \}, 
\end{align*}
i.e., the feasible region for $\bm{x}$ remains unchanged if all adaptive variables are restricted to depend affinely on $\bm{z}$. Hence, setting $X_f = \mathcal{X}_{\text{simplex}}$ in the proof of \Cref{theorem: DR-PARO} yields the result.

\subsection{\texorpdfstring{Proof \Cref{lemma: PARO-x-unique}}{}} \label{app: proof-PARO-x-unique}
The two cases are considered separately.
\begin{itemize}
\item \emph{Optimal objective value is zero}: Proof by contradiction. Suppose $\bm{y}^{\ast}(\cdot)$ is not a PARO extension of $\bm{x}^{\ast}$. Then, by \Cref{def: PARO-ext-y}, there exists a $\bm{\tilde{y}}(\cdot)$ such that $(\bm{x}^{\ast},\bm{\tilde{y}}(\cdot))$ is ARO to \eqref{eq: P} and for some $\bm{\tilde{z}} \in U$ it holds that
\begin{align*}
\bm{c}(\bm{\tilde{z}})^{\top} \bm{x}^{\ast} + \bm{d}^{\top} \bm{y}^{\ast}(\bm{\tilde{z}}) > \bm{c}(\bm{\tilde{z}})^{\top} \bm{x}^{\ast} + \bm{d}^{\top} \bm{\tilde{y}}(\bm{\tilde{z}}).
\end{align*}
However, then $(\bm{z},\bm{y}) = (\bm{\tilde{z}}, \bm{\tilde{y}}(\bm{\tilde{z}}))$ is feasible to \eqref{eq: check-PARO-extension} with positive objective value. This is a contradiction.
\item \emph{Optimal objective value is positive}:  Let $(\bm{\bar{z}},\bm{\bar{y}})$ denote the optimal solution to \eqref{eq: check-PARO-extension} and let $\bar{v}$ denote the optimal objective value. The decision rule 
\begin{align*}
\bm{y}(\bm{z}) = 
\begin{cases}
\bm{y}^{\ast}(\bm{z}) & \text{ if } \bm{z} \neq \bm{\bar{z}} \\
\bm{\bar{y}} & \text{ otherwise,} 
\end{cases}
\end{align*}
dominates the decision rule $\bm{y}^{\ast}(\cdot)$, so the latter is not PARO. We prove the last part of the lemma by contradiction. Suppose there exists a scenario $\bm{\tilde{z}}$ and a decision $\bm{\tilde{y}}$ such that 
\begin{align*} 
\Big(\bm{c}(\bm{\tilde{z}})^{\top} \bm{x}^{\ast} + \bm{d}^{\top} \bm{y}^{\ast}(\bm{\tilde{z}})\Big) - \Big( \bm{c}(\bm{\tilde{z}})^{\top} \bm{x^{\ast}} + \bm{d}^{\top}\bm{\tilde{y}} \Big) &> \bar{v}, \\
\bm{A}(\bm{\tilde{z}}) \bm{x}^{\ast} + \bm{B}\bm{\tilde{y}} &\leq \bm{r}(\bm{\tilde{z}}),
\end{align*}
i.e., $\bm{\tilde{y}}$ is a feasible wait-and-see decision for scenario $\bm{\tilde{z}}$, and the resulting objective value of $\bm{y}^{\ast}(\bm{\tilde{z}})$ exceeds that of $\bm{\tilde{y}}$ by more than $\bar{v}$. Then $(\bm{\tilde{z}}, \bm{\tilde{y}})$ is feasible to \eqref{eq: check-PARO-extension} with a strictly better objective value than $\bar{v}$. This is a contradiction.
\end{itemize}

\subsection{\texorpdfstring{Proof \Cref{lemma: PARO-extreme-points-d0}}{}} \label{app: proof-PARO-extreme-points-d0}
Proof by contradiction, analogous to proof of Theorem 1 of \citet{Iancu14}. Because $U$ is the convex hull of $\bm{z}^1,\dotsc,\bm{z}^N$, \eqref{eq: aux-d0-2} and \eqref{eq: aux-d0-3} ensure that $\bm{x}^{\ast}$ is ARO to \eqref{eq: P} (with $\bm{d}=\bm{0}$). Suppose $\bm{x}^{\ast}$ is not PARO to \eqref{eq: P}. According to \Cref{def: PARO-x} there exists an $\bm{\hat{x}}$ that is ARO to \eqref{eq: P} and
\begin{align*}
\bm{c}(\bm{z})^{\top}\bm{\hat{x}} &\leq \bm{c}(\bm{z})^{\top}\bm{x}^{\ast}, ~~\forall \bm{z} \in U, \\
\bm{c}(\bm{\hat{z}})^{\top}\bm{\hat{x}} &< \bm{c}(\bm{\hat{z}})^{\top}\bm{x}^{\ast}, ~~ \text{for some } \bm{\hat{z}} \in U.
\end{align*}
Because $\bm{\hat{x}}$ is ARO to \eqref{eq: P}, there also exist $(\bm{\hat{y}^1},\dotsc,\bm{\hat{y}^N})$ that, together with $\bm{\hat{x}}$, are feasible to \eqref{eq: aux-d0}.

The linear optimization problem $\min_{\bm{z} \in U} \bm{c}(\bm{z})^{\top}(\bm{\hat{x}} - \bm{x}^{\ast})$ attains the minimum in a vertex solution, so without loss of generality we can assume $\bm{\hat{z}} \in \text{ext}(U)$. Any point $\bm{\bar{z}} \in \text{ri}(U)$ can be written as a strict convex combination of the extreme points of $U$ \citep{Rockafellar70}, so $\bm{\bar{z}} = \sum_{i=1}^N \alpha_i \bm{z}^i$ for some $\bm{\alpha} \in \mathbb{R}^{N}$ with $\sum_{i=1}^N \alpha_i = 1$, $\alpha_i >0$ for all $i$. Then
\begin{align*}
\bm{c}(\bm{\bar{z}})^{\top}(\bm{\hat{x}} - \bm{x}^{\ast}) = \sum_{\substack{i=1 \\ \bm{z}^i \neq \bm{\hat{z}}}}^N \alpha_i \bm{c}(\bm{z}^i)^{\top}(\bm{\hat{x}} - \bm{x}^{\ast}) + \hat{\alpha} \bm{c}(\bm{\hat{z}})^{\top}(\bm{\hat{x}} - \bm{x}^{\ast}),
\end{align*}
where the first term of the RHS is nonpositive and the second term is strictly negative. This contradicts the fact that $(\bm{x}^{\ast},\bm{y}^{1\ast},\dotsc,\bm{y}^{N\ast})$ is optimal to \eqref{eq: aux-d0}.

\subsection{\texorpdfstring{Proof \Cref{lemma: PARO-CCG}}{}} \label{app: proof-PARO-CCG}
In iteration $0$ of \Cref{alg: CCG-2}, solution $\bm{x}^0$ is the ARO solution resulting from \Cref{alg: CCG}. In subsequent iterations, solution $\bm{x}^k$ is only replaced by a candidate solution $\bm{x}^c$ if $q(\bm{x}^c) \leq \text{OPT}$. Value $q(\bm{x}^c)$ is the optimal objective value of $Q(\bm{x}^c)$, i.e., it is the worst-case objective value of \eqref{eq: P} with fixed Stage-1 decision $\bm{x}^c$. Thus, $\bm{x}^k$ is only replaced by $\bm{x}^c$ if $\bm{x}^c$ is ARO, so in any iteration $\bm{x}^k$ is ARO. It remains to show that if the objective value of problem $P_2(\bm{x}^k,M^k)$ is nonnegative, solution $\bm{x}^k$ is PARO. Proof by contradiction.

Suppose $\bm{x}^k$ is not PARO. Then there exists another $\bm{x}^{\ast}$ that is ARO to \eqref{eq: P} that additionally satisfies the following two conditions:
\begin{enumerate}
\item For each $\bm{z} \in U$ there exists a $\bm{y}$ such that for all $\bm{y}^k$ with $\bm{A}(\bm{z})\bm{x}^k + \bm{B}\bm{y}^k  \leq \bm{r}(\bm{z})$ we have
\begin{align*}
 \bm{c}(\bm{z})^{\top}\bm{x^{\ast}} + \bm{d}^{\top} \bm{y} &\leq \bm{c}(\bm{z})^{\top}\bm{x}^k + \bm{d}^{\top}\bm{y}^k,\\
\bm{A}(\bm{z})\bm{x^{\ast}} + \bm{B} \bm{y} & \leq \bm{r}(\bm{z}).
\end{align*}
\item There exists a $\bm{z^{\ast}} \in U$ and a $\bm{y^{\ast}}$ such that for all $\bm{y}^k$ with $\bm{A}(\bm{z}^{\ast})\bm{x}^k + \bm{B}\bm{y}^k \leq \bm{r}(\bm{z^{\ast}})$ we have
\begin{align*}
 \bm{c}(\bm{z^{\ast}})^{\top}\bm{x}^{\ast} + \bm{d}^{\top} \bm{y^{\ast}} &< \bm{c}(\bm{z}^{\ast})^{\top}\bm{x}^k + \bm{d}^{\top}\bm{y}^k, \\
\bm{A}(\bm{z}^{\ast})\bm{x}^{\ast} + \bm{B} \bm{y}^{\ast} & \leq \bm{r}(\bm{z}^{\ast}).
\end{align*}
\end{enumerate}
Because $M^k \subseteq U$, the first condition implies that for each $\bm{z}^l \in M^k$ there exists a recourse decision $\bm{y}^{l \ast}$ such that $(\bm{x}^{\ast},\bm{y}^{l \ast})$ satisfies constraints \eqref{eq: P2-3} and \eqref{eq: P2-4}. The second condition is equivalent to the statement that there exists $\bm{z}^{\ast} \in U$ and a $\bm{y}^{\ast}$ such that
\begin{align*}
\max_{\bm{y}^k : \bm{A}(\bm{z}^{\ast})\bm{x}^k + \bm{B}\bm{y}^k \leq \bm{r}(\bm{z}^{\ast})} (\bm{c}(\bm{z}^{\ast})^{\top}\bm{x}^{\ast} + \bm{d}^{\top} \bm{y}^{\ast}) &- (\bm{c}(\bm{z}^{\ast})^{\top}\bm{x}^k + \bm{d}^{\top}\bm{y}^k) < 0,\\
\bm{A}(\bm{z}^{\ast})\bm{x}^{\ast} + \bm{B} \bm{y}^{\ast} & \leq \bm{r}(\bm{z}^{\ast}).
\end{align*}
Put together, this implies that $(\bm{z}^{\ast},\bm{x}^{\ast},\bm{y}^{\ast},\bm{y}^{1 \ast},\dotsc,\bm{y}^{ |M^k| \ast})$ is a feasible solution to $P_2(\bm{x}^k,M^k)$ with strictly negative objective value. This contradicts with $p_2^k =0$. Thus, $\bm{x}^k$ is PARO.

\subsection{\texorpdfstring{Proof \Cref{lemma: hybrid-ARF}}{}} \label{app: proof-hybrid-ARF}
We consider only adaptive robust feasibility and not optimality, so the objective of $P_{\text{hybrid}}$ can be ignored. According to \Cref{lemma: y-representation}, each adaptive variable $y_k(\bm{z})$, $k=1,\dotsc,n_y$ must satisfy bounds \eqref{eq: set-inequalities}. For $P_{\text{hybrid}}$ term $\varphi_i(\bm{\hat{z}},\bm{z}_{(i)}) = r_i(\bm{\hat{z}},\bm{z}_{(i)}) - \bm{a}_i(\bm{\hat{z}}, \bm{z}_{(i)})^{\top}\bm{x}$ depends only on $\bm{\hat{z}}$ and $\bm{z}_{(i)}$, for each $i=1,\dotsc,m$. Sets $\hat{U}$ and $U^i$ are disjoint for each $i=1,\dotsc,m$ so this is equivalent to
\begin{align}  \label{appeq: set-inequalities-hybrid}
\begin{aligned}
& \max_{S \in C_{k}^{-}} \Big\{ \sum_{p \in S} \max_{\bm{z}_{(p)} \in U^p} \big(\alpha(S,p) \varphi_p(\bm{\hat{z}},\bm{z}_{(p)})  - \sum_{l=k+1}^{n_y} \beta(S,l) y_l(\bm{z}) \big) \Big \} \leq y_k(\bm{z}) \\
& \hspace*{0cm} \leq \min_{T \in C_{k}^{+}} \Big\{ \sum_{q \in T} \min_{\bm{z}_{(q)} \in U^q} \big( \alpha(T,q) \varphi_q(\bm{\hat{z}},\bm{z}_{(q)})  - \sum_{l=k+1}^{n_y} \beta(T,l) y_l(\bm{z}) \big) \Big \},~~\forall \bm{\hat{z}} \in \hat{U}.
\end{aligned}
\end{align}
We proceed by backward induction. For $k=n_y$, i.e., the last eliminated variable, bounds \eqref{appeq: set-inequalities-hybrid} depend only on $\bm{z}$ and not on other adaptive variables. According to \Cref{lemma: y-representation}, each term $\varphi_i(\bm{z}_{(i)})$, $i=1,\dotsc,m$, appears in upper bounds with a positive coefficient and in lower bounds with a negative coefficient for all variables $y_1(\bm{z}),\dotsc,y_{n_y}(\bm{z})$ (if it appears), or vice versa. Hence, the worst-case scenario for $\bm{z}_{(i)} \in U^i$ (in terms of feasibility) is equal for all linear terms in the lower and the upper bound for all $i=1,\dotsc,m$. Plugging in this worst-case scenario yields lower and upper bounds on $y_{n_y}(\bm{z})$ depending only on $\bm{\hat{z}}$. Thus, there exists a decision rule for $y_{n_y}(\cdot)$ that is a function of only the non-constraintwise uncertain parameters $\bm{\hat{z}}$.

Suppose that for some $k$ the lower and upper bounds \eqref{appeq: set-inequalities-hybrid} for $y_k(\bm{z})$ depend only on $\bm{\hat{z}}$. Thus, there exists a decision rule for $y_{k}(\cdot)$ that is a function of only $\bm{\hat{z}}$. Plug this decision rule in the lower and upper bounds \eqref{appeq: set-inequalities-hybrid} for $y_{k-1}(\bm{z})$. Then, according to \Cref{lemma: y-representation}, each term $\varphi_i(\bm{z}_{(i)})$, $i=1,\dotsc,m$, appears in upper bounds with a positive coefficient and in lower bounds with a negative coefficient (if it appears), or vice versa. Hence, the worst-case scenario for $\bm{z}_{(i)} \in U^i$ (in terms of feasibility) is equal for all linear terms in the lower and the upper bound, for all $i=1,\dotsc,m$. Plugging in this worst-case scenario yields lower and upper bounds on $y_{k-1}(\bm{z})$ depending only on $\bm{\hat{z}}$. This completes the induction. 

Let $\bm{y}(\bm{\hat{z}})$ be the decision rule resulting from the above procedure. Because $\bm{x}$ is ARF to $P_{\text{hybrid}}$, the resulting pair $(\bm{x},\bm{y}(\bm{\hat{z}}))$ is ARF to $P_{\text{hybrid}}$.

\subsection{\texorpdfstring{Proof \Cref{cor: hybrid-ARO}}{}} \label{app: proof-hybrid-ARO}
We note that if \eqref{eq: P} has hybrid uncertainty and the objective \eqref{eq: P-1} contains adaptive variables, it can equivalently be written as 
\begin{subequations} \label{eq: P-epi}
\begin{align}
\min_{t,\bm{x},\bm{y}(\cdot)} ~&~ t,\\
\text{s.t.} ~&~ \bm{c}(\bm{\hat{z}},\bm{z}_{(0)})^{\top} \bm{x} + \bm{d}^{\top} \bm{y}(\bm{z}) \leq t~~\forall (\bm{\hat{z}},\bm{z}_{(0)}) \in \hat{U}\times U^0, \\
~&~ \bm{a}_i(\bm{\hat{z}},\bm{z}_{(i)})^{\top}\bm{x} + \bm{b}_i^{\top}\bm{y}(\bm{z}) \leq r_i(\bm{\hat{z}},\bm{z}_{(i)}),~~\forall (\bm{\hat{z}},\bm{z}_{(i)}) \in \hat{U}\times U^i,~~\forall i=1,\dotsc,m,
\end{align}
\end{subequations}
where $t \in \mathbb{R}$ is an auxiliary here-and-now decision variable. Problem \eqref{eq: P-epi} also has hybrid uncertainty, and a pair $(\bm{x},\bm{y}(\cdot))$ is ARO to \eqref{eq: P} if and only if there exists a $t \in \mathbb{R}$ such that $(\bm{x},\bm{y}(\cdot),t)$ is ARO to \eqref{eq: P-epi}. Thus, in the remainder of the proof we can assume $\bm{d}=\bm{0}$, i.e., the objective is independent of adaptive variables.

According to \Cref{lemma: hybrid-ARF}, for any ARF $\bm{x}$ there exists a decision rule $\bm{y}(\cdot)$ that depends only on $\bm{\hat{z}}$ such that $(\bm{x},\bm{y}(\cdot))$ is ARF to $P_{\text{hybrid}}$. Any $\bm{x}^{\ast}$ that is ARO to $P_{\text{hybrid}}$ is also ARF to $P_{\text{hybrid}}$, so also for each ARO $\bm{x}^{\ast}$ there exists such a decision rule $\bm{y}^{\ast}(\cdot)$. The objective is independent of adaptive variables, so $(\bm{x}^{\ast},\bm{y}(\cdot))$ is ARO for any ARF $\bm{y}(\cdot)$. Hence, $(\bm{x}^{\ast},\bm{y}^{\ast}(\cdot))$ is ARO to $P_{\text{hybrid}}$.

\subsection{\texorpdfstring{Proof \Cref{lemma: block-ARF}}{}} \label{app: proof-block-ARF}
We consider only adaptive robust feasibility and not optimality, so the objective of $P_{\text{block}}$ can be ignored. Remove index $0$ from its constraint set $K(v)$ (for some $v$). The set of constraints can be written as
\begin{align*}
\bm{a}_i(\bm{z}_{(v)})^{\top}\bm{x} + \bm{b}_i^{\top}\bm{y}_{(v)}(\bm{z}) \leq r_i(\bm{z}_{(v)}),~~\forall \bm{z} \in U,~\forall i \in K(v),~\forall v=1,\dotsc,V.
\end{align*}
Due to the block uncertainty structure, all adaptive variables can be eliminated by performing FME on each block $v$ separately. According to \Cref{lemma: y-representation}, bounds on each adaptive variable $y_k(\bm{z})$ can be represented by \eqref{eq: set-inequalities}. If for some $k=1,\dotsc,n_y$, variable $y_k(\bm{z})$ is an element of $\bm{y}_{(v)}(\bm{z})$ for some block $v$, any $S \in C_k^{-}$ or $T \in C_k^{+}$ is a subset of $K_{(v)}$, the original set of constraints for block $v$. The following two observations immediately follow for the given block $v$:
\begin{itemize}
\item For each $l=1,\dotsc,n_y$ the coefficient of $y_l(\bm{z})$ is zero if $y_l(\bm{z})$ is not an element of $\bm{y}_{(v)}$, i.e., $\beta(S,l) = 0$ for all $S \in C_k^{-} \cup C_k^{+}$.
\item For any $p$ in $S$ or $T$ it holds that $\varphi_p(\cdot)$ is a function of $\bm{z}_{(v)}$ only. 
\end{itemize}
For $k=n_y$, i.e., the last eliminated variable, this implies the lower and upper bounds on $y_{n_y}()$ are independent of $\bm{z}_{(w)}$ for $w\neq v$, and any feasible decision rule can be written as a function of $\bm{z}_{(v)}$ only. Plugging any such decision rule in the lower and upper bounds for $k=n_y-1$ yields the same result for $y_{n_y-1}()$. The final result follows from backward induction.

Let $\bm{y}(\bm{z})$ be the decision rule resulting from the above procedure. Because $\bm{x}$ is ARF to $P_{\text{block}}$, the resulting pair $(\bm{x},\bm{y}(\bm{z}))$ is ARF to $P_{\text{block}}$.

\subsection{\texorpdfstring{Proof \Cref{lemma: simplex-LDR-ARF}}{}} \label{app: proof-simplex-LDR-ARF}
We consider only adaptive robust feasibility and not optimality, so the objective of $P_{\text{simplex}}$ can be ignored. According to \Cref{lemma: y-representation}, in the FME procedure the bounds on variable $y_k(\bm{z})$ are given by \eqref{eq: set-inequalities}. It is sufficient to satisfy the bounds on $y_k(\bm{z})$ for all extreme points of uncertainty set $U$, so we can alternatively write: 
\begin{align} \label{appeq: y-bounds-simplex}
\begin{aligned}
&\max_{S_k \in C_{k}^{-}} \Big\{ \sum_{p \in S_k} \alpha(S_k,p) \varphi_p(\bm{x},\bm{z}^j)  - \sum_{l=k+1}^{n_y} \beta(S_k,l) y_l(\bm{z}^j) \Big \} \leq y_k(\bm{z}^j) \\
& \leq \min_{T_k \in C_{k}^{+}} \Big\{ \sum_{q \in T} \alpha(T_k,q) \varphi_q(\bm{x},\bm{z}^j)  - \sum_{l=k+1}^{n_y} \beta(T_k,l) y_l(\bm{z}^j) \Big \},~~\forall \bm{z}^j,~j=1,\dotsc,L+1.
\end{aligned}
\end{align}
For each $j=1,\dotsc,L+1$, let $l_k(\bm{z}^j)$ and $u_k(\bm{z}^j)$ denote the lower resp. upper bound on $y_k(\bm{z}^j)$ from \eqref{appeq: y-bounds-simplex}.  Affine independence of $\bm{z}^1,\dotsc,\bm{z}^{L+1}$ implies linear independence of $(1,\bm{z}^1),\dotsc,$ $(1,\bm{z}^{L+1})$. Hence, by basic linear algebra, there exists exactly one $(a_0,\bm{a}) \in \mathbb{R}\times\mathbb{R}^{L}$ such that $a_0 + \bm{a}^{\top}\bm{z}^j = l(\bm{z}^j)$ for all $j=1,\dotsc,L+1$. Consider the LDR $y_k(\bm{z}) = a_0 + \bm{a}^{\top}\bm{z}$. Then $l(\bm{z}^j) = y_k(\bm{z}^j) \leq u(\bm{z}^j)$ for all $j=1,\dotsc,L+1$. Hence, $y_k(\bm{z})$ is an LDR that satisfies bounds \eqref{appeq: y-bounds-simplex}. Alternatively, one can construct an LDR that passes through points $(\bm{z}^j,u(\bm{z}^j))$ for all $j=1,\dotsc,L+1$, or any LDR that is a convex combination of the previous two LDRs.

Thus, we can construct a decision rule for $y_k(\bm{z})$ that is linear in $\bm{z}$. For all $k=1,\dotsc,n_y-1$, this decision rule depends on $y_{k+1}(\bm{z}),\dotsc,y_{n_y}(\bm{z})$. For variable $y_{n_y}(\cdot)$, the constructed decision rule is independent of other adaptive variables. Plugging this in the decision rule for $y_{n_y-1}(\cdot)$ yields a decision rule that is again independent of other adaptive variables, and still linear in $\bm{z}$ because the coefficient for $y_{n_y}(\bm{z})$ in $l_{n_y-1}(\bm{z})$ and $u_{n_y-1}(\bm{z})$ does not depend on $\bm{z}$ (fixed recourse). Continuing this procedure yields LDRs for all adaptive variables $y_1(\cdot),\dotsc,y_{n_y}(\cdot)$.

Let $\bm{y}(\bm{z})$ be the decision rule resulting from the above procedure. Because $\bm{x}$ is ARF to $P_{\text{simplex}}$, the resulting pair $(\bm{x},\bm{y}(\bm{z}))$ is ARF to $P_{\text{simplex}}$.

\end{document}